\documentclass[11pt]{amsart}
\usepackage{dynkin-diagrams}
 \usepackage{tikz}

\usepackage{listings}[2013/08/05]
\usepackage{footmisc}

\lstdefinelanguage{GAP}{%
  morekeywords={%
    Assert,Info,IsBound,QUIT,%
    TryNextMethod,Unbind,and,break,%
    continue,do,elif,%
    else,end,false,fi,for,%
    function,if,in,local,%
    mod,not,od,or,%
    quit,rec,repeat,return,%
    then,true,until,while%
  },%
  sensitive,%
  morecomment=[l]\#,%
  morestring=[b]",%
  morestring=[b]',%
}[keywords,comments,strings]

\usepackage[T1]{fontenc}
\usepackage[variablett]{lmodern}
\usepackage{xcolor}
\usepackage{lmodern}

\usepackage[all,cmtip]{xy}

\usepackage[new]{old-arrows}
\usepackage{geometry}
\usepackage{graphicx}
\usepackage{lipsum}
\newsavebox{\myimage}

\usepackage{adjustbox}

\usepackage{hyperref}

\newcommand{\sk}{\smallskip}
\newcommand{\mk}{\medskip}
\newcommand{\bk}{\bigskip}

\usepackage{shuffle}

\usepackage{bigints}

\usepackage{tikz}

\newcommand{\xdasharrow}[2][->]{
\tikz[baseline=-\the\dimexpr\fontdimen22\textfont2\relax]{
\node[anchor=south,font=\scriptsize, inner ysep=1.5pt,outer xsep=2.2pt](x){#2};
\draw[shorten <=3.4pt,shorten >=3.4pt,dashed,#1](x.south west)--(x.south east);
}
}

\usepackage{amsmath}
\usepackage{amssymb}
\usepackage{amsthm}
\usepackage{latexsym}
\usepackage{pstricks}
\usepackage{amsbsy}
\usepackage{amscd}
\usepackage{mathrsfs}
\usepackage{tipa}
\usepackage{txfonts}
\usepackage{amsfonts}
\usepackage{graphicx}
\usepackage{tabularx}
\usepackage{listings}
\usepackage{tikz}
\usepackage{hyperref}

\usepackage{scalerel,stackengine}
\stackMath
\newcommand\reallywidehat[1]{%
\savestack{\tmpbox}{\stretchto{%
  \scaleto{%
    \scalerel*[\widthof{\ensuremath{#1}}]{\kern-.6pt\bigwedge\kern-.6pt}%
    {\rule[-\textheight/2]{1ex}{\textheight}}
  }{\textheight}%
}{0.5ex}}%
\stackon[1pt]{#1}{\tmpbox}%
}
\newtheorem{thm}{Theorem}[section]

\newtheorem{lem}[thm]{Lemma}

\newtheorem{prop}[thm]{Proposition}
\newtheorem{fact}[thm]{Fact}

\newtheorem{rem}[thm]{Remark}
\newtheorem{defn-prop}[thm]{Definition-Proposition}

\usepackage{graphicx,epsfig,amscd,graphics,xypic}

 \newcommand{\eq}[1][r]
   {\ar@<-3pt>@{-}[#1]
    \ar@<-1pt>@{}[#1]|<{}="gauche"
    \ar@<+0pt>@{}[#1]|-{}="milieu"
    \ar@<+1pt>@{}[#1]|>{}="droite"
    \ar@/^2pt/@{-}"gauche";"milieu"
    \ar@/_2pt/@{-}"milieu";"droite"}

 \newcommand{\incl}[1][r]
  {\ar@<-0.2pc>@{^(-}[#1] \ar@<+0.2pc>@{-}[#1]}

\author[L. Pirio]{\href{mailto:luc.pirio@uvsq.fr}{Luc Pirio}\textsuperscript{$\dagger$}}
\thanks{${}^{}$\hspace{-0.4cm}\textsuperscript{$\dagger$}
\href{https://lmv.math.cnrs.fr}{Univ.\,Paris-Saclay, UVSQ, CNRS (UMR 8100), Laboratoire de Math\'ematiques de Versailles}, Versailles, France.}

\title[Webs by conics on del Pezzo surfaces and hyperlogarithmic functional identities] {Webs by conics on del Pezzo surfaces and  \\  hyperlogarithmic functional identities}

\begin{document}

\maketitle


\newcommand{\homeo}{\textup{Homeo}^+(S,M)}
\newcommand{\homeoo}{\textup{Homeo}_0(S,M)}
\newcommand{\mg}{\mathcal{MG}(S,M)}
\newcommand{\mmg}{\mathcal{MG}_{\bowtie}(S,M)}
\newcommand{\Z}{\mathbf{Z}}
\newcommand{\Q}{\mathbf{Q}}
\newcommand{\N}{\mathbf{N}}

\hyphenation{ap-pro-xi-ma-tion}

\begin{abstract}
For $d$ ranging from 2 to 6, we prove that the web by conics naturally defined on any smooth del Pezzo surface of degree $d$ carries 
an interesting functional identity whose components all are a certain antisymmetric 
hyperlogarithm of weight $7-d$. Our approach is uniform with respect to $d$ and at the end relies on classical results about the action of Weyl groups on the set of lines contained in the considered del Pezzo surface. 
 This series of `del Pezzo's hyperlogarithmic functional identities' is a natural generalization of the famous and well-know 
 3-term and 
 5-term identities of the logarithm and dilogarithm  ({\it `Abel's relation'})  which correspond to the cases when $d=6$ and $d=5$ respectively.  
This text 
 ends with a section containing several questions and some possibly interesting perspectives.
\end{abstract}

\section{\bf Introduction}
\label{S:Introduction}

In the whole paper, we work over $\mathbf C$, in the analytic and/or algebraic categorie(s).   
In the whole paper,  unless there is a specific warning, $d$ denotes a fixed integer belonging to $\{2,3,\ldots,5\}$  
and $r$ stands for $9-d$. 


\subsection{Abel's 5-term relation of the dilogarithm and del Pezzo's quintic surface}
Several authors of the XIXth and XXth centuries have independently discovered equivalent versions of the nowadays so-called Abel's 5-term relation
$$
\boldsymbol{\big(\mathcal Ab\big)}
\hspace{3cm}
R\big(x\big)-R\big(y\big)-R\bigg(\frac{x}{y}\bigg)-R\left(\frac{1-y}{1-x}\right)
+R\left(\frac{x(1-y)}{y(1-x)}\right)=0\,
\hspace{3cm} {}^{}
$$
which is identically satisfied for  any $(x,y)\in \mathbf R^2$ such that $0<x<y<1$, by the famous \href{https://mathworld.wolfram.com/RogersL-Function.html}{\it Rogers' dilogarithm} $R$ defined by 
\begin{equation}
\label{Eq:R}
R(x)= {\bf L}{\rm i}_2(x) + \frac{1}{2}\,{\rm Log}(x)\,{\rm Log}(1 - x) - \frac
{{}^{}\hspace{0.1cm} \pi^2}{6}\, 
  \end{equation}
for $x\in ]0,1[$, where ${\bf L}{\rm i}_2$ stands for the classical bilogarithm, which is the weight 2 element of the famous class of special functions, the so called `polylogarithms' ${\bf L}{\rm i}_n$, (for any $n\geq 1)$, defined 
on the unit disk ${\bf D}=\{\, z\in \mathbf C\,, \, \lvert z\lvert<1\, \}$ 
as the sums of the convergent series
   ${\bf L}{\rm i}_n(z)=\sum_{k\geq 1} z^k/k^n$. 
 
The classical logarithm ${\rm Log}$  can be seen as a version of the first polylogarithm since 
${\bf L}{\rm i}_n(z)=-{\rm Log}(1-z)$ for any $z$ in ${\bf D}$. As is well known, it satisfies `Cauchy identity' 
\begin{equation}
\label{Eq:EFA-3-terms-Log}
{\rm Log}(x)+{\rm Log}(y)-{\rm Log}(xy) = 0
\end{equation}
 for $x,y \in {\bf D}$, which is known to be fundamental in mathematics (and in particular characterizes the logarithm up to a multiplicative constant).  Abel's relation is nowadays acknowledged as another fundamental functional identity as well, one reason for this being that it appears in hence somehow connects several distinct branches of mathematics, see \cite{Zagier}.
Some versions of the higher weight polylogarithms ${\bf L}{\rm i}_n$ satisfy similar but much more involved functional equations, but this is known only for $n\leq 7$ and no general pattern of what might be the functional identities satisfied by the $n$-th polylogarithm for $n$ arbitrary has emerged yet (see \S\ref{SS:Context} below for more perspectives and references). 
\begin{center}
\vspace{-0.4cm}
$\star$
\end{center}

In this text, we describe some generalizations in weight 3,4 and 5,  of 
the 3-term and 5-term identities of the logarithm and dilogarithm respectively, 
 which formally appear as very similar to the two latter classical identities, but concern hyperlogarithms, a class of functions admitting the polylogarithms and their variants as a specific case.  
To this end, let us take a geometric viewpoint on Abel's identity 
$\boldsymbol{\big(\mathcal Ab\big)}$, or more precisely on the 
five arguments $x,y,x/y$, $(1-x)/(1-y)$ and $x(1-y)/(y(1-x))$ of $R$ in it.\footnote{The case of the 3-term logarithm identity can be considered in the same way, but it is much less telling and so we will only consider the case of Abel's equation of dilogarithm in this Introduction.} Instead of viewing them as rational functions on $\mathbf P^2$, it is actually more natural to consider their lifts to the blow up, denoted by $X_4$,  of the projective plane at the four points $[1:0:0]$, $[0:1:0]$, $[0:0:1]$ and $[1:1:1]$. Its anticanonical class $-K_{X_4}$ is very ample and induces an embedding of it  into $\mathbf P^5$ whose image 
$\varphi_{\lvert -K_{X_4} \lvert}(X_4)$ 
is  the classical quintic del Pezzo surface, denoted here by ${\rm dP}_5$.

It is well known that $X_4$ carries exactly five fibrations by conics 
$$U_1,\ldots,U_5: X_4\simeq {\rm dP}_5\longrightarrow \mathbf P^1$$ and that these precisely correspond via the blow-up $X_4\rightarrow \mathbf P^2$ to the five rational functions $\mathbf P^2 \dashrightarrow \mathbf P^1$  associated to the five arguments of $R$ in  $\boldsymbol{(\mathcal Ab\big)}$. 

 All the classical facts recalled above allow to write down Abel's identity in the following form $$
\boldsymbol{\big(\mathcal Ab_{X_4}\big)}
\hspace{5cm}
\sum_{i=1}^5 \epsilon_i \,R\big(U_i\big)=0\, ,
\hspace{6cm} {}^{}
$$
for some constants $\epsilon_1,\ldots,\epsilon_5$ equal to $1$ or $-1$, this identity holding true locally at any sufficiently generic point of $X_4$, for some suitable branches of 
Rogers' dilogarithm.\footnote{The bilogarithm ${\bf L}{\rm i}_2$ hence $R$ (considering \eqref{Eq:R}) both extend as multivalued holomophic functions on $\mathbf P^1\setminus \{0,1,\infty\}$.} 

It is a general fact that any del Pezzo surface admits a finite number of conic fibrations. 
Considering this, it is very natural to wonder whether,  for any del Pezzo surface ${\rm dP}_d$ of degree $d\leq 5$, the identity $\boldsymbol{\big(\mathcal Ab_{X_4}\big)}$ admits a generalization with respect to the 
conic fibrations ${\rm dP}_d \rightarrow \mathbf P^1$. The answer is affirmative for 
$d$ ranging from 2 to 5  and this is what we are going to explain 
 now.

\subsection{Generalization to del Pezzo's surfaces of lower degree.} 
In order to state our result, we recall basic facts about del Pezzo surfaces. 
More details and explanations will be given in \S\ref{S:Main}.

Let $d\in \{2,\ldots,6\}$ and  set $r=9-d \in \{3,\ldots,7\}$.  
The blow-up $X_r={\bf Bl}_{p_1,\ldots,p_r}(\mathbf P^2)$ of the projective plane in $r$ points is a del Pezzo surface of degree $d$.  The anticanonical class $-K_{X_r}$ is ample and the degree will be taken with respect to it.  By definition, a `conic (fibration) stucture' on $X_r$ is the equivalence class, up to post-composition by projective automorphisms, of morphisms $f: X_r\rightarrow \mathbf P^1$ such that $f^{-1}(z)$ is a smooth conic (that is a smooth rational curve of degree 2) for all except for 
 a finite number of $z\in \mathbf P^1$.
 \newpage

The following facts are well known: 
\vspace{-0.18cm}
\begin{itemize}
\item[$\boldsymbol{(a).}$]  the number $l_r$ of `lines' contained in $X_r$ is finite; \sk

\item[$\boldsymbol{(b).}$]  the number $\kappa_r$ of conic structures on $X_r$ is finite as well;\sk 
\item[$\boldsymbol{(c).}$]   
any conic fibration $X_r\rightarrow \mathbf P^1$ 
has exactly $r-1$ non irreducible fibers, each such reducible fiber being the union of two `lines' included in $X_r$ intersecting transversely in one point.
\end{itemize}
\vspace{-0.1cm}
The values of $l_r$ and $\kappa_r$ in function of $r$ are given in the following table: 
$$ \begin{tabular}{|c||c|c|c|c|c|}
\hline
$r$ & 3 & 4& 5  & 6 & 7 \\ \hline \hline
$l_r$ & 6 & 10 &  16  & 27 & 56  \\ \hline
$\kappa_r$ & 3 & 5 &    10 & 27 & 126\\
\hline
\end{tabular}
$$

Let $U_1,\ldots,U_{\kappa_r}: X_r\rightarrow \mathbf P^1$ be $\kappa_r$ pairwise non equivalent conic fibrations. 
They are rational first integrals for the `web by conics on $X_r$', denoted by $ \boldsymbol{\mathcal W}_{X_r}$: one has $$ \boldsymbol{\mathcal W}_{X_r}= \boldsymbol{\mathcal W}\big( U_1,\ldots,U_{\kappa_r}\big) \, .$$ 
Since two smooth fibers of two distinct conical first integrals $U_i$ and $U_j$ intersect transversely, 
the singular set of the $\kappa_r$-web 
$ \boldsymbol{\mathcal W}_{X_r}$ is supported on a union of lines. Actually, 
it can be verified that this singular set  
precisely is the locus $L_r$  of lines 
included in $X_r$ hence that $ \boldsymbol{\mathcal W}_{X_r}$ is a regular web on  $Y_r=X_r \setminus L_r$. \sk

From $\boldsymbol{(c).}$ above, we know that 
the complement $\mathfrak R_i$ of $U_i(Y_r)$ in $\mathbf P^1$ is a finite set with $r-1$ elements denoted by $\rho_i^{1},\ldots,\rho_i^{r-1}$.  One assumes that $U_i$ 
has been chosen such that one of the $\rho_i^{t}$, say $\rho_i^{r-1}$, coincides with $\infty\in \mathbf P^1$.  Then the  rational differentials $\omega_i^{t}= dz/(z-\rho_i^{t})$ for $t=1,\ldots,r-2$ form a basis of the 
space of logarithmic 1-forms on $\mathbf P^1$ with poles on $\Sigma_i$.  Then 
let $AI_i^{r-2}$ be 
the hyperlogarithm of weight $r-2$ on $Z_i=\mathbf P^1\setminus \mathfrak R_i$,  whose symbol 
$\boldsymbol{\mathcal S}(AI_i^{r-2})$
is the $\mathfrak S_{r-2}$-antisymmetrization of $ 
\Omega_i^{r-2}=
\omega_i ^1\otimes \cdots \otimes \omega_i ^{r-2}$, namely 
$$
\boldsymbol{\mathcal S}\big(AI_i^{r-2}\big)=
{\rm Asym}\big(\Omega_i^{r-2} \big) = 
\sum_{\sigma \in \mathfrak S_{r-2}}\boldsymbol{\epsilon}(\sigma) \cdot
\omega_i^{\sigma(1)}\otimes \cdots \otimes \omega_i^{\sigma(r-2)}
$$
where $\boldsymbol{\epsilon}: \mathfrak S_{r-2}\rightarrow \{\, \pm 1\, \}$ stands for the signature morphism.

Our main result is the following: 

\begin{thm}
\label{Thm:main}
There exists a
$\kappa_r$-tuple $(\epsilon_i)_{i=1}^{\kappa_r}\in \{ \pm 1 \}^{\kappa_r}$, 
 unique up to sign, 
  such that 
for any $y\in Y_r$ and for a suitable choice of the branch of the hyperlogarithm $AI_i^{r-2}$ at $y_i=U_i(y)\in \mathbf P^1$ for each $i=1,\ldots,\kappa_r$, the following functional identity 
holds true  on an open neighbourhood of $y$ in $ Y_r$: 
$$ {}^{} \hspace{-6cm}
\boldsymbol{\big({\bf H Log}(X_r)\big) }
\hspace{3.7cm}
\sum_{i=1}^{\kappa_r} \epsilon_i \,AI_i^{r-2}\big(U_i\big)=0\, .
$$
\end{thm}
A few comments are in order regarding this result. 
\begin{itemize}
\item 
Of course, $\boldsymbol{\big({\bf H Log}(X_3)\big) }$ 
 is nothing else but the logarithm identity  \eqref{Eq:EFA-3-terms-Log}
and $\boldsymbol{\big({\bf H Log}(X_4)\big) }$ coincides  with  the `geometric identity' 
$\boldsymbol{\big(\mathcal Ab_{X_4}\big)}$ hence 
is equivalent to Abel's relation $\boldsymbol{\big(\mathcal Ab\big)}$.  In constrast, the three identities $\big({\bf H Log}(X_r)\big)$ 
  for $r=5,6,7$ are entirely new although that from a  formal point of view, they  appear as rather direct generalizations to the case of lower degree del Pezzo surfaces of the very classical functional identities \eqref{Eq:EFA-3-terms-Log} and  $\boldsymbol{\big(\mathcal Ab\big)}$.
  \sk
  \item 
  Given $r$,  the suitable branch $AI_i^{r-2}$ of the hyperlogarithm  which is referred to in the statement above is defined 
 in a natural constructive way for any $i=1,\ldots,\kappa_r$. Furthermore, we  actually prove a more invariant version of this theorem (see  Proposition \ref{Prop:Invariant-Formulation} below) 
by constructing each term $\epsilon_i \,AI_i^{r-2}(U_i)$ by means of the natural action of the suitable Weyl group on only one of these, say $\epsilon_1 AI_1^{r-2}(U_1)$ which can be assumed to be $AI_1^{r-2}(U_1)$. 
\sk  \item 
Actually, there is a nice conceptual interpretation of why $\boldsymbol{\big({\bf H Log}(X_r)\big) }$ holds true in terms of the space $\mathbf C^{{\mathcal L}_r}$ freely spanned by the set of lines contained in $X_r$. This space is acted upon in a natural way by a certain Weyl group $W(E_r)$ and from a representation-theoretic perspective, the LHS  of 
$\boldsymbol{\big({\bf H Log}(X_r)\big) }$ can be interpreted as the image 
of the  signature representation ${\bf sign}_r$ of the Weyl group $W(E_r)$ into the $(r-2)$-th wedge product of $\mathbf C^{{\mathcal L}_r}$.  And the reason why $\sum_{i=1}^{\kappa_r} \epsilon_i \,AI_i^{r-2}\big(U_i\big)$ vanishes identically follows from the fact that  ${\bf sign}_r$ does not appear with positive multiplicity in the decomposition 
of $\wedge^{r-2}\mathbf C^{{\mathcal L}_r}$ into irreducible $W(E_r)$-modules (see \S\ref{SS-Key-Clef} for more details). 
\end{itemize}

\subsection{\bf Context and motivations}
\label{SS:Context}
In this paper, we establish a link between fields whose relationship was noticed before, namely the theory of del Pezzo surfaces on the one hand, and that of functional equations on the other hand. This connection is made through the webs formed by the pencils of conics carried by del Pezzo surfaces. Below we quickly discuss these different fields from a historical perspective.\footnote{The subsection \S\ref{SSS:Surfaces generated by circles and conics} is a short version of a quite longer historical text we are currently writing on the subject. An interesting historical and mathematical overview of the  theory of del Pezzo surfaces can also be found in Dolgachev's book \cite{Dolgachev}, see the Historical Notes of the 8th chapter therein.}

\subsubsection{Surfaces generated by circles and conics} 
\label{SSS:Surfaces generated by circles and conics}
Several nice examples of doubly-ruled surfaces were known since Greek Antiquity (Archimedes) but it seems that it took  more than two thousands years 
for mathematicians to give examples of surfaces generated in two different ways by a 1-dimensional family of circles. 

In his thesis under the supervision of Monge (defended in 1803), \href{https://www.persee.fr/doc/inrp_0298-5632_1994_ant_19_1_8440}{Dupin} discovered 
 a special class of non spherical surfaces enjoying the remarkable property that all their curvature lines are circles.  These surfaces were named 
 `{\it cyclides}', and their 
 discovery by Dupin led to a huge amount of researches during the XIX-th century. 
 The study of cyclides is still active nowadays. 
 
 Several authors (Liouville, Lord Kelvin, etc) discovered that cyclides can be obtained by inversions with respect to spheres from simpler surfaces, such as tori. Since inversions transform circles into circles  (possibly of infinite radius), it followed from the discovery by Villarceau (in 1848) of the circles named after him on any  revolution torus
  that  through a generic point of a cyclide pass at least four circles contained in it. 
  The fact that cyclides carry several families of circles led to 
  quite some works during the XIX-th century, however with a gap from Villarceau's discovery until the year 1863. We believe that the reason behind that is the important 1849 breakthrough by Cayley and Salmon that a generic cubic surface contains 27 lines, to each of which is naturally associated a pencil of conics.\footnote{Given a line in a cubic surface, the associated pencil is the one whose elements are the residual intersections with the cubic of the 2-planes containing the considered line (see the last remark of \S\ref{SSS:C-Kr} for more details).} Most geometers of that time have mainly focused on the case of cubics for some time, leaving aside the case of other surfaces 
  ({\it e.g.}\,cyclides\footnote{Up to some degenerated cases, the cyclid are rational quartic surfaces with a conic of double points 
``at infinity''.})  
  for about fifteen years. \sk
  
  Cayley and Salmon's discovery led to an intense activity about the geometry of cubic surfaces, by many authors: Steiner, Schl\"affli, Clebsch were among the first ones and they were followed by a lot of others. That a cubic surface carries 27 pencils of conics in general\footnote{It was known already at that time that singular real cubic surfaces may contain strictly less than 27 lines.} led several geometers to study surfaces carrying several families of conics or even of circles (for Euclidean sufaces). The first important work in this direction beyond the case of cubic surfaces is the text
   \cite{Kummer1863} in which Kummer studied cyclides from the point of view of classical algebraic geometry. In particular, he proved that the most general cyclid in the ordinary projective space carries 10 families of conics. 
   
   Interesting results were obtained about Euclidean surfaces by the french geometers Darboux and Moutard, pretty much at the same time and seemingly independently. 
   In some 1863 letters to Poncelet, most certainly  inspired by the case of cubic surfaces, Moutard indicated how to  construct  a 27-web on any sufficiently generic surface in the Euclidean 3-space. He then turned the following year to the case of cyclides and obtained  (seemingly unaware of Kummer's work) that such a surface carries 10 families of conics, a result also obtained (independently) by Darboux at the same time.\footnote{However, Darboux and Moutard were dealing with real surfaces in the Euclidean 3-space $\mathbb E^3$ and it turns out that  the maximal number of circles included in a  given 
   cyclide  of  $\mathbb E^3$  and passing through one of its generic points  is 6.} 
Following Kummer's and Darboux-Moutard's works, many authors investigated surfaces carrying several families of conics or of circles, using an algebraic geometry approach or within the perspective of the Gaussian theory of surfaces in the Euclidean 3-space.  The most relevant ones considering the purpose of the present work are those which are part of the first approach  and some of the most famous contributors along it are the following geometers (in addition to Cayley, Salmon and Kummer already cited above and according to  a rough chronological order):   Steiner, Schl\"afli, Clebsch, Cremona, Sturm, and Del Pezzo.

Among the numerous works on the subject  produced at this time, 
some papers by Clebsch  are worth mentioning,   such as the famous 1866 paper 
\cite{Clebsch1866} in which he established that a general cubic surface is `{\it representable in the plane}' via a linear system of cubic curves passing through six fixed base points in general position. This allowed him to give very concrete `plane interpretations' of some particular curves on the corresponding cubic, such as lines or pencils of conics. 
He dealt with the case of cyclides in the same way
 the case of cyclides two years later  in \cite{Clebsch1868} and its approach 
 was taken up by many others after him, even by some differential geometers.\footnote{{\it E.g.}\,see the papers \cite{DarbouxdP5, DarbouxdP4} by Darboux, in which he uses Clebsch's algebraic approach to investigate the cases corresponding to (generic projections in $\mathbf P^3$ of the) del Pezzo surfaces of degree 4 and 5 in $\mathbf P^4$ and $\mathbf P^5$ respectively.} 

Another important step which happened about twenty years later is due to Segre and Veronese who  independently, in \cite{Segre1884} and \cite{Veronese1884}, 
realized then  proved that generically the cyclides considered in the ordinary projective space actually  are projections from a generic point of the surfaces in $\mathbf P^4$ obtained as the intersections  of two hyperquadrics in this four-dimensional space. 

Many results previously obtained by others are gathered, generalized and organized in a more conceptual way in del Pezzo's famous 1887 paper \cite{DelPezzo1887}.  In it, Del Pezzo studied and classified the algebraic surfaces of degree $d$ in $\mathbf P^d$. He proved that necessarily $n\leq 9$ and gave a uniform way to `represent such a surface on the plane' by means of a linear system of cubic curves with $9-d$ fixed base points. Such surfaces are named after del Pezzo thanks to his foundational paper and have since been studied by many classical and recent authors, from several perspectives (geometric or arithmetic for instance). 

The existing literature on del Pezzo surfaces is so huge that it is impossible to discuss the subject in its whole here.  We will only discuss below a few recent 
 papers about surfaces carrying covering families of circles or conics and where the emphasis is on this property. The study of circled surfaces (which are surfaces containing at least one Euclidean circle through any of their general points) has undergone a certain revival since the beginning of the 1980s, perhaps triggered by a nice article by Blum \cite{Blum} in which he revisits earlier and even classical statements about cyclides with a modern and rigorous approach.  Several papers by differential geometers with links to the theory of cyclides  and of circled surfaces have appeared since the early 1980s and starting from the beginning of 1990s, several people working in CAGD\footnote{CAGD is the acronym for `Computer Aided Graphic Design'.} became interested in these topics and started to  publish many papers on them.  The study of cyclides and more generally of circled surfaces is still active nowadays, in applied geometry as well as in pure differential geometry.  
 
Among the recent papers published in this area which are interesting considering the purpose of the current text, we would like to mention \cite{Schicho} where Schicho classified the complex projective surfaces in $\mathbf P^3$ carrying several covering families of conics. With his PhD student Lubbes, they investigated 10 years later in  \cite{LubbesSchicho}
 the problem of  classifying  pairs $(X,\mathcal F)$ where $X$ is a rational surface in a 
 projective space $\mathbf P^r$ and $\mathcal F$ is an irreducible minimal 
 covering family of generically  irreducible rational curves.\footnote{The adjective `minimal' refers here to the degree of the general element of $\mathcal F$  which is supposed to be as small as possible.} Lubbes has published since several  papers, 
 either about circled Euclidean surfaces ({\it e.g.}\,\cite{Lubbes2021a}) or about 
 the description and classification of minimal families of rational curves on algebraic surfaces embedded in a projective space.  For instance,  he described   
 all the minimal families of rational curves on complex weak del Pezzo surfaces
in \cite{Lubbes2014} and treated the real version of this problem a few years later in \cite{Lubbes2019}.
In particular, Lubbes described and classified  the 
 minimal families of conics on 
weak del Pezzo surfaces which are not induced by a fibration.  
That such minimal families exist\footnote{According to Lubbes, 
 that such `non-fibration minimal families' exist should be attributed to his PhD advisor J. Schicho.}  is a bit surprising and 
 is  interesting, especially from the perspective of web geometry.

\subsubsection{Webs of conics on projective surfaces} 
The objects considered below are the same as in the preceding subsection, but the perspective is different: 
 in the preceding subsection, the focus was on projective surfaces carrying several pencils of conics. Here, our interest concerns the webs which are formed by the pencils of conics on such a surface. \sk

Web-geometry has been linked to the theory of embedded surfaces since its very beginning. For instance and as is well-known, the birth of web-geometry is attributed to the discovery by Thomsen in \cite{Thomsen} of the fact that the isotherm-asymptotic surfaces in $\mathbf P^3$ can be characterized as those whose Darboux's 3-web is hexagonal.\footnote{See \S1.1.1 and \S4.2.1.2. in \cite{PThese} for more details.} However, it seems that the first web-geometers around Blaschke\footnote{Web geometry was developed by a group of geometers leaded by Blaschke, at Hamburg  in the years 1927-1936.} did not consider much webs by conics on projective surfaces, even if it was already known at the end of the 1930's that maximal rank planar webs naturally live on projective surfaces.\footnote{See \cite[Chap.\,8]{PThese} for more details.}  As far as we know, the only paper  of this period where such a web is considered is \cite{Burau} in which Burau  studied the web by conics on a cubic surface. He determined  which of the 3-subwebs of this 27-web are hexagonal and proved that any 27-web with the same set of hexagonal 3-subwebs actually is equivalent to a web by conics on a cubic surface.  
Despite its importance in the constitution of one of the main problems in the geometry of webs (namely, that of classifying exceptional webs), the fact that Bol's web naturally lives on the quintic del Pezzo surface in $\mathbf P^5$ does not seem to have appeared in the literature until the  paper \cite{Burau2}, by Burau again,  but dating of 1966 that is 30 years after his first paper on the web by conics on cubic surfaces was published.

From the end of the 1930's until the beginning of the XXI-th century, webs were studied by different groups\footnote{See the  appendix of \cite{PP} for a short historical overview of web geometry.} but only very few connections with the world of algebraic projective varieties were made. In addition to Burau's second paper just cited, one has first to mention the attempt by Chern and Griffiths to establish the algebraization of webs of codimension 1 with maximal rank. The other work which should be cited is Damiano's paper \cite{Damiano} where for any $n\geq 2$, he applied Gelfand-MacPherson theory of {\it `generalized dilogarithm forms'} to the study of $(n-1)$-th abelian relations of the curvilinear web on $\mathcal M_{0,n+3}$ defined by the $n+3$ forgetful maps $\mathcal M_{0,n+3}\rightarrow \mathcal M_{0,n+2}$.\footnote{However there was an error in \cite{Damiano}. 
It has been explained and corrected in  our recent paper \cite{PirioM0n+3}.} Although both contain serious errors, the works of Chern-Griffiths on the one hand and of Damiano on the other hand, are certainly the main reasons of the renewal of the study of webs with maximal rank from the mid 1990's onwards. However, if many important results appeared on this problem since, the specific case of webs by conics on projective surfaces was not really considered by the people working in this direction (H\'enaut, Mar\`in, Pereira, Pirio,  Tr\'epreau to name a few).  

The new interest for cyclides mentioned above led several authors 
to look more in depth at the webs by circles they carry.  They were motivated by the geometric study of cyclides in order to find applications in CAGD, but also by some links with a famous problem posed long ago by Blaschke and Bol, about the determination of the 3-webs by circles which are hexagonal  ({\it cf.}\,\cite[p.\,31]{BB}).  
 This problem is still open and has  given rise to several papers, many in relation with cyclides. A nice example of such a paper is  \cite{PSS}: in it, among other things (such as many nice pictures of webs by circles  on cyclides),  the authors  provide a complete classification of all possible hexagonal webs
of circles on Darboux cyclides. This results motivated several researchers to consider  other cases from the same perspective. 
 The last and most interesting work in this direction is the recently published paper \cite{Lubbes2021b} in which Lubbes studies webs by minimal families of rational curves on real surfaces $S$ in projective spaces. Several interesting results are obtained such as a combinatorial criterion ensuring the hexagonality of any 3-web formed by such families of curves  on $S$ from which one gets a classification (up to linear projections) of the real surfaces $S\subset \mathbf P^n$ covered by a hexagonal 3-web of conics.

A notable point about all the webs by rational curves of minimal degree appearing in the literature mentioned above is that the unique web-geometric property considered about them is that of hexagonality of some of their 3-subwebs, which is quite specific.  At the exception of Bol's web on the del Pezzo quintic surface which was known to carry the 5-term identity of the dilogarithm, we are unaware of any previous result about primitive ARs of length strictly bigger than 3 for such webs.\footnote{The {\it length} of an AR $\boldsymbol{\Phi}$ of a web $\boldsymbol{\mathcal W}$ is the smallest integer $k$ 
such that $\boldsymbol{\Phi}$ can be considered as an AR of a $k$-subsweb of $\boldsymbol{\mathcal W}$. And $\boldsymbol{\Phi}$ is {\it primitive} if 
it cannot be written as a linear combination of ARs of strictly smaller lengths.}


\subsubsection{Functional identities of polylogarithms  and hyperlogarithms}
A  detailed historical survey with many references on  polylogarithms and the functional identities they satisfy can be found in \cite[\S2]{ClusterWebs}. We will be much more sketchy below.

The discovery that the logarithm is a function, actually a primitive of the inverse function,    and that it satisfies the functional equation  
\eqref{Eq:EFA-3-terms-Log} 
is quite ancient and goes back to the XVII-th century. It was recognized very early that this functional equation is fundamental and that  the very nature of the logarithm somehow is contained in it.\footnote{For instance, in \cite[III]{Pfaff} Pfaff wrote about  \eqref{Eq:EFA-3-terms-Log} that {\it `indoles logarithmorum hac aequatione fundamentali continetur'} (`the nature of logarithms is contained in this basic equation').}

The dilogarithm was considered by Euler and polylogarithms  of weight higher than or equal to 2 have been investigated starting from the very beginning of the XIXth century.  Many authors of that time   have independently discovered  several functional equations  satisfied by low-order polylogarithms.  Among the most famous ones from that period, one can mention (in chronological order): Spence (1809), Hill (1829), Kummer (1840), Abel (1881), Newman (1892) and Rogers (1907).  In our opinion, the most significant results obtained by these authors are the following: 
\vspace{0.1cm} 
\begin{itemize}
\item[$-$] the 5-terms identity of the dilogarithm, in particular Rogers' version which is homogeneous (does not involve logarithmic terms);\vspace{0.1cm} 
\item[$-$] the Spence-Kummer identity of the trilogarithm (discovered independently by  both); \vspace{0.1cm}
\item[$-$] Kummer's identities for the tetra- and pentalogarithm. 
\end{itemize}

All the functional equations satisfied by the weight $n$ polylogarithm 
discovered during the XIX-th century  (and possibly all known so far) 
  have the general form
\begin{equation}
\label{Eq:Equation-Lin}
\sum_{k=1}^N n_k \, {\bf L}{\rm i}_n\big( U_k \big) 
=\boldsymbol{\mathcal L}_{<n}
\end{equation}
where the $n_k$'s are integers, the $U_k$'s rational functions and 
$\boldsymbol{\mathcal L}_{<n}$ stands for a polynomial expression in terms of the form 
$ {\bf L}{\rm i}_{m_s}\big( V_s \big)$  for some rational functions $V_s$ 
and some positive integers 
 $m_s<n$.\sk

The first remarkable apparition of a dilogarithmic function outside the field of functional equations goes back to the work of Lobachevsky  who in 1836 
gave  a formula for the volume of a 3-dimensional hyperbolic orthoscheme 
as a linear combination of a  real-valued  dilogarithmic function  evaluated on some angles related to the dihedral angles of the considered hyperbolic polytope.  
Within  3-dimensional hyperbolic geometry, one can give a nice geometric explanation  (of a real global version, {\it cf.}\,\S\ref{S:single-valued}) of the 5-term relation of the dilogarithm (see \cite[\S4]{Zagier} for more details).
The links between polylogarithms and hyperbolic geometry go far beyond this case  and are still under investigation nowadays (in addition to Zagier's paper juste cited, see Goncharov (1999) or Rudenko (2020) for instance).

Rogers' version of the 5-term relation of the dilogarithm received a great attention from the members of Blaschke's school of web geometry when Bol discovered in \cite{Bol} that the web (now named after him) carrying this identity actually has maximal rank\footnote{For the record, this contradicted a previous statement by Blaschke, asserting that this web has rank 5.} but  is not linearizable, hence not algebraizable.  This result marked the birth of a special subdomain of web geometry, namely the one of classifying `{\it exceptional webs}' that is webs of maximal which however are not algebraic in the classical sense. Bol explicitly asked at the end of \cite{Bol} if it is possible to generalize the notion of abelian function in order to encompass the classical examples and the dilogarithm $R$ in a same general class of functions. Moreover, 
according to Bol (but he gave no source), such questions were already considered by Abel and motivated him to consider his version of the 5-term equation of the dilogarithm.  Some recent works by Goncharov (1995) and  
Kerr-Lewis-Lopatto (2018) show that  to some extent, Bol's question above can be answered by the affirmative.

Zagier has proved that the dilogarithm together with the 5-term identity it satisfies play a crucial role for describing some part of the third group of $K$-theory 
of a given number field $F$ (see \cite[\S4.A]{Zagier} for more details): a global single-valued version of the dilogarithm $D: \mathbf P^1\rightarrow \mathbf C$ (the `{\it Bloch-Wigner dilogarithm}', see \ref{S:single-valued}) can be interpreted as a regulator map $K_3(F) \rightarrow \mathbf R$ and the fact that $D$ satisfies the global version of 
Abel's identity allows to relate a certain part of $K_3(F)\otimes \mathbf Q$ to the group $B_2(F)$ defined as the quotient of the  $\mathbf Q$-vector space freely spanned by the elements $[z]$, for $z\in F$ distinct from $0$ and $1$ by the subspace spanned  by the elements 
$[x]-[y]-[{x}/{y}]-[(1-y)/(1-x)]+ [x(1-y)/(y(1-x))]$ for all $x,y\in F$ such that all the quantities between brackets are well defined.  Zagier conjectured that this extends to the higher $K$-theory groups of $F$  and that  for any $n\geq 2$, 
using the functional identities  of the form \eqref{Eq:Equation-Lin} satisfied by ${\bf L}{\rm i}_n$, one can cook up a group $B_n(F)$ which describes a certain part (namely the `indecomposable part') of $K_{2n-1}(F)\otimes \mathbf Q$.

Zagier's conjecture induced a strong interest for the quest of new functional equations for polylogarithms which led to several interesting outcomes.  First, in several papers, Gangl found several new such identities, all in weight less than or equal to 7. Zagier's conjecture has been proved to be true for $n=3$ by Goncharov (1995) and more recently for $n=4$ by Goncharov-Rudenko (2018). In both cases, an important step is in describing and then using  a  new functional equation satisfied by (some modified versions of) the  4-th and the 5-th order polylogarithm respectively.  
 The corresponding functional equation in the case n=3 is in 3 variables and with 22 terms and is constructed by means of projective geometry (of some configurations of points in the plane). The case $n=4$ is more sophisticated and the authors use the cluster algebras technology to build the suitable functional equation. Motivated by Zagier's conjecture and the theory of motives, Goncharov has studied polylogarithms (in several variables) in several papers and has stated several deep conjectures about them. Several researchers have worked and are still working on this topic, and several papers about (mainly multivariables) polylogarithms have appeared during the last years, some of them focusing on the functional equations they satisfied.

 In contrast to the classical polylogarithms, very little is known about the functional identities satisfied by other hyperlogarithms. If the so-called {\it `Nielsen polylogarithms'} have been studied by several classical and modern authors, the study of the functional equations they satisfied is very recent, see 
 \cite{CGR} (published in 2021).
%
\mk

Recently, polylogarithmic identities (mainly dilogarithmic indentities) have appeared in other fields such as in the theory of cluster algebras or in mirror symmetry of log Calabi-Yau complex manifolds.  Because of the so many fields where they appear and also in regard to the importance they might have (see \cite[\S2]{GoncharovGelfandSeminar} for instance),  the interest for studying the functional identities satisfied by polylogarithms is now well acknowledged. However, a general scheme about what might the basic structure of these functional identities is still lacking, except in weight 1 and 2.\footnote{Regarding the dilogarithmic identities, one of the most interesting results is quite recent: that Abel's identity is fundamental regarding all the identities of the form \eqref{Eq:Equation-Lin} when $n=2$ has been rigorously established only in 2020 (in \cite{RdJ}).}  Some authors expect that for each weight $n$, there exists a fundamental identity of the form \eqref{Eq:Equation-Lin} from which all the others might be deduced (by formal linear combination and specializations). 
For instance, in \cite{Griffiths}, Griffiths wrote the following regarding this: 
\begin{quote}
{\it We do not attempt to formulate this question precisely $-$ intuitively, we are asking whether or not for each $k$ there is an integer $n(k)$ such that there is a ``new'' $n(k)$-web of maximum rank one of whose abelian relations is a (the?) functional equation with $n(k)$ terms for the $k$-th polylogarithm ${\bf L}{\rm i}_k$? Here, ``new'' means the general extension of the phenomena above for the logarithm when $k = 1$, where $n(1) = 3$, for the Bol web when $k = 2$ and $n(2) = 5$, and for the H\'enaut web\footnote{Griffiths refers here to the planar 9-web associated to the Spence-Kummer's identity of the trilogarithm, known in the litterature as 
{\it Spence-Kummer's web}  (see \cite[\S2.2.3.1]{ClusterWebs}).}
 when $k = 3$ and $n(3) = 9$.}
\end{quote}

It is reasonable to take as `the' fundamental equation of the logarithm and the dilogarithm, the identities \eqref{Eq:EFA-3-terms-Log} and $\boldsymbol{\big(\mathcal Ab\big)}$ respectively. Goncharov's 22 terms identity in 3-variables might indeed be `the' fundamental identity for the trilogarithm, and the one for the tetralogarithm might be that recently discovered by 
Goncharov and Rudenko.  
According to these authors, a uniform way to get `the' fundamental functional equations for the polylogarithms might be to look at them on the moduli spaces $\mathcal M_{0,n+3}$, moreover in relation with the cluster structure of the latter space. \mk

Our result in this text shows that there exists a very natural continuation of the identities 
\eqref{Eq:EFA-3-terms-Log} and $\boldsymbol{\big(\mathcal Ab\big)}$ in weight up to 5. 
Moreover, our description of these new functional equations is uniform with respect to the weight.  If these identities have a geometric origin, they are not related to moduli spaces of marked smooth rational curves but to del Pezzo surfaces. And most importantly, these new functional equations do not involve classical polylogarithms  but different (totally antisymmetric) hyperlogarithms, which would 
 deserve to be investigated more in depth in our opinion.

\subsection{Structure of the paper}
The sequel of the paper is organized as follows.

 In {\bf Section \S\ref{S:Preliminary-Material}}, 
we introduce the material about hyperlogarithms and del Pezzo surfaces which will be used in the sequel.   In particular, it is explained there ({\it cf.}\,Proposition \ref{Prop:Symbolic}) that functional identities satisfied by hyperlogarithms can be handled symbolically, which allows to reduce the proof of our main theorem to the verification that a certain (antisymmetric) tensorial identity holds true.  
 {\bf Section \S\ref{S:Main}} is the main section of this text and is where Theorem \ref{Thm:main} is proved. Two distinct proofs are indicated: the first,  in 
\S\ref{SS:Elementary-Proof}, 
is elementary and relies on linear algebra computations.  Our second proof, detailed  in \S\ref{S:Representation-theoretic-proof}, is more conceptual and uses  arguments and computations within the theory of representation of certain Weyl groups acting naturally on the Picard lattices of the considered del Pezzo surfaces.  
In {\bf Section \S\ref{S:Miscellaneous}}, we discuss several points related to the main result of the present text, such as the case when $r=5$ (or equivalently $d=4$) about which we make a few remarks.  
The final {\bf Section \S\ref{S:Questions}} is devoted to some perspectives or problems  which we believe are interesting and in our opinion deserve further investigations.  The paper ends with an {\bf Appendix} in which we give details about the representation theory computations  used in \S\ref{S:Representation-theoretic-proof} in the case of the quartic del Pezzo surface (which corresponds to the case $r=5$).
\subsection{\bf Acknowledgements}
During the preparation of this text, the author benefited of many enlightening exchanges with Ana-Maria Castravet. He is very grateful to her for that. 
He thanks as well Nicolas Perrin for some discussions about Lie theory and Maria Chlouveraki for her help about the use of GAP for dealing with the characters of Weyl groups. 

\subsection{\bf Warning}
This text is not intended to be published in a peer reviewed mathematical journal. 
It is a long version of a more concise text, written with Ana-Maria Castravet,  
in which all del Pezzo surfaces of degree less than or equal to 5, degree 1 included, are treated in a simple and uniform way.

\section{\bf Preliminary material: hyperlogarithms and del Pezzo surfaces}
\label{S:Preliminary-Material}
This section is devoted to the two main ingredients considered in the paper: one variable hyperlogarithms from the one hand and del Pezzo surfaces on the other hand. Everything here is classical and well-known and is gathered below mainly for the sake of completeness. No proof will be provided but some (canonical or useful in our opinion) references will be given instead. 
\subsection{\bf Hyperlogarithms and the functional they satisfy}
\label{SS:Hyperlogarithms}
We now recall some classical facts about iterated integrals  in order to introduce an algebraic tensorial formalism for dealing with ARs whose components are iterated integrals. This material is discussed 
more in depth in 
\cite[\S1.4]{ClusterWebs} to which we refer the reader for details and references. 

\subsubsection{}
We fix a finite  subset $\mathfrak B\subset \mathbf P^1$ (thought as a set of branching points) and set $H=H_{\mathfrak B}=H^0\big(\mathbf P^1, \Omega^1_{\mathbf P^1}(\mathfrak B)\big)$.
Given  a {\bf weight} $w\in \mathbf N_{>0}$ and a 
$w$-tensorial product   $\otimes_{i=1}^w \omega_i\in H^{\otimes w}$, 
its {\bf iterated integral}
along a continuous path $\gamma : [0, 1]\rightarrow \mathbf P^1\setminus \mathfrak B$,
 is defined as the complex number 
\begin{equation}
\label{E:II}
\int_{\gamma} \omega_1\otimes \cdots \otimes \omega_w = \int_{\gamma_*(\Delta_w)} p_1^*(\omega_1)
\wedge \cdots \wedge p_w^*(\omega_w)
\end{equation}
where 
$p_s$ denotes the projection $\big(\mathbf P^1\setminus \mathfrak B\Big)^w\rightarrow \mathbf P^1\setminus \mathfrak B$ onto the $s$-th factor and $\gamma_*(\Delta_w)$ stands for the image by $\gamma\times \cdots \times \gamma: [0,1]^w\rightarrow \mathbf P^1\setminus \mathfrak B$ of the standard $w$-dimensional simplex  $\Delta_w\subset [0,1]^w$.
\sk

In the case under scrutiny, it can be proved\footnote{This follows easily from   the fact that we are considering iterated integrals on spaces of dimension 1.} that an iterated integral  \eqref{E:II}
 depends only on the homotopy class of $\gamma$ 
  hence  can be viewed as a multivalued holomorphic function of $z=\gamma(1)\in \mathbf P^1\setminus \mathfrak B$
as soon as we assume that $x=\gamma(0)$ is a fixed base point.  
 Choosing an affine coordinate $\zeta$ centered at $x$ on $\mathbf P^1$, for any $z$ sufficiently close to $x$ we denote by $\gamma_z$ the path $[0,1]\ni t\mapsto \zeta^{-1}(t \zeta(z))\in \mathbf P^1$. 
 Then one obtains a well-defined $\mathbf C$-linear map 
 \begin{align*}
 \label{Eq:IIxw}
  {\rm II}^w_{x}=  {\rm II}^w_{\mathfrak B,x} \, : \hspace{0.2cm}
H^{\otimes w} \,  &  \longrightarrow \hspace{0.2cm} \mathcal O_{x} \\ 
\otimes_{i=1}^w \omega_i & \longmapsto \bigg( z \mapsto \int_{\gamma_z}  \omega_1\otimes \cdots \otimes \omega_w \bigg)
\end{align*}
which can be easily seen to be independent of the chosen affine coordinate $\zeta$. 

By definition, an element of  
$$\mathcal H_{x}^{ w}={\rm Im}\Big({\rm II}^w_{x} \Big)  \subset \mathcal O_{x}
$$ 
is (the germ at $x$ of) a {\bf hyperlogarithm of weight $\boldsymbol{w}$} (on $\mathbf P^1$, with respect to $\mathfrak B$). 
 When $x$ is assumed to be fixed, no confusion can arise hence 
 we   drop it from all the notation in what follows (we will write   ${\rm II}^w$, $\mathcal H^w$  and $L_{\boldsymbol{a}}$ instead of ${\rm II}^w_x$, 
$\mathcal H_{x}^w$ and $L_{\boldsymbol{a},x}$, etc.)\sk

  Given a basis $\underline{\nu}=(\nu_1,\ldots,\nu_s)$ of $H$ (thus $s=\lvert \mathfrak B\lvert-1$) and for a word $\boldsymbol{{a}}=a_1\cdots a_w$ of length $w$ with $a_k\in \{1,\ldots,s\}$ for all $k$, we will denote 
 ${\rm II}^w_{x} (\nu_{a_1}\otimes \cdots \otimes \nu_{a_w})$  by  $L_{\boldsymbol{a},x}$.  
One verifies easily that  $L_{\boldsymbol{a},x}$ can be characterized inductively as follows: 
it is the germ  of holomorphic function at $x$  such that $L_{\boldsymbol{a},x}(x)=0$ and satisfying 
$$dL_{\boldsymbol{a},x}= \nu_{a_1} \, L_{\boldsymbol{a}',x}$$
on a neighbourhood of $x$, where 
 $\boldsymbol{a}'$ stands for the subword of $\boldsymbol{a}$ obtained by removing its first letter from it, namely 
 $\boldsymbol{a}'= a_2\cdots a_w$. \sk

\subsubsection{}
The iterated integrals satisfy several nice properties. To state these, it is useful to introduce the following objects constructed by taking direct sums: 
\begin{equation*}
H^{\otimes \bullet}=\oplus_{w\in \mathbf N} H ^{\otimes w}\, , \qquad \mathcal H ^{\bullet}=\oplus_{w\in \mathbf N} \mathcal H^{ w}
\quad \mbox{ and } \quad 
{\rm II}^\bullet
=\oplus_{w\in \mathbf N} {\rm II}^w:\,  H^{\otimes \bullet} \rightarrow 
\mathcal H^{ \bullet}\subset \mathcal O_{x}\, . 
\end{equation*}

One denotes by $\shuffle$ the \href{https://en.wikipedia.org/wiki/Shuffle_algebra}{`shuffle product'} 
defined on the free algebra on the set of words in the alphabet $\{1,\ldots,s\}$ (or equivalently, in the alphabet in the elements $\nu_1,\ldots,\nu_s$ of the chosen basis of $H$). 
Here are interesting properties satisfied by iterated integrals that we will use in the sequel: 
\begin{enumerate} 
\item[1.]  relatively to the algebra structure on $H^{\otimes \bullet}$ induced by the shuffle product on words in $\{1,\ldots,s\}$,
the map ${\rm II}^\bullet : H^{\otimes \bullet} \rightarrow \mathcal O_{x}$  is a  morphism of complex algebras, {\it i.e.} for any two words  $\boldsymbol{a}$ and $\boldsymbol{a}'$, one has 
$
L_{\boldsymbol{a}}\, L_{\boldsymbol{a}'}= L_{ \boldsymbol{a} \shuffle \boldsymbol{a}'}
$
as germs of holomorphic functions at $x$.  It follows that 
$\mathcal H^{ \bullet}_x$ is a subalgebra of $\mathcal O_{x}$;
\sk 
\item[2.]  the iterated integrals $L_{\boldsymbol{a}}$ for 
all words $\boldsymbol{a}$, are $\mathbf C(z)$-linearly independent. Consequently,  the map ${\rm II}^\bullet : H^{\otimes \bullet} \rightarrow \mathcal H_{x}^{\bullet}$ is an isomorphism of complex algebras; 
\sk 
\item[3.] in particular,  the  word $\boldsymbol{{a}}$ (hence its length $w=w(\boldsymbol{a})$) is well-defined by 
 $L_{\boldsymbol{a}} \in \mathcal O_{x}$.  By definition,  $\boldsymbol{{a}}$ is the {\bf symbol}  of the latter (with respect to the  basis $\underline{\nu}$) and $w$ is its {\bf weight}.\sk 
 
More generally, for any iterated integral $L=\sum_{\boldsymbol{a}} \lambda_{\boldsymbol{a}} L_{\boldsymbol{a}}\in \mathcal H^{\otimes\bullet}$ (with 
$\lambda_{\boldsymbol{a}} \in \mathbf C$ non-zero for all but a finite number of words $\boldsymbol{a}$'s), one defines its {\bf weight} $\boldsymbol{w(L)}$ as $\max\{ \,   w(\boldsymbol{a})\, \lvert \, \lambda_{\boldsymbol{a}} \neq 0\,  \}$ and its {\bf symbol} (at $x$ and with respect to the basis $\underline{\nu}$) by 
$$
\boldsymbol{\mathcal S}(L)=\boldsymbol{\mathcal S}_{{\hspace{-0.05cm}}x}(L)=\sum_{w(\boldsymbol{a})=w(L)} \lambda_{\boldsymbol{a}}\cdot {\boldsymbol{a}}\, ; $$
\sk
\item[4.]  a consequence of point {1.}\,above is that $\mathcal H^{\otimes \bullet} $ is a subalgebra of 
$\mathcal O_{x}$ which is graded by the weight: for any weights $w,w'\in \mathbf N$,  one has 
$\mathcal H^{ w} \cdot  \mathcal H^{ w'} \subset   \mathcal H^{w+w'}$;
\sk
\item[5.] analytic continuation  along a path  $\gamma$ joining two points $x$ to $y$ in $\mathbf P^1\setminus \mathfrak B$   induces an injective morphism of  algebras  $\mathcal C^\gamma$ between $\mathcal H_{x}^{\bullet} $ and 
$\mathcal H_{y}^{\bullet} $ which is compatible with the filtrations associated to the gradings by the weight. In other terms,  for any weight $w\geq 0$, one has 
 $$ \mathcal C^\gamma\Big(\mathcal H_{x}^{w} \Big)\subset \mathcal H_{y}^{\leq w}=\bigoplus_{w'\leq w} \mathcal H_{y}^{ w'}\, .$$
 Moreover, the   map induced by $\mathcal C^\gamma$ between ${\rm Gr}^w \mathcal H_{x}^{\bullet}=\mathcal H_{x}^{w}$
and ${\rm Gr}^w \mathcal H_{y}^{\bullet}=\mathcal H_{y}^{w}$
is reduced to the identity, {\it i.e.} for any 
  $L_{\boldsymbol{a},x}\in \mathcal H_{x}^w$, one has 
 $$  \mathcal C^\gamma\Big(L_{\boldsymbol{a},x}\Big)-L_{\boldsymbol{a},y} \in 
  \mathcal H_{y}^{<w}=\oplus_{\tilde w<w}\mathcal H_{y}^{\tilde w}\, ; $$
\item[6.] from the preceding points, it follows  that $(i).$ for any $x\not \in \mathfrak B$, iterated integrals elements of $\mathcal H_x$ extend as global but multivalued holomorphic function on $\mathbf P^1$, with their branch loci all contained in $\mathfrak B$;  $(ii).$ 
the  symbol and the weight  of a hyperlogarithm do not depend on the base point; $(iii).$ hyperlogarithms have unipotent monodromy. 
\end{enumerate}

Because it will be useful at some points in the sequel, 
when a basis $\underline{\nu}$ of $H$ has been fixed, we will identify any 
word $\boldsymbol{a}=a_1\ldots a_w$ with the tensor product 
 $\nu_{\boldsymbol{a}}=\nu_{a_1}\otimes \cdots  \otimes \nu_{a_w}\in H^{\otimes w}$.  Hence ${\mathcal S}(L_{\boldsymbol{a}})$ will refer to the word $\boldsymbol{a}$ or to the associated tensor $\nu_{\boldsymbol{a}}$ depending on what is the most convenient for us. 

\begin{rem}
\label{Pg:Tangential-base-point}
In all the discussion above, 
the choice of the additional base-point $x$ outside $\mathfrak B$ is arbitrary in most cases hence may appear as unnatural. Another approach, encountered in many papers, is to take $x$ as one of the points of $\mathfrak B$ 
 but this requires in addition to specify  a non-trivial  (real) tangent direction $\zeta$ at $x$: one then talks  of a {\it `tangential base point'}. By iterated integrations 
of 1-forms in $H_{\mathfrak B}$  along smooth paths $\gamma:[0,1]\rightarrow \mathbf P^1$ with $
\gamma(]0,1])\subset \mathbf P^1\setminus \mathfrak B$, $\gamma(0)=x\in \mathfrak B$ and $
\gamma'(0)=\zeta$,  one defines a class of multivalued holomorphic functions  on $\mathbf P^1\setminus \mathfrak B$ which enjoys similar versions of all the properties  discussed above. 
However, a certain technical problem has to be taken into account, namely that a given symbol  $\boldsymbol{a}\in H^{\otimes \bullet}$ may be  {\it `divergent'}\footnote{The simplest example of such a divergent symbol is $\nu_0=du/u$ in case $x=0$ and $\zeta=\partial/\partial z_1$ when $z_1$ stands for the real  part of the variable $z\in \mathbf C$: for any smooth path $\gamma$ emanating from the origin with tangent vector $\zeta$ at this point, the weight 1 integral $\int_\gamma \nu_0$ is divergent, precisely because the integrand has a logarithmic singularity at the origin.\label{foufou}}. For such a symbol, it is necessary to define the iterated integral value $\int_\gamma \nu_{\boldsymbol{a}}$ by means of a certain 
 `regularization process' 
 which is well-known and just  mentioned here for the sake of completeness (for more details, see 
  \cite[\S2.1]{BanksPanzerPym}{\rm )}. 
\end{rem}

\subsubsection{}
 Most of the time, it is assumed that $\mathfrak B$ has been normalized such that $\infty$ belongs to it. In this case, the basis of $H$ one usually uses is the one formed by the $\eta_k=dz/(z-b_k)$ for 
$k=1,\ldots,\lvert \mathfrak B\lvert -1$  
 where  $k\mapsto b_k$ stands for a 1-1 labelling of the elements of $\mathfrak B\setminus \{\infty\}$.   
\sk

As a classical example, let us consider the case when $\mathfrak B$ has 3 elements. In this case, it is often assumed that $
 \mathfrak B=\{0, 1,\infty\}$, $(\omega_0,\omega_1)=(dz/z,dz/(1-z))$ is taken as a basis for $H$ and one then speaks of  {\bf polylogarithms} instead of hyperlogarithms.\footnote{This is the most common convention when $\mathfrak B$ has cardinality 3 but at least another one exists: when dealing with cluster algebras, it is more convenient to take $
\mathfrak B_{-1}=  \{0, -1,\infty\}$ as ramification locus and $(dz/z, dz/(1+z))$ for a basis of $H=H^0\big(\mathbf P^1, \Omega^1_{\mathbf P^1}({\rm Log}\,\mathfrak B_{-1})\big)$.}
 This is justified by the following fact: 
 using the `unitary horizontal tangential base point at the origin' ({\it cf.}\,Remark \ref{Pg:Tangential-base-point} above and especially footnote \ref{foufou})
 as a base-point (as is customary to do in the polylogaithmic case),  we get that 
  the iterated integral 
 associated to the weight $n$ symbol
 $$\omega_0^{\otimes (n-1)}\otimes \omega_1=(dz/z)\otimes \cdots \otimes (dz/z) \otimes \big(dz/(1-z)\big)$$ 
coincides with the $n$-th classical  polylogarithm ${\bf L}{\rm i}_n$, 
on the whole unit disk $\mathbf D$ for instance.\sk

Relatively with the chosen tangential base point, both ${\rm Log}$ and  ${\rm Log}(1-x)=-{\bf L}{\rm i}_1(x)$ can be seen as weight 1 polylogarithmic integrals on $\mathbf P^1\setminus \{0,1,\infty\}$, with associated symbols $\omega_0$ and $-\omega_1$ respectively. It follows that  Rogers's dilogarithm $R$ is an iterated integral of the same kind, of weight 2 and  whose symbol is given by 
\begin{align}
 \label{Eq:S(R)}
 \mathcal S(R)= \, &\,  \mathcal S\big({\bf L}{\rm i}_2\big)-\frac{1}{2}\mathcal S\big( {\rm Log} \cdot
 {\bf L}{\rm i}_1
 \big) 
\nonumber  
 \\ =\, &\,  \mathcal S\big({\bf L}{\rm i}_2\big)-\frac{1}{2}\mathcal S\big( {\rm Log}\big) \shuffle 
\mathcal S \big( {\bf L}{\rm i}_1
 \big) 
 \\
 =\, &\,  01-\frac{1}{2} 0\shuffle 1=01-\frac{1}{2}\big( 01+10\big)= \frac{1}{2}\big( 01-10\big)
 \nonumber  
\end{align}
where, to simplify the formulas,  we have written $0$ and $1$ instead of $\omega_0$ and $\omega_1$ respectively.

\subsubsection{}
\label{SSS:lalala}
For any $w>0$, the permutation group $\mathfrak S_w$ acts naturally on $ H^{\otimes w}$ by 
permutation of the components of the tensor products, namely 
one has 
$$\sigma\cdot (\eta_1\otimes \ldots\otimes \eta_w)=
(\eta_1\otimes \ldots\otimes \eta_w)^\sigma=
\eta_{\sigma(1)}\otimes \ldots\otimes \eta_{\sigma(w)}$$ for every $(\sigma,\eta_1\otimes \ldots,\otimes \eta_w)\in \mathfrak S_w\times H^{\otimes w}$.  
For $\nu=\nu_1\otimes \ldots \otimes \nu_s\in H^{\otimes s}$ where 
$(\nu_1,\ldots,\nu_s)$  is a given basis of $H$, 
one defines two special elements of $H^{\otimes s}$ by setting
$$
{\rm Sym}^s\big( \nu)=\frac{1}{s!}
\sum_{\sigma \in \mathfrak S_s} 
\nu_{\sigma(1)}\otimes \ldots,\otimes \nu_{\sigma(s)}
\qquad \mbox{ and }\qquad 
{\rm Asym}^s\big( \nu)=\frac{1}{s!}
\sum_{\sigma \in \mathfrak S_s}  \epsilon(\sigma)\,
\nu_{\sigma(1)}\otimes \ldots,\otimes \nu_{\sigma(s)}\, .
$$
 Up to a non zero constant, both do not depend on the choice of a basis for $H$.

 In this paper, the  hyperlogarithms we will mainly consider are those whose symbol is the antisymmetrization 
 ${\rm Asym}^s\big( \nu)$. For any $z\in \mathbf P^1\setminus \mathfrak B$, the hyperlogarithm $AI_z^s$ is defined as the holomorphic germ at $z$ 
 with this antisymmetrization for symbol: 
 $$
 AI_z^s= 
 {\rm II}^s_z\Big( {\rm Asym}^s\big( \nu\big)\Big) \in \mathcal O_{\mathbf P^1,z}\, .
 $$
 In this notation, 
  `A' and `I' refer to the words  `{\it Antisymmetric}' and `{\it Integral}' respectively.
 \sk 
 
The case when $s=2$ is already interesting, since one has ${\rm Asym}_2(a_1a_2)=(a_1a_2-a_2a_1)/2$, a symbol which is formally the same as the one of Rogers' dilogarihtm (see \eqref{Eq:S(R)}). 
Assuming that the letters $a_1$ and $a_2$ respectively correspond to the 1-forms $dz/z$ and $dz/(1-z)$ 
(which does not require any loss of generality), we get that  for any $z\in \mathbf P^1\setminus \{0,1,\infty\}$, one has 
$$ {}^{}\hspace{0.7cm}
AI_z^2(x)= -
\frac{1}{2}
\bigintsss_{z}^x
\left(
\frac{{\rm Log}\left(\frac{1-\sigma}{1-z}\right)}{\sigma} -
\frac{{\rm Log}\left(\frac{\sigma}{z}\right)}{1-\sigma}
\right)\, d\sigma
$$
for $x\in \mathbf P^1$ sufficiently close to $z$, hence $AI_z^2$ can be seen as a holomorphic version of Rogers' dilogarithm $R$, localized at $z$.

It is interesting to wonder whether $AI_z^s$ can be constructed from hyperlogarithms of lower order,  in other words whether it belongs to the subalgebra generated by the hyperlogarithms of order at most $s-1$. 
Since ${\rm II}_z^\bullet$ is an injective morphism of algebras, this is equivalent to knowing whether the symbol ${\rm Asym}^s(a_1\cdots a_s)$ can be written as a linear combination of non trivial shuffle products.

This is not the case for $s=2$.  Indeed, it is well-known that the dilogarithmic symbol 
$$\mathcal R_{a_1a_2}={\rm Asym}^2(a_1a_2)=
\frac{1}{2}\,\bigg( \, a_1a_2-a_2a_1\bigg)$$
(compare with \eqref{Eq:S(R)})  cannot be written as a linear combination 
of the four shuffles $a_i\shuffle a_j$ with $i,j=1,2$.  
The situation is  different for $s\geq 3$. 
We only discuss below the cases $s=3,4,5$  which are the ones we 
 are dealing with in this paper.

 Denoting by ${a}^{s}$ the word $a_{1}\cdots a_{s}$ for any $s$, 
one has
 \begin{align}
 \label{F:Asym-s}
{\rm Asym}^3\Big(\, {a}^3\, 
\Big)=\, & \, \frac{1}{3}\,\bigg( \, 
 a_1\shuffle \mathcal R_{a_2 a_3}-a_2\shuffle \mathcal R_{a_1a_3}
+ a_3\shuffle \mathcal R_{a_1a_2}\, 
 \bigg)
 \nonumber
 \\
{\rm Asym}^4\Big(\,
{a}^4\,
\Big)=\, & \,\frac{1}{6}\, \bigg( \, 
\mathcal R_{a_1a_2}\shuffle \mathcal R_{a_3a_4}
-\mathcal R_{a_1a_3}\shuffle \mathcal R_{a_2 a_4}
+\mathcal R_{a_1a_4}\shuffle \mathcal R_{a_2a_3} \bigg)
\\{\rm Asym}^5\Big(\,
{a}^5\,
\Big)
=\, & \,
\frac{1}{5}\, a_1\shuffle  {\rm Asym}^4\big(\,
a_2 \cdots a_5
\big)
+\frac{2}{5!} \hspace{-0.1cm}
\sum_{\sigma \in {\rm Fix}(1)}
\epsilon(\sigma) \, 
\big( a_{1}a_2\shuffle   a_3a_4a_5\big)^{\sigma}\, , 
 \nonumber
 \vspace{-0.2cm}
\end{align}
where we have set 
$
\big( a_{1}a_2\shuffle 
 a_3a_4a_5\big)^{\sigma}=
a_{\sigma(1)}a_{\sigma(2)}\shuffle 
 a_{\sigma(3)}a_{\sigma(4)}a_{\sigma(5)}
$
 for any $\sigma  \in \mathfrak S_5$ and  with  ${\rm Fix}(1)$ standing for the subgroup of 
 $\mathfrak S_5$ (isomorphic to 
$\mathfrak S_4$) formed by the permutations fixing 1.

The formulas above can be checked by elementary but tedious formal computations,  are rather concise but not written in (anti)symmetric form (except the formula for 
 ${\rm Asym}^3\big(\, {a}^3\, 
\big)$) and are certainly specializations  of a general formula holding true for 
$s$ arbitrary. Since we only deal with the cases when $s\leq 5$ in this paper, we do not 
elaborate further on the general case. \sk

From the formulas \eqref{F:Asym-s}, it follows that for $s=3,4,5$, the antisymmetric hyperlogarithm $AI_z^s$ can be expressed as a linear combination of products of hyperlogarithms of order $<s$.  Let us elaborate with more details on the case $s=3$. 
 One denotes  $b_1,\ldots,b_s,b_{s+1}$ the elements of $\mathfrak B$ with $b_{s+1}=\infty$. Then given a base point $z\in \mathbf P^1$ not in $\mathfrak B$,  $AI_z^s$ and $AI_{\hat{\imath},z}^s$ stand for the (germs of) hyperlogarithms at $z$ whose symbols are  
 the antisymmetrizations of 
$\otimes_{k=1}^s d{\rm Log}(x-b_k)$ and $\otimes_{k\neq i}^s d{\rm Log}(x-b_k)$ respectively.  Then the  formula for ${\rm Asym}^3\big(\, {a}^3\, 
\big)$ in \eqref{F:Asym-s} is equivalent to the fact that 
the following equality holds true for any $x$ sufficiently close to $z$ on the projective line: 
  \begin{equation}
  \label{Eq:AI-z-3}
 AI_z^3\big(x\big)=
 \frac13 
 \sum_{i=1}^3 (-1)^{i-1} {\rm Log}\left( 
 \frac{x-b_i}{z-b_i}
 \right)\cdot AI_{\hat{\imath},z}^{2}\big(x\big) 
 \, .
\end{equation}

\subsubsection{}
\label{SSS:M}
Until now iterated integrals were considered and discussed only on the projective line (minus a finite set) but they actually  can be considered in many more general settings, for instance 
when working on a complex manifold $M$ of arbitrary dimension with 
differential forms belonging to a fixed vector subspace $H$ of the space $H^0(M,\Omega^1_M)$ of holomomorphic 1-forms on $M$. 
For  any tensor  $ \eta=\otimes_{s=1}^w \eta_s\in H^{\otimes w}$ of weight $w>0$ and for any smooth path $\gamma$ in $M$, 
 one  defines  a complex value $\int_\gamma {\eta}$ by means of formula  \eqref{E:II}.  However,  unless when ${\eta}$ satisfies some {\it integrality conditions} first given by  K.T. Chen in full generality,\footnote{These integrability conditions can be made explicit easily but there is no point for doing so  here.} the iterated integral $ \int_\gamma {\eta}$ may depend on the path $\gamma$ and not only on its homotopy class (with fixed extremities).  
A simple but interesting case when  
Chen's integrability conditions 
 are automatically satisfied is when all the elements of $H$ are closed (which happens for instance when $\dim(M)=1$). In this case, everything that has been said about dimension 1 iterated integrals generalizes  straightforwardly.  
 In particular, for any $m\in M$, there is a map
 $${\rm II}^w_m: \, H^w \rightarrow \mathcal O_{M,m}$$ which is easily seen to be an injective
 morphism of complex algebras. 
 One can thus speak of iterated integrals on $M$, 
of weights of such functions, and  of their symbols as well, but this only once  a basis for $H$ has been chosen.

\vspace{-0.2cm}
\subsubsection{} 
\label{SS:Ui:X->Ci}
Let $X$ be a projective manifold coming with 
$d$ holomorphic fibrations  $U_i: X\rightarrow C_i$ ($i=1,\ldots,d$)  over projective curves. 
We assume that there exists a hypersurface $\Sigma$ of $X$ such that the restriction of each 
$U_i$ to $Y=X\setminus Z$ is a regular submersion. Hence $C_i^*=U_i(X)$ is an open subset of $C_i$, that is 
 $C_i=C_i^*\sqcup \mathfrak B_i $ where  
$\mathfrak B_i=C_i\setminus C_i^*=U_i(\Sigma)$ is finite. 
 For any $i$, under the assumption 
 that the dimension    $s_i$ of the space of logarithmic 1-forms $\widetilde H_i=H^0\big({C}_i, \Omega^1_{{C}_i}\big({\rm Log}\, \mathfrak B_i \big)\big)$ is positive, we choose a basis $\underline{\tilde \nu}^i=({\tilde \nu}^i_1,\ldots, \tilde {\nu}^i_{s_i})$ of it. Then the pull-backs $\nu^i_k=U_i^*\big( {\tilde \nu}^i_k)$ for $k=1,\ldots,s_i$ span a space of dimension $s_i$ of rational 1-forms on $X$ whose restriction on $Y$ are holomorphic. We set 
 $$
  H=\sum_{i=1}^d H_i 
 \;   \qquad \mbox{ with }\qquad   \;
 H_i=U_i^*\big( \widetilde H_i\big)={\rm Span}_{\mathbf C} \Big( \big\langle 
 \,
 \nu^i_1,\ldots,\nu^i_{s_i}\,
\big \rangle \Big)
\qquad \mbox{for } i=1,\ldots,d\, , 
$$
all these spaces being considered as subspaces of $ H^0(Y,\Omega_Y^1)\cap \Omega^1_{\mathbf C(X)}$.
\sk

For $y\in Y$, we set $y_i=U_i(y)$ for any $i\in \{1,\ldots,d\}$.  
On $Y$, the 1-forms elements of $H$  are closed because all the $\nu^i_k$ are. 
From the  remarks in \S\ref{SSS:M},  
working with 
1-forms in $H$ on $Y$ 
and 
with
1-forms belonging to $\widetilde{H}_i$ on $C_i^*$, one gets injective morphisms of $\mathbf C$-algebras 
\begin{equation}
\label{Eq:II-II}
{\rm II}^\bullet_y : H^{\otimes\bullet} \stackrel{\sim}{\longrightarrow} 
\mathcal H_y^\bullet \hookrightarrow \mathcal O_{Y,y}
\quad  \qquad \mbox{ and }\qquad\quad
{\rm II}^\bullet_{y_i}:  {\widetilde H}_i^{\otimes \bullet} 
 \hookrightarrow \mathcal O_{C_i,y_i} 
 \qquad \big(i=1,\ldots,d\big)
\end{equation}
which enjoy the property of being compatible with respect to pull-back under $U_i$, 
{\it i.e} the relation 
\begin{equation}
\label{Eq:II o Ui*=Ui* o II}
{\rm II}^\bullet_y \circ U_i^*= U_i^* \circ {\rm II}^\bullet_{y_i}
\end{equation}
 holds true, where the map $U_i^*$ in the LHS of this equality is the 
  injective linear morphism 
$U_i^*: {\widetilde H_i}^{\otimes \bullet} \stackrel{\sim}{\longrightarrow}  {H_i}^{\otimes \bullet}\hookrightarrow {H}^{\otimes \bullet}$ induced by the standard pull-back ${\widetilde H_i}\stackrel{\sim}{\rightarrow} {H_i}: \tilde \nu^i_k\mapsto \nu^i_k=U_i^*(\tilde \nu^i_k)$. More explicitly, 
 this means that  for  any $w>0$ and any tensor $\tilde \nu\in {\widetilde H_i}^{\otimes w}$, 
 one has $$L_{U_i^*(\tilde \nu),y}=L_{\tilde \nu,y_i}\circ U_i\, .$$

\subsubsection{} 
\label{SS:Ui:X->Ci-2}
We continue to use the setting and the notations introduced just above, we just assume 
that the $U_i$'s define foliations which are pairwise transverse on $Y$ hence define
a `web' by hypersurfaces $\boldsymbol{\mathcal W}=\boldsymbol{\mathcal W}(U_1,\ldots,U_d)$ on $Y$. 
Since we will only consider this case, we assume from now on that the curves $C_i$ all are rational, {\it i.e.} for every $i=1,\ldots,d$, one has 
$C_i=\mathbf P^1$ and $\mathfrak B_i$ has at least three elements, $\infty$ being one of them.

An effective algebraic approach to build some ARs for $\boldsymbol{\mathcal W}$ can be deduced from the  fact that the algebra morphisms in \eqref{Eq:II-II} are injective together with property \eqref{Eq:II o Ui*=Ui* o II}. 
For any $w>0$,  since each $H_i=U_i^* \Big({\bf H}^0\Big(\mathbf P^1,\Omega^1_{\mathbf P^1}({\rm Log}\,\mathfrak B_i \big)\Big)\Big)$ naturally embeds into $H$, one can construct a linear map
$$
\Phi^w :  \bigoplus_{i=1}^d H_i^{\otimes w} \longrightarrow H^{\otimes w}\, , \hspace{0.1cm}
( \nu_i )_{i=1}^d \longrightarrow  \sum_{i=1}^d  \nu_i
$$
whose kernel $\boldsymbol{K}^w$ can be seen as the set of  weight $w$ hyperlogarithmic ARs for $\boldsymbol{\mathcal W}$ (with respect to the chosen first integrals $U_i$). Indeed, given $( \nu_i )_{i=1}^d\in 
\oplus_{i=1}^d H_i^{\otimes w}$, there exists a unique $d$-tuple $(\tilde \nu_i)_{i=1}^d 
\in\oplus_{i=1}^d  \widetilde H_i^{\otimes w}$ such that 
$\nu_i=U_i^*(\tilde \nu_i)$ for every $i$.  Then for any $y\in Y$, it follows easily from above that 
$( \nu_i )_{i=1}^d$  belongs to $\boldsymbol{K}^w$ if and only if $\sum_{i=1}^d L_{\tilde \nu_i,y_i}(U_i)=0$ as a holomorphic germ at $y$. 
This gives rise to a linear identification 
$$
\boldsymbol{K}^w \simeq \boldsymbol{{\mathcal H} \hspace{-0.05cm} LogAR}^w_y(\boldsymbol{\mathcal W})$$
where the RHS stands for  the space of germs at $y$ of  abelian relations for $\boldsymbol{\mathcal W}$ whose components are hyperlogarithms of weight $w$: 
$$ \boldsymbol{{\mathcal H} \hspace{-0.05cm} LogAR}^w_y(\boldsymbol{\mathcal W})=
\bigg\{\, 
\big( L_i)_{i=1}^d \in \prod_{i=1}^d \widetilde{\mathcal H}_i^{w} \, \big\lvert 
\, 
\sum_{i=1}^d L_i(U_i) \equiv 0 
\,
\bigg\}\, .
$$

Assume now that the branching loci $\mathfrak B_i$ all have the same cardinal, denoted by $s$. Then $\wedge^s H_i$ is 
a vector subspace of $H_i^{\otimes s}$ of 
dimension 1,  spanned by 
$$\Omega_i=\omega^i_1\wedge \ldots \wedge \omega_s^i={\rm Asym}^s\big(\omega^i_1\otimes \ldots \otimes \omega_s^i\big)$$
where  $(\omega_k^i)_{k=1}^s$ stands for the (fixed) basis of $H_i$ we are working with.  
For any $i$, let $\widetilde \Omega_i$ be the element of $\wedge^s \widetilde H_i$ such that $\Omega_i=U_i^*\big( \widetilde \Omega_i \big)$. 
Given  $y\in Y$, one sets $y_i=U_i(y)\in  \mathbf P^1\setminus \mathfrak B_i$ 
and one denotes by  $AI^s_{i,y_i}$ the germ of hyperlogarithm at $y_i$ 
whose symbol is $\widetilde \Omega_i$, {\it i.e.}
$$ AI^s_{i, y_i}={\rm II}_{y_i}^s\Big(\widetilde \Omega_i\Big)\in \mathcal O_{\mathbf P^1\setminus \{\mathfrak B_i\},y_i}\, .$$

The following proposition then follows immediately from the discussion above: 

\begin{prop} 
\label{Prop:Symbolic}
For $(c_i)_{i=1}^d \in \mathbf C^d$, the following conditions are equivalent: 
\begin{enumerate}
\item[1.]  one has $\sum_{i=1}^d c_i\,\Omega_i=0$ in $\wedge^s H\subset H^{\otimes s}$;
\sk
\item[2.] there exists $y\in Y$, such that $\sum_{i=1}^d c_i\,{AI}^s_{i,y_i}(U_i)=0$ as a holomorphic germ at $y$ on $Y$;
\sk
\item[3.] for any $y\in Y$, one has $\sum_{i=1}^d c_i\,{AI}^s_{i,y_i}(U_i)=0$ as a holomorphic germ at $y$ on $Y$.
\sk
\end{enumerate}
\end{prop}

\subsection{\bf Del Pezzo surfaces and root systems}
 Here we review 
very classical and nowadays well-known material about del Pezzo surfaces. 
No proofs are provided below. For some  proofs as well as for details and references, we 
refer to the book chapters \cite[Chap.IV]{Manin} or 
\cite[Chap.8]{Dolgachev}. 

\subsubsection{Notations and general facts} 
\label{SSS:dPr-properties}
Here is a list of notations and facts we will use in the sequel. 
\begin{enumerate}
\item[{\bf (a).}]  We denote by  $r$  an integer in $\{3, 4,5,6,7\}$ and $d$ stands  for $9-r\in \{2,3,4,5,6\}$.
\sk
\item[{\bf (b).}] Here $p=p_1+\cdots+p_r$ denotes a 0-cycle of degree $r$ on $\mathbf P^2$, with the $p_i$'s assumed to be in general position: no three of the $p_i$'s are aligned and  no six of them lie on a same conic.
\mk
\item[{\bf (c).}] One denotes by $\mu=\mu_r : {\rm dP}_d={\bf Bl}_p(\mathbf P^2)\rightarrow \mathbf P^2$ the blow-up of the projective plane at the $p_i$'s: it is a degree $d$ del Pezzo surface, {\it i.e} its canonical  sheaf,  denoted by $K_r=K_{X_r}$, is anti-ample  with anticanonical degree $(-K_r)^2$ equal to $d$;
\mk
\item[{\bf (d).}]  The anticanonical linear system $\lvert -K_r\lvert$ is very ample for $r\leq 6$ and one denotes 
the associated anticanonical model by $X_r=\varphi_{\lvert -K_r\lvert}(\mathbf P^2)$. It is a smooth surface of degree $d$ in $\mathbf P^d$.  For $r=7$,  $\varphi_{\lvert -K_7\lvert}: {\rm dP}_2\rightarrow \mathbf P^2$ is a covering ramified along a smooth quartic curve but   $-2K_7$ is very ample and allows to embeds ${\rm dP}_2$ into  $\mathbf P(1,1,1,2)$.  We denote by $X_7$ the image of ${\rm dP}_2$ into 
this weighted projective space. Except if there is a reason not to do so,  we will allow ourselves to write $X_r$ for ${\rm dP}_d$.
 \mk
\item[{\bf (e).}]  For $i=1,\ldots,r$, we denote by $\ell_i$ the class of the exceptional divisor $\mu^{-1}(p_i)\simeq \mathbf P^1$ in the Picard group ${\bf Pic}(X_r)$ of $X_r$.  If $h$ stands for the class of the pull-back under $\mu$ of a general line in $\mathbf P^2$, then one has ${\bf Pic}(X_r)=\mathbf Z h\oplus \bigoplus_{i=1}^r \mathbf Z \ell_i\simeq \mathbf Z^{r+1}$. Moreover, the intersection pairing $(\cdot,\cdot): {\bf Pic}(X_r)^2\rightarrow \mathbf Z$ is non degenerate with signature $(1,r)$.
 \mk 
\item[{\bf (f).}]  The canonical class of $X_r$ is given by $K_r=K_{X_r}=-3h+\sum_{i=1}^r \ell_i$. It is   such that 
$K^2_r=d$.  Its orthogonal $K_r^\perp=\big\{ \, \rho \in {\bf Pic}(X_r)\, \big\lvert \, 
(\rho\cdot K_r)=0\, \big\}$ is free of rank $r$ and spanned by the $r$ classes 
\begin{equation}
\label{Eq:Fundamental-Roots}
\rho_i=l_i-l_{i+1}  \quad \mbox{ for } \quad i=1,\ldots,r-1\qquad 
\mbox{ 
 and } \qquad \rho_r=h-\ell_1-\ell_2-\ell_3\, .
 \end{equation}
  \mk 
\item[{\bf (g).}]  
\vspace{-0.5cm}
Each of the class $\rho_i$ is such that 
$(\rho_i,K_r)=0$ and $ \rho_i^2=-2$. Together with the positive definite symmetric form $ - \big( \cdot,\cdot)\lvert_{K_r^\perp}$, the $\rho_i$'s define a root system of type $E_r$, with the convention that $E_4=A_4$ and $E_5=D_5$ (see Figure \ref{Fig:DynkinDiagram} below). 
The corresponding set of roots is $\mathcal R_r=\{ \, \rho\, \lvert \, (\rho,K_r)=0\, \mbox{ and } \, \rho^2=-2\}$ which is contained into the associated 
root space $R_r=K_r^\perp\otimes_{\mathbf Z} \mathbf R=\oplus_{i=1}^r \mathbf R\,\rho_i$. 
Endowed with (the restriction of)  $-(\cdot,\cdot)$, the latter is  a Euclidean vector space;
 \mk 
 \begin{figure}[h!]
\begin{center}
\scalebox{2.1}{
 \includegraphics{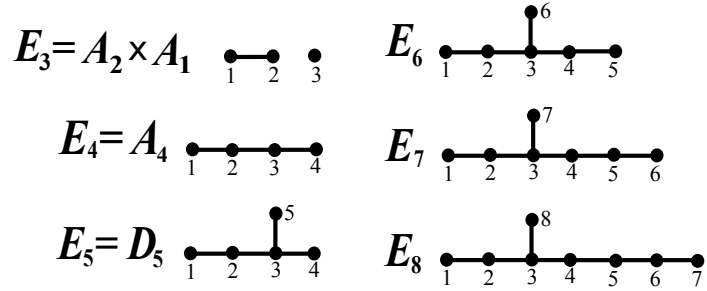}}
 \vspace{-0.1cm}
\caption{Dynkin diagrams $E_r$ (with $k$ standing for 
 $\rho_k$ for any $k=1,\ldots,r$)} 
\label{Fig:DynkinDiagram}
\end{center}
\end{figure}
\item[{\bf (h).}] For each root $\rho$, the map 
\begin{equation}
\label{Eq:s-rho}
s_\rho: R_r\rightarrow R_r, \, d\mapsto d+(d,\rho)\,\rho
\end{equation}
 is a hyperplane reflection of $R_r$, admitting $\rho$ as $(-1)$-eigenvector hence with invariant hyperplane $\rho^\perp\subset R_r$.  The subgroup of the group of orthogonal transformations of  $R_r$ 
spanned by the $s_{\rho_i}$'s for $i=1,\ldots,r$ is a Coxeter group (that is it is finite). It is the so-called `Weyl group of type $E_r$' and will be denoted by $W(E_r)$ or shortly by $W_r$.
\end{enumerate}

\subsubsection{Lines and conics} 
\label{SSS:Lines-Conics-properties}
We now discuss several notions, objects and facts related to the lines  and to the conics contained in del Pezzo surfaces. For more specific additional references, see for instance \cite{Lee1,Lee2} and the classical references (by Coxeter and Du Val) cited therein.

\begin{enumerate}
\item[{\bf (i).}] By definition, a {\it `line'} is a class $l\in {\bf Pic}(X_r)$ such that $(K_r,l)=(l,l)=-1$. 
This terminology makes sense for the following reason: from $l^2=-1$, one deduces first that the associated linear system $\lvert l\lvert$ is a singleton which is a smooth rational curve embedded in $X_r$, denoted in the same way. Second, when $-K_r$ is very ample (that is when $r \leq 6$), the condition $(K_r,l)=-1$ means that $\varphi_{\lvert -K_r\lvert}(l)$ is a projective line in the target projective space $\mathbf P^d$. 
\mk 
\item[{\bf (j).}]  The set  $\mathcal L_r$ of lines included in $X_r$ is finite and its elements all can be explicitly given. For instance, see  Table \ref{Ff:r=7} below in the case when $r=7$ and \cite[\S25.5.2]{Manin} in full generality. 
\mk 
\item[{\bf (k).}]  The Weyl group $W_r$ acts transitively on the set  of lines $\mathcal L_r$ which gives rise to a structure of $W_r$-module on $\boldsymbol{R}^{{\mathcal L}_r}$ for any ring $\boldsymbol{R}$ (such as the ring of integers $\mathbf Z$ or the field $\mathbf C$). Moreover, since $W_r$ acts by permutations on $\mathcal L_r$ and because the natural map $\boldsymbol{Z}^{{\mathcal L}_r}\rightarrow {\bf Pic}(X_r)$ is $W_r$-equivariant, it follows that 
\begin{equation}
\label{Eq:Fact-Multiplicity}
\begin{tabular}{l}
{\it both the trivial and the reflection representations appear with positive} \\ 
{\it  multiplicity in the decomposition 
of $\mathbf Z^{ {\mathcal L}_r}$ 
 into $W_r$-irreducibles.}\mk
\end{tabular}
\end{equation}
\item[{\bf (l).}]  
To simplify, one sets $s_i=s_{\rho_i}$ for $i=1,\ldots,r$.  
The subdiagram of $E_r$ formed by all its vertices at the exception of the $(r-1)$-th one (with the edge adjacent to this vertex removed as well) is isomorphic to the Dynkin  diagram $E_{r-1}$ hence 
$\langle \, s_i \, \lvert \,  i\neq r-1\, \rangle\subset W_r$
 is isomorphic to $W_{r-1}$ and will be denoted the same (a bit abusively). One verifies that, when $r>3$, this subgroup coincides with the stabilizer ${\rm Stab}(\ell_r)$ of the $r$-th exceptional line $\ell_r$ from which one gets the following formula for the cardinal  $l_r$ of $\mathcal L_r$: one has 
 $l_r=\lvert \mathcal L_r\lvert =\lvert W_r \lvert / \lvert W_{r-1}\lvert $. For $r=3$, one has 
 ${\rm Stab}(\ell_3)=\langle s_1\rangle \simeq \{\pm 1\}$ hence $l_3=\lvert \mathcal L_3\lvert =\lvert W_3 \lvert / 2=6 $.
\mk 
\item[{\bf (m).}]  Let $E_{r-1}'$ be the full subdiagram of $E_r$ formed by its $r-1$ first vertices. It is isomorphic to the Dynkin diagram $A_{r-1}$ hence the associated subgroup 
$W_{r-1}'=W(E_{r-1}')=\langle \, s_1,\ldots,s_{r-1}\,\rangle \subset W_r$ is isomorphic to the Weyl group of type $A_{r-1}$, that is to the symmetric group $\mathfrak S_{r}$. 
A natural explicit  isomorphism $W_{r-1}'\simeq \mathfrak S_{r}$ can be obtained 
by noti\-cing that 
the set $\{\ell_1,\ldots,\ell_{r}\}\subset {\bf Pic}(X_r)$ 
 is left invariant by 
$W'_{r-1}$  and that under the bijection 
$\{\ell_k\}_{k=1}^r \rightarrow \{ k\}_{k=1}^r$, 
 $\ell_k \mapsto k$, 
 $s_i$ identifies with the transposition $(i,i+1)$ for  $i=1,\ldots,r-1$. 
\mk 
\item[{\bf (m).}]  A `{\it conic class}' on $X_r$ is an element  $\mathfrak c\in {\bf Pic}(X_r)$ such that $(K_r,\mathfrak c)=(\mathfrak c,\mathfrak c)=0$.  From the second condition, one deduces that for any such class, 
the linear system $\lvert \mathfrak c\lvert$ has dimension 1 and gives rise to a  fibration
 $ \mathfrak C_{\mathfrak c}
: X_r\rightarrow 
\lvert \mathfrak c\lvert^\vee\simeq \mathbf P^1$ whose fibers are conics ({\it i.e.} rational 
curves with anticanonical degree 2);  
\mk 
\item[{\bf (n).}] The  set $\mathcal K_r$ of conic classes is finite and  can be described explicitly (see \S\ref{SSS:r=7} below in the case when $r=7$). Moreover, the Weyl group $W_r$ acts transitively on $\mathcal K_r$.  
Let $E_{r-1}''$ be the subdiagram of $E_r$ with vertices the $\rho_k$'s for $k=2,\ldots,r$.
  One has  $E_{2}''\simeq A_1+A_1$, $E_{3}''\simeq A_3$ whereas $E_{r-1}''$  is of Dynkin type $D_{r-1}$ when $r>4$. 
  Consequently, $W_{r-1}''=\langle s_2,\ldots,s_r\rangle$ 
is isomorphic to   $\mathfrak S_2\times \mathfrak S_2$ and $\mathfrak S_4$ for $r=3$ and $r=4$ respectively, and to 
 $W(D_{r-1})\simeq  \big( \mathbf Z/2\mathbf Z\big)^{r-2}\ltimes \mathfrak S_{r-1}$ when $r>4$. In any case, one verifies that 
$W_{r-1}''$ coincides with the stabilizer of 
 $h-\ell_1\in \mathcal K_r$. It follows that if $\kappa_r$ stands for the cardinal of $\mathcal K_r$, then $\kappa_3=3$, 
$\kappa_4=5$ and 
$\kappa_r=\lvert W_r\lvert /\lvert W_r''\lvert =
\lvert W_r\lvert / \big( 2^{r-2}\, (r-1)!\big)$ for $r=5,6,7$. 
\mk 
\item[{\bf (o).}]  Each conic fibration $\mathfrak C_{\mathfrak c}: X_r\rightarrow  \mathbf P^1$ admits exactly $r-1$ degenerate fibers, all of which are formed by two  elements of $\mathcal L_r$ (two lines) intersecting transversely.   Since any line in $X_r$ appears as a component of a reducible fiber of a conic fibration on $X_r$, it follows in particular that 
the `{\it conical web}' on $X_r$, which by definition is the web by conics $\boldsymbol{\mathcal W}_{X_r}= \boldsymbol{\mathcal W}\big( \mathfrak C_{\mathfrak c}\, \lvert \, \mathfrak c \in \mathcal K_r\, \big)$, is a web by conics on $X_r$ which is regular 
on the complement $Y_r=X_r\setminus L_r$ of the union of lines $L_r=\cup_{\ell \in \mathcal L_r} \ell \subset X_r$. 
\end{enumerate}

The objects and quantities just considered above regarding lines and conic classes are given in explicit form in the following table: 
\begin{table}[!h]
\scalebox{1}{\begin{tabular}{|l||c|c|c|c|c|c|}
\hline
${}^{}$ \hspace{0.4cm}  
\begin{tabular}{c}\vspace{-0.35cm}\\
$\boldsymbol{r}$
\vspace{0.13cm}
\end{tabular} & $\boldsymbol{3}$ &  $\boldsymbol{4}$ & $\boldsymbol{5}$  & $\boldsymbol{6}$ & $\boldsymbol{7}$  & $\boldsymbol{8}$\\ \hline \hline
${}^{}$ \quad \begin{tabular}{c}\vspace{-0.35cm}\\
$\boldsymbol{E_r}$
\vspace{0.1cm}
\end{tabular}
 & $A_2\times A_1$ & $A_4$ & $D_5$   & $E_6$ & $E_7$  & $E_8$  \\ \hline 
${}^{}$ \quad 
\begin{tabular}{c}\vspace{-0.35cm}\\
$\boldsymbol{E'_{r-1}}$
\vspace{0.13cm}
\end{tabular}
 &$A_2$  & $A_3$ & $A_4$   & $A_5$ & $A_6$ & $A_7$  \\ \hline 
${}^{}$ \quad
\begin{tabular}{c}\vspace{-0.35cm}\\
$\boldsymbol{E''_{r-1}}$
\vspace{0.13cm}
\end{tabular} & $A_1\times A_1$ & $A_3$ & $D_4$   & $D_5$ & $D_6$ & $D_7$  \\ \hline 
${}^{}$ \quad 
\begin{tabular}{c}\vspace{-0.35cm}\\
$\boldsymbol{W_r=W(E_{r})}$
\vspace{0.1cm}
\end{tabular}
& 
$\mathfrak S_3\times \mathfrak S_2$
 & $\mathfrak S_5$ & $\big(\mathbf Z/2\mathbf Z)^4
 \ltimes  \mathfrak S_5  
$   & $W(E_6)$ & $W(E_7)$ & $W(E_8)$  \\ \hline 
${}^{}$ \quad  
\begin{tabular}{c}\vspace{-0.35cm}\\
$\boldsymbol{\omega_r=\lvert W_r\lvert}$
\vspace{0.1cm}
\end{tabular}
& $12$ & $5!$ & $2^4\cdot
5!
$   & $2^7\cdot 3^4\cdot 5$ & $2^{10}\cdot 3^4 \cdot 5\cdot7$ & 
$2^{14}\cdot 3^5 \cdot 5^2\cdot7$
  \\ \hline 
${}^{}$ \quad 
\begin{tabular}{c}\vspace{-0.35cm}\\
$\boldsymbol{l_r={\lvert \mathcal L_r \lvert}}$
\vspace{0.1cm}
\end{tabular}
&6 & 10 &  16  & 27 & 56 & 240  \\ \hline
${}^{}$ \quad 
\begin{tabular}{c}\vspace{-0.35cm}\\
$\boldsymbol{\kappa_r={\lvert \mathcal K_r \lvert}}$
\vspace{0.13cm}
\end{tabular}
& 3 & 5 &    10 & 27 & 126 & 2160 \\
\hline
\end{tabular}}
\bk 
\caption{}
\end{table}

\begin{rem}
Seen as points in ${\bf Pic}(X_r)$, the elements of $\mathcal L_r$ are the vertices of a  semi-regular polytope, the so-called `Gosset polytope' $(r-4)_{21}$ ({\it cf.}\,\cite{Lee1}{\rm )}.  The latter enjoys several nice combinatorial and geometric properties which allow to interpret some of the objects discussed above in very concrete ways.  For instance $W_r$ is a subgroup of finite index (2) of the group of Euclidean automorphisms of 
$(r-4)_{21}$. One can also relate the conical fibrations $\mathfrak C_c: X_r\rightarrow \mathbf P^1$ to the facets  of a special kind of $(r-4)_{21}$.\footnote{More precisely, the facets of Gosset's polytope $(r-4)_{21}$ are of two distinct kinds: some are hypersimplices $\alpha_{r-1}$, the others are crosspolytopes $\beta_{r-1}$ (also called `half-measure polytopes' and noted by $h\gamma_{r-1}$). 
The former facets are associated with morphisms $X_r\rightarrow \mathbf P^2$ corresponding to the blow-up of $r$ points in general position in the projective plane, the latter are in 1-1 correspondance with the conical fibrations on $X_r$. See \cite{Lee1} for more details.}  We do believe that the fact that identity $\boldsymbol{\big({\bf H Log}(X_r)\big) }$ holds true is linked to some deep combinatorial and geometric properties of  $(r-4)_{21}$.  But since this aspect of things does not intervene in the proof we give of our main result (namely Theorem \ref{Thm:main}), we will not elaborate further on this in the sequel, the unique exception being the very allusive subsection \S\ref{SS:A conceptual interpretation?} at the very end.
\end{rem}

\subsubsection{The case $r=7$} 
\label{SSS:r=7}
By way of example, we discuss here in a very concrete way  the case $r=7$ (having in mind that all the other ones  can be treated in a completely similar way).
\sk 

Since ${\rm dP}_2=X_7$ is the total space of the blow-up $b=b_7 : X_7\rightarrow \mathbf P^1$ of the projective plane in $r=7$ points $p_1,\ldots,p_7$ in general position, our approach here (which actually is a very classical one going back  at least to Coxeter) will be given,  by means of the  push-forward by the birational morphism $b$, 
an incarnation as concrete as possible on $\mathbf P^2$ of the objects (lines and fibrations by conics on $X_r$) which are relevant to  the study of the web $\boldsymbol{\mathcal W}_{X_7}$.\sk

Let us first describe, in terms of the seven points $p_i$'s,  the images of the lines in $X_7$ by the map $b$.  There are four distinct possibilities for what can be the image $\ell'=b(\ell)$
by $b$ of a line $\ell\subset X_7$, : 
\begin{itemize}
\item[$-$] $\ell'$ is 0-dimensional if and only if it is one of the $p_i$'s. This occurs if and only if $\ell$ is one of the seven exceptional divisors $\ell_1,\ldots,\ell_7$ of $b$.  
More accurately,  what corresponds  scheme-theoretically to $\ell_i$ in the projective plane is the `first infinitesimal neighbourhood of $p_i$', that is the projectified normal bundle of $p_i$ in $\mathbf P^2$, denoted by $p_i^{(1)}\simeq  \mathbf P^1$;
\sk
\item[$-$] if $\ell$ is not one of the $\ell_i$'s, then it is an irreducible rational curve in $\mathbf P^2$ of degree $\delta'= 1,2,3$: \sk
\begin{itemize}
\vspace{-0.4cm}
\item[$\bullet$] $\delta'
 =1$ if and only if $\ell'$ is a line joining two points among the seven base points; \sk
\item[$\bullet$] $\delta'
 =2$ if and only if $\ell'$ is one of the conics passing through five of the $p_i$'s; and  \sk
\item[$\bullet$] $\delta'
 =3$ if and only if $\ell'$ is a cubic passing through  all the seven base points and with a node at one of them.   
\end{itemize} 
\end{itemize}
\sk

The above description in $\mathbf P^2$ of the lines on $X_7$ is summarized in Table  
\ref{Ff:r=7} below, where we use the following notations: $i$ and $j$ stand for two distinct elements of 
$\{1,\ldots,7\}$, $I$ is any subset of $\{1,\ldots,7\}$ of cardinality 5, for any such $I$ we set $\ell_I=\sum_{i\in I} \ell_i$ and $\boldsymbol{\ell}$ stands for $\sum_{j=1}^7 \ell_j$. 
With these notations at hand, we set $\ell_{ij}$ for the line $\langle p_i,p_j\rangle \subset \mathbf P^2$, $\mathcal C_{ij}$ stands for the conic passing through all the base points except $p_i$ and $p_j$ and we denote by $\mathcal C^3_{i}$ the (unique) rational cubic of $\mathbf P^2$ passing through all the $p_j$'s and with a node at $p_i$.  In what follows, since it will not cause any problem, we will use the notation of the first column of Table  
\ref{Ff:r=7} to designate any object (class in the Picard group or rational curve in $\mathbf P^2$) in the corresponding rows.  For instance, depending on the context, $\mathcal C^3_{i}$ will either stand for the cubic curve in $\mathbf P^2$ defined just above, or 
the class $3h-\boldsymbol{\ell}-\ell_i\in {\bf Pic}(X_7)$, or for the unique rational curve (`line') in $X_r$ belonging to the associated linear system $\lvert 3h-\boldsymbol{\ell}-\ell_i\lvert$. 
\begin{table}[!h]
 \begin{tabular}{|c|l|c|c|}
\hline
{\bf Line} 
&  {\bf Class in} $ \boldsymbol{{\bf Pic}(X_7)}$ 
 & {\bf Number of such lines} & 
\begin{tabular}{c}\vspace{-0.35cm}\\
{\bf Model  in $\boldsymbol{\mathbf P^2}$}
\vspace{0.13cm}
\end{tabular}     \\ \hline \hline
\begin{tabular}{c}\vspace{-0.35cm}\\
$\ell_{i}$
\vspace{0.13cm}
\end{tabular} & ${}^{}$ \,  $\ell_i$  & 7 & first infinitesimal neighbourhood  $p_i^{(1)}$
 \\ \hline 
\begin{tabular}{c}\vspace{-0.35cm}\\
$\ell_{ij}$
\vspace{0.13cm}
\end{tabular} & ${}^{}$ \,  $h-\ell_i-\ell_j$  & 21 & line joining  $p_i$ to $p_j$ 
 \\ \hline 
\begin{tabular}{c}\vspace{-0.35cm}\\
$\mathcal C_{ij}$
\vspace{0.13cm}
\end{tabular} & ${}^{}$ \,  $2h-\boldsymbol{\ell}+\ell_i+\ell_j$  & 21 &
\begin{tabular}{c}\vspace{-0.35cm}\\
 conic through  the $p_k$'s, $k\not \in \{ i,j\}$
\vspace{0.13cm}
\end{tabular}
 \\ \hline 
\begin{tabular}{c}\vspace{-0.35cm}\\
$\mathcal C^3_{i}$
\vspace{0.13cm}
\end{tabular} & ${}^{}$ \,  $3h-\boldsymbol{\ell}-\ell_i$  &  7  &
\begin{tabular}{c}\vspace{-0.35cm}\\
 cubic through all  the $p_l$'s with a node at $p_i$
\vspace{0.13cm}
\end{tabular}
 \\ \hline 
\end{tabular} \mk 
\caption{Lines on ${\rm dP}_2$ and the corresponding `curves' in the projective plane}
\label{Ff:r=7}
\vspace{-0.6cm}
\end{table}

Now we discuss  the conic fibrations on $X_7$ 
and their singular fibers.  For any such fibration $\mathfrak C_{\mathfrak c} : X_7\rightarrow \mathbf P^1$ (with  $\mathfrak c\in \mathcal K_7$),
we use the following notations: 
\begin{itemize}
\item[$\bullet$] one denotes by 
$\lvert \mathfrak C_{\mathfrak c} \lvert $ the linear pencil of rational curves on $\mathbf P^2$ associated with 
the dominant rational map  $\mathfrak C_{\mathfrak c} \circ b^{-1}:\mathbf P^2\dashrightarrow \mathbf P^1$; and 
\sk
\item[$\bullet$] $\mathfrak C_{\mathfrak c}^{\rm red}$ stands for the set of singular fibers of $ \mathfrak C_{\mathfrak c}$ (one has  $\lvert \mathfrak C_{\mathfrak c}^{\rm red}\lvert =r-1$, {\it cf.}\,\S\ref{SSS:Lines-Conics-properties}.{\bf (o)}).
\sk
\end{itemize}
The conic classes are of five distinct types relatively to the blow-up $b: X_7\rightarrow \mathbf P^2$, depending on their intersection number with the class $h$ (which is the one of the preimage $b^{-1}(l)$ of a generic line $l$ of $\mathbf P^2$). This quantity can take any of the values from 1 to 5 (both included) and can also be characterized as the degree of any irreducible member of the linear system $\lvert \mathfrak C_{\mathfrak c} \lvert $. \sk

Explicit descriptions of the conical classes $\mathfrak c\in \mathcal K_7 $, of the linear systems $\lvert \mathfrak C_{\mathfrak c} \lvert $, of the non irreducible fibers in the latter 
 are given in Table \ref{Table:Red-Fibers-dP2} below, where  the following notations are used:  $i,j$ and $k$ stand for pairwise  distinct elements of $[7]$;  $I$ and $J$ denote subsets of $ [7]$ of cardinal 4 and 3 respectively;  $\{i_1,i_2,i_3,i_4\}$ denotes an arbitrary labelling of the elements of $I$ and similarly
$\{j_1,j_2,j_3\}$ and $\{k_1,k_2,k_3,k_4\}$ stand for the same but for  the sets $J$ and $K=J^{\rm c}=[7]\setminus J$ respectively; and as before, $\boldsymbol{\ell}=\sum_{i=1}^7 \ell_i$ 
 is the formal sum of the exceptional divisor  of the blow up $b: X_7\rightarrow \mathbf P^2$. 
\begin{table}[!h]
 \begin{tabular}{|l|c|c|c|}
\hline
${}^{}$ \hspace{0.25cm} {\bf Conic class} $ \boldsymbol{\mathfrak c}$ 
&  {\bf Number of such} $ \boldsymbol{\mathfrak c}$ 
 &
${}^{}$ \hspace{0.5cm}  
\begin{tabular}{c}\vspace{-0.35cm}\\
{\bf Linear system $\lvert \boldsymbol{\mathfrak C_{\mathfrak c}} \lvert $ }
\vspace{0.13cm}
\end{tabular} &  
$\boldsymbol{\mathfrak C_{\mathfrak c}^{\rm red}}$  \\ \hline \hline
${}^{}$ \hspace{0.25cm} \begin{tabular}{c}\vspace{-0.35cm}\\
$h-\ell_i$\vspace{0.13cm}
\end{tabular}
& 7 &  
lines through $p_i$ &  
$\ell_{ij}+\ell_j$ 
  \\ \hline
 ${}^{}$ \, \begin{tabular}{c}\vspace{-0.35cm}\\
$2h-\sum_{i\in I}\ell_i$\vspace{0.13cm}
\end{tabular}
& 35 &  
${}^{}$ \hspace{-0.2cm}  
\begin{tabular}{l}\vspace{-0.35cm}\\
 conics through 
  the $p_i$'s, $i\in I$ 
\end{tabular}&  
${}^{}$ \, \begin{tabular}{c}\vspace{-0.35cm}\\
$\ell_{i_1i_2}+\ell_{i_3i_4}$ \vspace{0.1cm}   \\
$
\ell_{i_3} + \mathcal C_{i_1i_2}
$
 \vspace{0.13cm}
\end{tabular}
    \\ \hline
 ${}^{}$ \,  \begin{tabular}{c}\vspace{-0.35cm}\\
$3h-\boldsymbol{\ell}+\ell_i-\ell_j$\vspace{0.13cm}
\end{tabular}
 &  42 & 
 ${}^{}$ \hspace{-0.2cm}   
\begin{tabular}{l}\vspace{-0.35cm}\\
 cubics through the $p_k$'s for \\
$k\neq i$, with a node at $p_j$ 
 \vspace{0.13cm}
\end{tabular}&  
${}^{}$ \, \begin{tabular}{c}\vspace{-0.35cm}\\
$\ell_{jk} + \mathcal C_{{ik}}$  \vspace{0.1cm}   \\
 $\ell_i + \mathcal C_{j}^3 $
 \vspace{0.13cm}
\end{tabular}
   \\ \hline
${}^{}$ \,    \begin{tabular}{c}\vspace{-0.35cm}\\
$4h-\boldsymbol{\ell}-\sum_{j\in J} \ell_j$\vspace{0.13cm}
\end{tabular}
 & 35   &
 ${}^{}$ \hspace{-0.2cm}  
\begin{tabular}{l}\vspace{-0.35cm}\\
 quartics through the $p_k$'s \\
with a node at $p_j$ for $j\in J$
 \vspace{0.13cm}
\end{tabular}&   
\begin{tabular}{c}\vspace{-0.35cm}\\
$\mathcal C_{k_1k_2} + \mathcal C_{k_3k_4}$  \vspace{0.1cm}   \\
 $\ell_{j_1j_2} + \mathcal C_{j_3}^3 $
 \vspace{0.13cm}
\end{tabular}
   \\ \hline
${}^{}$ \,    \begin{tabular}{c}\vspace{-0.35cm}\\
$5h-2\boldsymbol{\ell}+\ell_i $\vspace{0.13cm}
\end{tabular}
 & 7   &
 ${}^{}$ \hspace{-0.2cm}  
\begin{tabular}{l}\vspace{-0.35cm}\\
 quintics through the $p_k$'s with \\
a node at $p_k$ except for $k=i$
 \vspace{0.13cm}
\end{tabular}&   
$\mathcal C_{ij}+\mathcal C^3_j$
   \\ \hline
\end{tabular} \bk    \vspace{0.1cm}
\caption{Conic classes on ${\rm dP}_2$ and their reducible fibers}
\label{Table:Red-Fibers-dP2}
\end{table}

\section{\bf Main section: proofs}
\label{S:Main}
The section is devoted to proving Theorem  \ref{Thm:main}. 
We give two proofs of it : the first relies on elementary (but heavy) computations of linear algebra whereas the second (albeit being computational as well, but quite less than the first proof), is representation-theoretic and uses the structure of $W_r$-module of $\mathbf C^{{\mathcal L}_r}$. 
The case when $r=3$ (equivalently  $d=6$) is special: the corresponding Weyl group is the direct product $\mathfrak S_3\times \mathfrak S_2$ and above all, it 
can be treated by hand very easily. For these reasons,  it will be left aside in what follows.
\begin{center}
\vspace{-0.3cm}$\star$
\end{center}

We continue to use the notations introduced above, recalling the most important ones for the sake of clarity: $r\in \{4,\ldots,7\}$ and $d\in  \{2,\ldots,5\}$ are integers such that  $d=9-r$. We denote by  ${\rm dP}_d=X_r={\bf Bl}_{P_r}(\mathbf P^2)$ the blow-up of the projective plane in $r$ points in general position. We label  by $U_k: X_r\rightarrow \mathbf P^1$ the fibrations by conics on $X_r$, with $k=1,\ldots,\kappa_r$. Then 
$$ \boldsymbol{\mathcal W}_{{\rm dP}_d}= \boldsymbol{\mathcal W}_{X_r}=\boldsymbol{\mathcal W}( U_1,\ldots,U_{\kappa_r})$$
is a $\kappa_r$-web by conics on $X_r$, which is regular on 
$Y_r=X_r\setminus L_r$,  where  $L_r$ stands for the union of the lines contained in $X_r$.

For any $k=1,\ldots,\kappa_r$, let $L_{r,k}$ stand for the set of lines in $X_r$ contracted (onto a points) by $U_k$. Then 
 $\mathfrak R_k=U_k(L_{r,k}) $ is a finite subset of $\mathbf P^1$ with $r-1$ elements denoted by $\rho_k^{1},\ldots,\rho_k^{r-1}$.  One assume that $U_k$ 
has been chosen such that one of the $\rho_k^t$'s, say $\rho_k^{r-1}$, coincides with the point at infinity $\infty\in \mathbf P^1$.  Then the logarithmic 1-forms $\eta_k ^t= dz/(z-\rho_k^t)$ for $t=1,\ldots,r-2$ form a basis of the 
space ${\bf H}_k={\bf H}^0\big( \mathbf P^1, \Omega^1_{\mathbf P^1}({\rm Log}\,\mathfrak R_k )\big)$.

We apply the material of \S\ref{SS:Ui:X->Ci}
and \S\ref{SS:Ui:X->Ci-2} 
to the iterated integrals of weight $r-2$ on $\mathbf P^1\setminus \mathfrak R_k $ whose symbols are the antisymmetrization of $ \eta_k ^1\otimes \cdots \otimes \eta_k ^{r-2}$ for $k=1,\ldots,\kappa_r$. For any such $k$, 
$$\Omega_k= {\rm Asym}^{r-2}\Big(  \,
 d\,{\rm Log}\,\big(\,U_i-\rho_k^{1}\,\big) \cdots \otimes  d\,{\rm Log}\,\big(\,U_i-\rho_k^{r-2}
 \,\big)\,\Big) $$ belongs to 
 $U_k^*({\bf H}_k\big)$, 
 the latter being naturally a vector subspace of $\boldsymbol{\mathcal H}_{X_r}={\bf H}^0\Big(X_r, \Omega^1_{X_r}\big({\rm Log}\, L_r\big)\Big)$.
 Considering Proposition
\ref{Prop:Symbolic},  we aim to establish that there exists $(\epsilon_k)_{k=1}^{\kappa_r}\in \{\pm 1\}^{\kappa_r}$, unique up to sign,  such that $\sum_{k=1}^{\kappa_r} \epsilon_k\,\Omega_k=0$ in $\wedge ^{r-2} \boldsymbol{\mathcal H}_{X_r}$.  At this stage, we use the fact that  $X_r$ does not carry any non trivial holomorphic 1-form.  Denoting by 
${\rm Res}_\ell$ the residue map along $\ell$ for any line $\ell\in \boldsymbol{\mathcal L}_r$, this implies that $$\oplus_{\ell\in \boldsymbol{\mathcal L}_r} 
{\rm Res}_\ell \, 
: \, \boldsymbol{\mathcal H}_{X_r} \longrightarrow \mathbf C^{ \boldsymbol{\mathcal L}_r}$$ is injective, which gives rise to an injective  $\mathbf C$-linear morphism $\wedge ^{r-2} \boldsymbol{\mathcal H}_{X_r}\longrightarrow  \wedge ^{r-2}  \mathbf C^{ \boldsymbol{\mathcal L}_r}$.  Denoting now by $\omega_k$ the image of $\Omega_k$ by the latter map, we want to prove that there exists $(\epsilon_k)_{k=1}^{\kappa_r}\in \{\pm 1\}^{\kappa_r}$, unique up to sign, such that the following relation holds true in  the wedge space $\wedge ^{r-2}  \mathbf C^{ \boldsymbol{\mathcal L}_r}$:
$$
\sum_{k=1}^{\kappa_r} \epsilon_k\,\omega_k=0
 \, .$$

\subsection{An elementary but computational proof} 
\label{SS:Elementary-Proof}
For any $k=1,\ldots,\kappa_r$, let $\mathfrak C_k^{\rm red}$ be the set of reducible conics among the fibers of the fibrations by conics $U_k : X_r\rightarrow \mathbf P^1$.  It is not difficult to give an explicit description of these sets in each case (for instance, see  the subsection \S\ref{SSS:r=7} just above for the case $r=7$).

Once for all, for each $k$,  we fix a labelling 
 $\tilde C_{k}^1,\ldots,\tilde C_{k}^{r-1}$ of the elements of 
$\mathfrak C^{red}_{k}$ and we set $C_{k}^s=\tilde C_{k}^s-\tilde C_{k}^{r-1}$ for $s=1,\ldots,r-2$. Then for any such $s$, one can write 
$$C_{k}^s=\ell_{k,1}^s+\ell_{k,2}^s-l_{k}-\hat l_{k}$$ 
for some lines $\ell_{k,1}^s,\ell_{k,2}^s, l_{k},\hat l_{k}\in \mathcal L_{r}$ (with $\tilde C_{k}^s=\ell_{k,1}^s+\ell_{k,2}^s$ and 
$\tilde C_{k}^{r-1}=l_{k}+\hat l_{k}$).  Consequently, to each reducible conic $C_{k}^s$ of $\mathfrak C_s$ corresponds a vector element of $\mathbf Z^{\boldsymbol{\mathcal L}_r}$, with only four non zero components, two with the value $+1$, the others two being equal to $-1$.  
The divisors $C_{k}^1,\ldots, C_{k}^{r-2}$ span a free $\mathbf Z$-submodule of 
$\mathbf Z^{\boldsymbol{\mathcal L}_r}$, of rank $r-2$, which 
we denote by $\mathcal U_{k}$.  Hence the wedge product 
$$\omega_{k}=C_{k}^1\wedge \cdots \wedge C_{k}^{r-2}$$ 
is a basis of the free $\mathbf Z$-module $\wedge^{r-2} \mathcal U_{k}\simeq \mathbf Z$.  The latter is naturally embedded in $\wedge^{r-2} \mathbf Z^{\boldsymbol{\mathcal L}_r}$ which is $\mathbf Z$-free of rank ${ l_r \choose r-2}$.  In order to perform explicit computations (using a computer algebra system), we label in an arbitrary but fixed way the elements of $\mathcal L_r$. This allows to identify $\mathbf Z^{\boldsymbol{\mathcal L}_r}$ with $\mathbf Z^{l_r}$, which naturally gives rise to a  $\mathbf Z$-linear isomorphism 
$$\tau^{\wedge(r-2)}  : \wedge^{r-2} \mathbf Z^{\boldsymbol{\mathcal L}_r}
\rightarrow 
\mathbf Z^{{ l_r \choose r-2}} \, .$$

As a basis for the $\mathbf Z$-module $\mathbf Z^{{ l_r \choose r-2}}$, we take the set 
of wedges $\wedge^{r-2}_I e=e_{i_1}\wedge \cdots \wedge e_{i_r}$ indexed by the set, abusively denoted by 
${ l_r \choose r-2}$ as well,  of 
tuples $I=(i_1,\ldots,i_{r-2})$ 
such that $1\leq i_1<\cdots <i_{r-1}\leq l_r$. 

For any $k=1,\ldots,\kappa_r$ and any $s=1,\ldots,r-2$, let $c^s_{k}$ be the $l_r$-tuple belonging to 
 $ \mathbf Z^{l_r}$ corresponding to $C_{k}^s$ via the chosen identification $\mathbf Z^{\boldsymbol{\mathcal L}_r}\simeq \mathbf Z^{l_r}$, and let $M_{k}\in {\rm M}_{(r-2)\times l_r}(\mathbf Z)$ be the integer matrix whose rows are the $c^s_{k}$ for $s=1,\ldots,r-2$. 
Then the vector $\varpi_{k}=\tau^{\wedge(r-2)}(\omega_{k}\big) \in \mathbf Z^{{ l_r \choose r-2}}$
  can easily seen to be the 
 one the $I$-component of which is equal to the $(r-2)\times I$-minor determinant of the matrix 
 $M_{k}$, this for any $(r-2)$-tuple $I=(i_1,\ldots,i_{r-2})$ as above.  Using a computer algebra system,\footnote{We have implemented the computational approach described here whithin Maple. The Maple worksheets we used for performing the computations are available upon request.} it is then not hard to effectively build the vectors 
 $\varpi_{k}\in \mathbf Z^{{ l_r \choose r-2}}$ for all $k=1,\ldots,\kappa_r$ and to verify that the vector space of 
 $(\epsilon_k)_{k=1}^{\kappa_r} \in \mathbf C^{ \kappa_r}$ such that $\sum_{k=1}^{\kappa_r} \epsilon_k\, \varpi_{k}=0$  has dimension 1 and is spanned by an element such that $\epsilon_k=\pm 1$ for any $k=1,\ldots, \kappa_r$.  
 Combined with Proposition \ref{Prop:Symbolic}, 
 this proves Theorem \ref{Thm:main}. 

\begin{rem}
The formal computations described above serve to verify the existence (and the unicity) of a certain hyperlogarithmic abelian relation for $\boldsymbol{\mathcal W}_{X_r}$, which is a regular web on $Y_r=X_r\setminus L_r$. Since the blow-up $b_r: X_r\rightarrow \mathbf P^2$ induces an isomophism between $Y_r$ and 
the complement of $b_r(L_r)$ in  the projective plane, one should be able to perform similar computations 
leading to the same result for the push-forward web $(b_r)_*\big( 
\boldsymbol{\mathcal W}_{X_r}
\big)$, that is some computations involving data coming from the plane $\mathbf P^2$ and the 
images by $b_r$ of the lines in $X_r$. But since $\ell_1,\ldots,\ell_r$  are contracted, one can expect that they do not play any role regarding the abelian relation under scrutiny and that only the  lines elements of $\ell\in {\mathcal L}_r\setminus \{\ell_i\}_{i=1}^r$,  which precisely are those which are not  contracted by $b_r$,  are meaningful for that regard.
 It turns out that it is the case indeed: 
 actually all the computations above can be done starting not with $\mathbf Z^{{\mathcal L}_r}$ but with 
its quotient by the rank $r$ submodule spanned by the $\ell_i$'s. This has the advantage of lightening the calculations (a little). 
 \end{rem}

%
%

\subsection{A representation-theoretic proof}
\label{S:Representation-theoretic-proof}
The proof  in the previous subsection is elementary, which is a nice feature of it. 
A less pleasant counterpart to this is that it does not tell much on what is going on.  Here we discuss another approach for proving Theorem \ref{Thm:main} which is more conceptual even if it relies  on explicit computations at some points as well.\sk

  The main new ingredient here is the natural action of  the Weyl group $W_r$ on the set of lines ${\mathcal L}_r$ contained in $X_r$. The key result we will use is the explicit description  of $\mathbf C^{{\mathcal L}_r}$ as a $W_r$-representation, which is given by the following: 
\begin{prop}
\label{Prop:Key}
 For $r=4,5,6,7,8$, one has the following decompositions of\, $\mathbf C^{{\mathcal L}_r}$ in irreducible $W_r$-modules: 
\begin{align}
\label{Eq:Decomposition-Wr}
\mathbf C^{{\mathcal L}_4}=& \, {\bf 1}\oplus V_{[41]}^4\oplus V_{[32]}^5 \nonumber \\
\mathbf C^{{\mathcal L}_5}=&  \, 
{\bf 1}\oplus V_{[4,1]}^5\oplus V_{[3,2]}^{10} 
 \nonumber   \\
\mathbf C^{{\mathcal L}_6}=&  
\, {\bf 1}\oplus V^{6,1}\oplus V^{20,2} 
\\
\mathbf C^{{\mathcal L}_7}=&   \, 
{\bf 1} \oplus V^{7,1}\oplus V^{21,3}\oplus V^{27,2}
%
%
\nonumber 
\\
\mathbf C^{{\mathcal L}_8}=&   \,  
{\bf 1} \oplus V^{8,1} 
\oplus V^{35,2} 
\oplus V^{84,4} 
\oplus V^{112,3} \, .
\mk
 \nonumber 
 \sk
\end{align}
\end{prop}
\sk

\begin{rem}
{\bf 1.} We use the following notations for the irreducible $W_r$-modules appearing in the above decompositions: first, ${\bf 1}$ stands for the trivial representation whereas 
all the other irreducibles are labeled with the same label as the one used  in GAP3 for the corresponding characters. 
For the classical types $A_4$ and $D_5$, these labels are the classical ones, namely are given by means 
 of partitions of  $4$  and of bipartitions of  $5$ respectively (for $D_5$, see also Table \ref{Table:CharTableWD5} and the paragraph just below in the Appendix). 
 The superscripts appearing there are the degree of the corresponding representations. For instance, the component $V_{[3,2]}^{10}$ appearing in the decomposition of $\mathbf C^{{\mathcal L}_5}$ stands for the irreducible $W(D_5)$-module associated to the bipartition $(3,2)$ of 5, which moreover is of degree 10. \sk

 For the exceptional types $E_6,E_7$ and $E_8$, the labels we use are pairs of integers $(d,e)$ with 
 $d$ standing for the degree of the considered $W(E_r)$-module and $e$ denoting its `$b$-invariant', that is  the smallest integer $m$ such that the module under scrutiny is contained in the $m$-th symmetric power 
 of the corresponding reflection representation, which  has label $(r,1)$ accordingly. 
\mk \\
\noindent {\bf 2.} The decomposition 
$\mathbf C^{{\mathcal L}_4}= {\bf 1}\oplus V_{[41]}\oplus V_{[32]}$ 
 can be found  a few lines below the exact sequence {\rm (2.1)} 
in  \cite{DolgachevFarbLooijenga}. 
There,  it is said that this can be obtained {\it `by looking at
pairs of complementary pentagons inside of the Petersen graph'}. 
It would be interesting 
to make this approach clearer/more explicit and to investigate whether it can be generalized to the lower degree del Pezzo surfaces ${\rm dP}_d$ for any $d$ ranging from $5$ to 1 or not. \mk \\
\noindent {\bf 3.} For $r\leq 7$, the set of lines ${\mathcal L}_r$ identifies in a natural way with the set of weights of  a minuscule representation of the Lie algebra of type $E_r$.\footnote{This generalizes as follows for $r=8$:  $\mathcal L_8$ identifies with the set of non zero weights of the quasi-minuscule representation of the Lie algebra $\mathfrak e_8$, which is the adjoint representation in this case.} Hence it is quite natural to wonder what the  
decomposition of $\mathbf C^{{\mathcal L}_r}$ into irreducible $W_r$-modules is.  However, we have not been able to find an answer to this question in the existing literature (except in the $r=4$ case as mentioned just above). 
\end{rem}

Our proof  of Proposition \ref{Prop:Key} goes by elementary explicit computations of the character theory of the corresponding Weyl group.  We describe succinctly our approach just below and refer to the Appendix for more details (in particular, the case when $r=5$ is fully detailed there).

\begin{proof}[Proof (sketched)] 
Given $r$ as in the statement of the proposition, we fix a labelling of the lines in $X_r$ which gives  rise to a  basis $\mathcal B_r$ of $\mathbf C^{{\mathcal L}_r}$ hence to a linear identification $\mathbf C^{{\mathcal L}_r}\simeq \mathbf C^{l_r}$. 
For any fundamental root $\rho_i$ 
 of the corresponding root system $\mathcal R_r$ (see \eqref{Eq:Fundamental-Roots}), 
 it is straightforward to compute explicitly  the matrix 
  of the reflection $s_i=s_{\rho_i}$ ({\it cf.}\,\eqref{Eq:s-rho}) in the basis 
  $\mathcal B_r$, denoted by 
  \begin{equation}
  \label{Eq:Mi}
M_{i}={\rm Mat}_{\mathcal B_r}\big(s_{i}\big) \in {\rm GL}_{l_r}\big(\mathbf Z\big)\, .
\end{equation}

Let $c_r$ stand for the number of conjugacy classes in $W_r$.  Next we consider a family of tuples $\boldsymbol{i}_k=(i_{k,1},\ldots,i_{k,m_k})\in \{1,\ldots,r\}^{m_k}$ for $k=1,\ldots,c_r$ and some positive $m_k$'s such that the group elements $s_{\boldsymbol{i}_k}=s_{i_{1,k}}\cdots s_{i_{k,m_k}}$ form a set of representatives of all the conjugacy classes of 
$W_r$. 
 We denote by  $\chi_r$   the character of the $W_r$-representation $\mathbf C^{{\mathcal L}_r}$. 
For any $k=1,\ldots , c_r$, one has 
 $$
 \chi_r\big( s_{\boldsymbol{i}_k}\big)={\rm Trace}\Big( M_{i_{1,k}}\cdots M_{i_{k,m_k}}
 \Big) \, .
 $$
Since the $r$ matrices $M_i$'s are known, one can make $\chi_r$ completely explicit. 
Using a character table of $W_r$, it is then straightforward to determine the decomposition of $\chi_r$ as a sum of  irreducible 
$W_r$-characters. 
\end{proof}

\subsection{} 
\label{SS-Key-Clef}
We now explain how to deduce our main theorem from Proposition \ref{Prop:Key}.
We use similar notations as those introduced in \S\ref{SS:Elementary-Proof} but now work over $\mathbf C$ instead of $\mathbf Z$. \mk 

Given a conic class $\mathfrak c\in \mathcal K_r$, here we denote by $\mathfrak c^{red}$ the set of $r-1$ non irreducible conics in the associated linear system $\lvert \mathfrak c\lvert \subset {\bf H}^0\big(  X_r, 
-2K_r\big)$. Each $C\in \mathfrak c^{red}$ can be written $C=l_C+\ell_C$ for two lines  $l_C,\ell_C \in \mathcal L_r$. 
Then we can define an injection $\mathfrak c^{red}\hookrightarrow \mathbf C^{{\mathcal L}_r}$ which in turn gives rise to a linear embedding $$i_{\mathfrak c}:  \mathbf C^{\mathfrak c^{red}}\hookrightarrow \mathbf C^{{\mathcal L}_r}\, .$$ The stabilizer $F_{\mathfrak c}$ of $\mathfrak c $  in $W_r$ acts by permutations on $\mathfrak c^{red}$ from which one gets a $F_{\mathfrak c}$-action on $\mathbf C^{\mathfrak c^{red}}$. Considering the action of the latter group on $\mathcal L_r$ induced by restriction of that of $W_r$, one obtains that $i_{\mathfrak c}$ is a morphism of 
$F_{\mathfrak c}$-representations.  Taking the $(r-2)$-th wedge product, we get 
another morphism of $F_{\mathfrak c}$-representations which is easily seen to be injective as well: 
$$i_{\mathfrak c}^{r-2} :  \wedge ^{r-2} \mathbf C^{\mathfrak c^{red}}\hookrightarrow  \wedge^{r-2} \mathbf C^{{\mathcal L}_r}\, .$$

Since $F_{\mathfrak c}$ acts by permutations on $\mathfrak c^{red}$, the element $1_{\mathfrak c}=\sum_{ C\in \mathfrak c^{red}}C\in \mathbf C^{\mathfrak c^{red}}$ is $F_{\mathfrak c}$-invariant  and its supplementary, noted by $\mathcal U_{\mathfrak c}$,  is the one spanned (over $\mathbf C$) by the elements 
$C-C'\in  \mathbf C^{\mathfrak c^{red}}$ for all $C,C'\in {\mathfrak c^{red}}$.  We denote the same the restrictions  
of $i_{\mathfrak c}
$  and $i_{\mathfrak c}^{r-2} $ to 
 $\mathcal U_{\mathfrak c}$ and $\wedge ^{r-2} \mathcal U_{\mathfrak c}$
 respectively. Likewise,  ${\bf sign}$ again stands for the restriction of the signature morphism ${\bf sign}: W_r\rightarrow \{ \pm 1\}$  to $F_{\mathfrak c}$. 

\begin{lem}
\label{L:Fc-module}
 {\rm 1.} As a group,  $F_{\mathfrak c}$ is isomorphic to $W(D_{r-1})$.

{\rm 2.} The decomposition $ \mathbf C^{\mathfrak c^{red}}={\bf 1}\oplus \mathcal U_{\mathfrak c}$ with 
${\bf 1}=\langle\, 1_{\mathfrak c}\rangle$ 
actually is a direct sum of $F_{\mathfrak c}$-modules. 

{\rm 3.} Up to the isomorphism in {\rm 1.}, one has 
$ \mathcal U_{\mathfrak c} \simeq V_{[.(r-2)1]}$  and 
$\wedge ^{r-2} \mathcal U_{\mathfrak c} \simeq {\bf sign}$
as representations. 
\end{lem}
\begin{proof} 
 Since $W_r$ acts transitively on $\mathcal K_r$, it suffices to prove the lemma when $\mathfrak c=h-\ell_1$. 
 In this case, one verifies that $F_{\mathfrak c}$ is isomorphic to the subroup of $W_r$ generated by 
 the $s_i$'s for $i=2,\ldots,r$. Thus $F_{\mathfrak c}=W''_{r-1}=W(D_{r-1})$ (with the convention that $D_3=A_3$). 

The second point follows at once from the fact that $W''_{r-1}$ acts by permutations on $\mathfrak c^{red}$. 

To prove 3, we notice that the elements of $\mathfrak c^{red}$ are the conics $C_j=[h-l_1-l_j]+[l_j]$'s for $j=2,\ldots,r$, where each element between brackets is an element  of ${\mathcal L}_r$ and the notation means that $C_j$ has the two  lines $h-l_1-l_j$ and $l_j$ (viewed here as curves contained in $X_r$) as irreducible components. Direct computations give that as permutations of $\mathfrak c^{red}=\{ C_2,\ldots,C_r\}$: \\
${}^{}$ \quad  $-$ $s_j$ is the transposition exchanging $C_j$ and $C_{j+1}$ for $j=2,\ldots,r-1$; \\
${}^{}$ \quad  $-$ $s_r$  is the transposition exchanging $C_2$ and $C_{3}$. 

By elementary character theory of $W(D_{r-1})$, one gets $ \mathbf C^{\mathfrak c^{red}}={\bf 1}\oplus V_{[.(r-2)1]}$ where  $ V_{[.(r-2)1]}$ is the representation (of degree $r-2$) associated to the bipartition $(-, (r-2)1)$ where 
$-$ stands for the empty partition and $(r-2)1$ for the partition $r-1=(r-2)+1$.  
Finally, setting  $$\varpi_{\mathfrak c}=\Big( \big(C_2-C_r\big)\wedge \cdots \wedge  \big(C_{r-1}-C_r\big) \Big)  \in 
 \wedge ^{r-2} \mathcal U_{\mathfrak c}\, ,$$
  it follows from the above that 
$s_2\cdot \varpi_{\mathfrak c}=-\varpi_{\mathfrak c}$. 
Since $\varpi_{\mathfrak c}\not = 0$ and because 
the trivial and 
the signature representations  are the only two $W(D_5)$-irreps of degree 1, it follows that $\wedge ^{r-2} \mathcal U_{\mathfrak c}\simeq  {\bf sign}$. 
\end{proof}

From now on, we fix one of the conic classes that we denote ${\mathfrak c}_0$. As it is easy to check (left to the reader), all the constructions and results below do not depend on this choice (possibly only up to sign but this will not matter).  Although it is a bit formal, it will be convenient to use the following notation below: for any $\mathfrak c\in {\mathcal K}_r$ and  given an element $\varpi_{\mathfrak c} \in 
\wedge ^{r-2} \mathcal U_{\mathfrak c}$ we denote by $\omega_{\mathfrak c}$ the corresponding element in 
$
\wedge ^{r-2} \mathbf C^{{\mathcal L}_r}$, that is one has $i_{\mathfrak c}^{r-2}(\varpi_{\mathfrak c})=\omega_{\mathfrak c}$. 

Our main goal here can be formalized as follows: setting $ \Psi^{r-2}=\oplus_{\mathfrak c \in {\mathcal K}_r} i_{\mathfrak c}^{r-2}$, we get a 
 linear map 
 $$ 
 \Psi^{r-2} \, :\, 
\bigoplus_{\mathfrak c \in {\mathcal K}_r}  \wedge^{r-2}  
\mathcal U_{\mathfrak c}
\longrightarrow \wedge^{r-2}\big( \mathbf C^{{\mathcal L}_r}\big)
$$
and we aim to construct a specific non trivial element 
in its kernel $\boldsymbol{\mathcal K}^{r-2}={\rm ker}( \Psi^{r-2})$, noted by 
${\bf hlog}^{r-2}$. 
For this sake, we are going to take into account the action of the Weyl group $W_r$ and establish that $ \Psi^{r-2}$  can actually be extended as a morphism of $W_r$-modules.  The  sought-after element ${\bf hlog}^{r-2}$ will be defined as the  generator (defined up to sign) of  a degree 1 sub-$W_r$-representation of $\boldsymbol{\mathcal K}^{r-2}$.

We now define a natural $W_r$-action on the domain range of  $ \Psi^{r-2}$.
The action of $W_r$ on ${\mathcal K}_r$ corresponds to a 
group morphism that we will denote by  $\kappa: W_r\rightarrow \mathfrak S_{{\mathcal K}_r}$. To simplify, for $w\in W_r$ and $\mathfrak c\in {\mathcal K}_r$ we denote by $w({\mathfrak c})$ the conic class $\kappa(w)({\mathfrak c})$.  Since $W_r\cdot {\mathfrak c}_0={\mathcal K}_r$, it follows that 
for any ${\mathfrak c}\in \mathcal K_r$, there exists $\sigma_{\hspace{-0.05cm} {\mathfrak c}_0}^
{{\mathfrak c}}\in W_r$ such that $\sigma_{\hspace{-0.05cm}{\mathfrak c}_0}^
{{\mathfrak c}}( {\mathfrak c}_0)={\mathfrak c}$.  
One takes $\sigma_{\hspace{-0.05cm} {\mathfrak c}_0}^
{{\mathfrak c}_0}=1$ but there is no canonical choice for any of the others  
$\sigma_{\hspace{-0.05cm}{\mathfrak c}_0}^
{{\mathfrak c}}$'s: one has $W_r=
\sqcup_{ \mathfrak c \in {\mathcal K}_r} 
\sigma_{\hspace{-0.05cm} {\mathfrak c}_0}^
{{\mathfrak c}}
F_{\hspace{-0.05cm}{\mathfrak c}_0}$ 
which makes clear that any of the $\sigma_{\hspace{-0.05cm} {\mathfrak c}_0}^{{\mathfrak c}}$ for ${\mathfrak c}\neq {\mathfrak c}_0$ is only defined up to right composition with an element of $F_{\hspace{-0.05cm}{\mathfrak c}_0}$. 
\sk
 
 Since  $\mathcal U_{{\mathfrak c}_0}$ carries a natural $\mathbf Z$-structure, it follows that the 1-dimensional vector space $\wedge ^{r-2} \mathcal U_{{\mathfrak c}_0}$ has a canonical generator,  well-defined up to sign. We fix and denote by $\varpi_{{\mathfrak c}_0}$ one of these two generators.  Denoting by $\bullet$ the $W_r$-action on $\wedge ^{r-2} \mathbf C^{ {\mathcal L}_r }$ for clarity, one verifies that 
for any ${\mathfrak c}\in {\mathcal K}_r$, 
$\sigma_{\hspace{-0.05cm} {\mathfrak c}_0}^
{{\mathfrak c}} \bullet \omega_0$ belongs to the image of 
$\wedge ^{r-2}
\mathcal U_{ {\mathfrak c}}$ by 
  $i^{r-2}_{ {\mathfrak c}}$. 
Then for any conic class $\mathfrak c$, we set 
\begin{equation}
\label{Eq:varpi-frak-c}
\varpi_{\mathfrak c}= 
\boldsymbol{\epsilon}
\big(  
\sigma_{\hspace{-0.05cm} {\mathfrak c}_0}^{{\mathfrak c}}
\big)\cdot \big(i_{\mathfrak c}^{ r-2}\big)^{-1}\Big( \sigma_{\hspace{-0.05cm} {\mathfrak c}_0}^{{\mathfrak c}}\bullet \omega_0\Big)
\in 
\wedge ^{r-2}
\mathcal U_{ {\mathfrak c}}\, .
\end{equation}

\begin{lem}
\label{L:basis-varpi-c}
The tuple $\big(\varpi_{\mathfrak c}\big)_{ \mathfrak c \in {\mathcal K}_r}$ 
is well defined (up to sign) and forms a basis of $\oplus_{\mathfrak c \in {\mathcal K}_r}  \wedge^{r-2}  
\mathcal U_{\mathfrak c}$.
\end{lem}
\begin{proof}
Since $\varpi_{\mathfrak c}\neq 0$ and $\dim\, \wedge ^{r-2}
\mathcal U_{ {\mathfrak c}}=1$ for any conic class $\mathfrak c$, it is clear than 
$\big(\varpi_{\mathfrak c}\big)_{ \mathfrak c \in {\mathcal K}_r}$ is a basis.

For any $\mathfrak c \in {\mathcal K}_r$, any 
$\tilde \sigma_{\hspace{-0.05cm} {\mathfrak c}_0}^
{{\mathfrak c}} \in W_r$ such that 
$\tilde \sigma_{\hspace{-0.05cm} {\mathfrak c}_0}^
{{\mathfrak c}}({\mathfrak c}_0)={\mathfrak c}$ necessarily is written 
$\tilde \sigma_{\hspace{-0.05cm} {\mathfrak c}_0}^
{{\mathfrak c}}=\sigma_{\hspace{-0.05cm} {\mathfrak c}_0}^
{{\mathfrak c}}\,\nu$ for some $\nu \in F_{\hspace{-0.05cm}{\mathfrak c}_0}$. 
Thus one has $\tilde \sigma_{\hspace{-0.05cm} {\mathfrak c}_0}^
{{\mathfrak c}}\bullet \omega_0= 
\big( \sigma_{\hspace{-0.05cm} {\mathfrak c}_0}^
{{\mathfrak c}}\,\nu \big)
\bullet \omega_{{\mathfrak c}_0}
=
 \sigma_{\hspace{-0.05cm} {\mathfrak c}_0}^
{{\mathfrak c}}\bullet \big( 
\nu \bullet \omega_{{\mathfrak c}_0}
\big)$. 
By Lemma \ref{L:Fc-module}.3), one has  $\nu \bullet \omega_{{\mathfrak c}_0}=
\boldsymbol{\epsilon} (\nu)\,
\omega_{{\mathfrak c}_0}$ hence 
$\tilde \sigma_{\hspace{-0.05cm} {\mathfrak c}_0}^
{{\mathfrak c}}\bullet \omega_0= 
\boldsymbol{\epsilon} (\nu)\,
 \sigma_{\hspace{-0.05cm} {\mathfrak c}_0}^
{{\mathfrak c}}\bullet \omega_0$ and consequently 
$
\boldsymbol{\epsilon} 
\big(\tilde \sigma_{\hspace{-0.05cm} {\mathfrak c}_0}^
{{\mathfrak c}}\big)\,\big( 
\tilde \sigma_{\hspace{-0.05cm} {\mathfrak c}_0}^
{{\mathfrak c}}\bullet \omega_0 \big)=
\boldsymbol{\epsilon} 
\big( \sigma_{\hspace{-0.05cm} {\mathfrak c}_0}^
{{\mathfrak c}}\big)\,\big( 
 \sigma_{\hspace{-0.05cm} {\mathfrak c}_0}^
{{\mathfrak c}}\bullet \omega_0 \big)$. This implies that 
$\varpi_{\mathfrak c}$ does not depend on $ \sigma_{\hspace{-0.05cm} {\mathfrak c}_0}^
{{\mathfrak c}}$ but only on the coset  $\sigma_{\hspace{-0.05cm} {\mathfrak c}_0}^
{{\mathfrak c}} F_{\hspace{-0.05cm} {\mathfrak c}_0 }$, that is on the associated conic class ${\mathfrak c}$.  Finally, it is straightforward (and left to the reader) to verify that $\big(\varpi_{\mathfrak c}\big)_{ \mathfrak c \in {\mathcal K}_r}$ is independent of the choice of the base conic class ${\mathfrak c}_0$.  
\end{proof}

For any $w\in W_r$ and any $\mathfrak c\in {\mathcal K}_r$, one sets 
\begin{equation}
\label{Eq:w-circ-varpi-frak-c}
w \circ \varpi_{\mathfrak c} = \boldsymbol{\epsilon}(w)\,
\varpi_{w(\mathfrak c)}\, .
\end{equation}
Since  $\big(\varpi_{\mathfrak c}\big)_{ \mathfrak c \in {\mathcal K}_r}$ is a basis of $\oplus_{\mathfrak c \in {\mathcal K}_r}  \wedge^{r-2}  
\mathcal U_{\mathfrak c}$,  $\circ$ can be extended by linearity to the whole space. We thus obtain a linear $W_r$-action, again denoted by $\circ$, with regard to which the following holds true: 
\begin{lem}
\label{L:Psi-r-2-equiv}
{\rm 1.} The map $\Psi^{r-2}$  is $W_r$-equivariant.\sk

{\rm 2.} As $W_r$-modules, one has $\bigoplus_{\mathfrak c \in {\mathcal K}_r}  \wedge^{r-2}  
\mathcal U_{\mathfrak c}\simeq \boldsymbol{\epsilon}\otimes \mathbf C^{{\mathcal K}_r}$.
\end{lem}
\begin{proof}
The first point follows easily from the very definition 
\eqref{Eq:varpi-frak-c}
of the $\varpi_{ {\mathfrak c}  }$'s and the second  
is a direct consequence of \eqref{Eq:w-circ-varpi-frak-c}. 
\end{proof}

We now define the main hero of this paper, namely 
\begin{equation}
\label{Eq:hero}
{\bf hlog}^{r-2}= \sum_{{\mathfrak c}\in {\mathcal K}_r} 
\varpi_{{\mathfrak c}}\,.
\end{equation}
From Lemma \ref{L:basis-varpi-c}, it follows that it is a non zero element of  $\oplus_{\mathfrak c \in {\mathcal K}_r}  \wedge^{r-2}  
\mathcal U_{\mathfrak c}$ which is canonically defined up to sign.
Its most important properties are the content of the following
\begin{prop} 
\label{Prop:Invariant-Formulation}
 {\rm 1.} The element ${\bf hlog}^{r-2}$ spans the signature subrepresentation of $\oplus_{\mathfrak c \in {\mathcal K}_r}  \wedge^{r-2}  
\mathcal U_{\mathfrak c}$. 

 {\rm 2.}  Hence for 
$r=4,\ldots,7$,  
 ${\bf hlog}^{r-2}$ lies in ${\rm ker}\big( \Psi^{r-2}\big)$, 
{\it i.e.}\,one has 
$
\sum_{{\mathfrak c}\in {\mathcal K}_r} 
\omega_{{\mathfrak c}}=0
$
in 
$\wedge ^{r-2}\mathbf C^{{\mathcal L}_r}$.
\end{prop}
\begin{proof}
For $w\in W_r$, it follows from \eqref{Eq:w-circ-varpi-frak-c} and the very definition of ${\bf hlog}^{r-2}$ that 
\begin{equation}
w\circ {\bf hlog}^{r-2}=
\sum_{{\mathfrak c}\in {\mathcal K}_r} 
w\circ \varpi_{{\mathfrak c}}=
\sum_{{\mathfrak c}\in {\mathcal K}_r} 
\boldsymbol{\epsilon}(w)\, 
 \varpi_{w({\mathfrak c})}=
 \boldsymbol{\epsilon}(w)\, \left(
\sum_{{\mathfrak c}\in {\mathcal K}_r} 
\varpi_{w{\mathfrak c}}
\right)=\boldsymbol{\epsilon}(w)\,  {\bf hlog}^{r-2}\,.
\end{equation}
Since ${\bf hlog}^{r-2}$ is not trivial, it spans a $W_r$-subrepresentation of $\oplus_{\mathfrak c \in {\mathcal K}_r}  \wedge^{r-2}  
\mathcal U_{\mathfrak c}$ isomorphic to $\boldsymbol{\epsilon}$.  Since 
there is a unique such subrepresentation (as 
it follows immediately from Lemma \ref{L:Psi-r-2-equiv}.2 and \eqref{Eq:Decomposition-Kr}), one gets the first point.

To get the second point, recall that $\Psi^{r-2}$ is a morphism of $W_r$-representations according to Lemma \ref{L:Psi-r-2-equiv}.1. If 
${\bf hlog}^{r-2}$ were not zero, then it would span a sub-module of 
$\wedge ^{r-2} \mathbf C^{ {\mathcal L}_r  } $ isomorphic to  $\boldsymbol{\epsilon}$ according to 1. Since there is no such submodule 
in $\wedge ^{r-2}\mathbf C^{{\mathcal L}_r}$
for 
$r=4,\ldots,7$ according to the fact just below, 
we necessarily have $\Psi^{r-2}\big( {\bf hlog}^{r-2}\big)=0$.
\end{proof}

\begin{fact}
\label{Fact-mult-signature}
For $r=4,\ldots,8$, let $m^{\boldsymbol{\epsilon}}_r$ 
be the multiplicity of the signature representation in the decomposition 
of $\wedge ^{r-2} \mathbf C^{ {\mathcal L}_r}$ into $W_r$-irreducibles.  Then 
\begin{equation}
\label{m-signa}
m^{\boldsymbol{\epsilon}}_r=
\begin{cases} 
\, 0 \hspace{0.3cm} \mbox{for } \, r=4,\ldots,7 \\
\, 5 \hspace{0.3cm} \mbox{for } \, r=8\, .
\end{cases}
\end{equation}
\end{fact}
\begin{proof}
This has been obtained by means of formal computations in GAP (verified using Maple as well), see at the end of the Appendix for more details.  
\end{proof}

\subsubsection{}
\label{SSS:C-Kr}
As shown just above, it is not necessary to know the complete decomposition of $\mathbf C^{{\mathcal K}_r}$ into $W_r$-irreducibles to get our main result but we find this interesting. 
Following a similar approach to the one used in the proof of Proposition \ref{Prop:Key}, we have obtained that as $W_r$-modules,  one has  (up to isomorphism): 
\begin{align}
\label{Eq:Decomposition-Kr}
\mathbf C^{{\mathcal K}_4}= & \,  {\bf 1}\oplus V_{[41]}^4
 \nonumber \\ 
 \mathbf C^{{\mathcal K}_5}=&  \,  {\bf 1}\oplus V_{[-,41]}^4\oplus V_{[4,1]}^{5}
  \nonumber \\
   \mathbf C^{{\mathcal K}_6}= &  \,   {\bf 1}\oplus V^{6,1}\oplus V^{20,2} 
  \\
 \mathbf C^{{\mathcal K}_7}=&  \,    {\bf 1}\oplus V^{7,1}\oplus V^{27,2}\oplus V^{35,4}\oplus V^{56,3}
  \nonumber \\
\mathbf C^{{\mathcal K}_8}= &  \,  {\bf 1}\oplus V^{8,1}\oplus
V^{35,2}
\oplus V^{50,8} 
\oplus V^{84,4} 
\oplus V^{112,3}
\oplus V^{210,4}
\oplus V^{400,7}
\oplus V^{700,6} 
 \nonumber
%
%
\end{align}

 Let us make three remarks about these decompositions.
\begin{itemize}
\item
 Computing the decompositions in irreps above has been implemented quite easily on a  computer and was rather fast except for $\mathbf C^{{\mathcal K}_8}$. The  degree of the latter (one has $\lvert {\mathcal K}_8\lvert =2160$) and the relatively high number  of conjugacy classes in $W(E_8)$ (namely 112) make the calculation time explode (it took several hundreds of computations vs just a few hours for all the other cases). 
 It would be interesting to figure a conceptual approach for establishing the decompositions \eqref{Eq:Decomposition-Kr}.
\sk
\item Any $W_r$-irrep involved in one of the decompositions \eqref{Eq:Decomposition-Wr}  or \eqref{Eq:Decomposition-Kr} appears in it with multiplicity one. This is an interesting fact for which 
we have no explanation yet.
\mk
\item One should not be surprised by the facts that not only $\mathbf C^{{\mathcal L}_6}$ and 
$\mathbf C^{{\mathcal K}_6}$ have  the same dimension but that they share the same decomposition into $W_r$-irreps as well. Indeed, there is a well-known natural bijection $\beta : {{\mathcal L}_6}\rightarrow {{\mathcal K}_6}$ which is given by $l\mapsto -K3-l$ in terms of classes in the Picard group. Geometrically, this bijection can be described as that which associates to any line $l$ contained in the anticanonical embedding $ {\rm dP}_3$ (which is a cubic surface in $\mathbf P^3$) the pencil of conics on it induced by the plane sections containing $l$. Because 
the $W(E_6)$-action on ${\bf Pic}({\rm dP}_3)$ lets both the canonical class $K_3$ and the set of lines $ {\mathcal L}_6$ invariant, 
it follows that the bijection $\beta $ is $W(E_6)$-equivariant 
which in its turn gives immediately  that $\mathbf C^{{\mathcal L}_6}$ and 
$\mathbf C^{{\mathcal K}_6}$ are canonically isomorphic $W(E_6)$-modules.
\sk
\end{itemize}

\subsubsection{}
\label{SS:r=8-HLog6}
As a final remark, let us mention that both our approaches  above do not say much about the case $r=8$. Indeed, even if $\wedge ^6 \mathbf C^{ {\mathcal L}_8}$ admits five copies of the signature as $W(E_8)$-modules,  it may be the case that 
the image by $\Psi^6$ of $  {\bf hlog}^{6}$ be zero. 
Our representation-theoretic arguments to deal with the other cases 
do not allow to conclude here.  As for our attempt to check whether $\Psi^6( {\bf hlog}^{6})=0$ by means of brute force computations as described in \S\ref{SS:Elementary-Proof}, it failed: although elementary, the calculations involved were too cumbersome and time-consuming to be carried out successfully.

\section{\bf Miscellaneous}
\label{S:Miscellaneous}

In this section, we first discuss the case when $d=4$ with more details: 
after having made the identity ${\bf HLog}^{3}$ more explicit, 
 we discuss some very interesting web-geometric properties of $\boldsymbol{\mathcal W}_{ {\rm dP}_4}$.  We end in \S\ref{S:singular} by adding a few words about the rich and interesting case of singular del Pezzo surfaces.

\subsection{\bf Explicit version of  ${\bf HLog}^3$}
\label{S:singular}
The formal identity ${\bf HLog}^2$ is equivalent to Abel's relation
$\boldsymbol{\big(\mathcal Ab\big)}$
 which is written in quite explicit form.  It is interesting to make the other hyperlogarithmic identities ${\bf HLog}^{r-2}$ ($r=4,\ldots,7$) as explicit as possible as well. Here we discuss the case when $r=5$ (or equivalently, the case when $d=4$).\sk

A smooth del Pezzo quartic surface in $\mathbf P^4$ is (isomorphic to) the blow-up of $\mathbf P^2$ at the following five points:   $p_1=[1:0:0]$, $p_2=[0:1:0]$, $p_3[0:0:1]$,  $p_4[1:1:1]$ and $p_5=[\pi : \gamma : 1]$, 
for some parameters $\pi,\gamma\in \mathbf C$ such that these five points are in general position, which corresponds to the condition 
\begin{equation}
\label{Eq:PG-pi-gamma}
\pi\gamma(\pi-1)(\gamma-1)(\pi-\gamma)\neq 0
\end{equation}
that 
 we assume to be satisfied in what follows.  One denotes by $b: {\rm dP}_4\rightarrow \mathbf P^2$ 
the blow-up of the projective plane at these  five points. \sk

We want to make a bit more explicit the push-forward of $\boldsymbol{\mathcal W}_{\rm dP4}$ by $b$ (that we will denote in the same way) as well as  the  associated weight 3 hyperlogarithmic identity $({\bf H Log}^3)$ when expressed in the 
affine coordinates $x,y$ corresponding to the affine embedding $\mathbf C^2\hookrightarrow \mathbf P^2$, $(x,y)\mapsto [x:y:1]$. 
This is straightforward: relatively to these coodinates, the web $\boldsymbol{\mathcal W}_{\rm dP4}$ under scrutiny is easily seen as coinciding with the 
web $\boldsymbol{\mathcal W}\big( U_1,\ldots,U_{10}\big)$
 defined  by the following ten rational functions: 
\begin{align}
\label{Eq:WdP4-Ui}
U_1= & \,  x && U_6=     \frac{
   (1 - x)\gamma + x + (\pi - 1)y - \pi}{ (x - 1)(y-\gamma)}
 \nonumber \\
U_2= & \,  \frac1y
&& U_7=
          \frac{ 
    (x - y)(y-\gamma)}{ y(\pi y - \gamma x - \pi + \gamma + x - y)}
   \nonumber  \\
U_3= & \, \frac{y}{x} && 
U_8=  \frac{    
     -x(x(\gamma - 1) + (1-y )\pi - \gamma + y)}{ (x - y)(x - \pi)}
\\
U_4= & \, 
\frac{x-y}{x-1} 
&&
U_9= \,
          \frac{  
      y(x - \pi)}{ x (y-\gamma)} 
 \nonumber\\
  U_5= & \, 
   \frac{\gamma(\pi-x ) }{\pi y - \gamma x}
&& 
U_{10}=  \frac{
       x(y - 1)}{y (x - 1)} 
        \nonumber
\end{align} 

The rational curves in $\mathbf P^2$\footnote{The `directrices' of the points  $p_1,\ldots,p_5$ according to Du Val terminology in \cite{duVal}.}
  corresponding to the 11 lines in $X_5$ distinct from 
  the exceptional divisors $\ell_1,\ldots,\ell_5$
are the lines at infinity plus the closures in $\mathbf P^2$ of the affine curves with equation $\mathscr L_i =0$, where the $\mathscr L_i's$ are the components of the following 10-tuple
  of polynomials in $x,y$:
\begin{align*}
 \mathscr L= \big( 
\mathscr L_i 
\big)_{i=1}^{10}
=
\bigg( \, x \, , &  \,  y \, , \,   y -\gamma  \, , \,    x -1\, , \,  x -\pi \, , \,   x -y \, , \,  y -1\, , \,   \\
&\, \gamma\,  \Big((x -y) \pi +x (y -1)\Big)  -\pi \, y   \big(x -1\big)
\, , \,  
\gamma \, \big(x -1\big)  -\pi\, (y -1)  +y -x \, , \,   \gamma  \, x -\pi  \, y\, 
\bigg) \, .
\end{align*}
We denote by $A_{\mathscr L}$ the union in $\mathbf C^2$ of the curves cut out by the equations $\mathscr L_i=0$ with $i=1,\ldots,10$. Then  
$\boldsymbol{\mathcal W}\big( U_1,\ldots,U_{10}\big)$ is a non singular web on the Zariski open set $ b(Y_5)=b( X_5\setminus L_5)= \mathbf C^2 \setminus A_{\mathscr L}$.

One considers the  associated logarithmic forms 
 $ h_i=d\log(\mathscr L_i)={d\mathscr L_i}/{\mathscr L_i}$ with $i=1,\ldots,10$. 
 Then for each $i$, the spectrum\footnote{By definition, the `spectrum' of a rational map $f\in \mathbf C(x,y)$ is the set of values $\lambda \in \mathbf P^1$ such that  $f^{-1}(\lambda)$ is not irreducible.} of the rational map $U_i : \mathbf P^2\dashrightarrow \mathbf P^1$ with respect to the coordinates $x,y$ we are working with is 
 ${\mathfrak R}_i=\big( \, 0\, , \,  1 \, , \, {r}_i\, , \,  \infty \, \big)$
 where the $r_i$'s are the components of the following $10$-tuple of complex numbers: 
$$ 
{\mathfrak r}=
\big( {\mathfrak r}_i\big)_{i=1}^{10}
=
\left(\,
\pi 
 \, , \,
\frac{1}{\gamma}
 \, , \,
\frac{\gamma}{\pi}
 \, , \,
\frac{\pi-\gamma}{\pi -1}
 \, , \,
\frac{\gamma  (\pi -1)}{\pi-\gamma}
 \, , \,
\frac{\gamma -\pi}{\gamma}
 \, , \,
\frac{1}{1-\pi}
 \, , \,
1-\gamma
 \, , \,
\frac{\pi -1}{\gamma -1}
 \, , \,
\frac{\pi  (\gamma -1)}{\gamma  (\pi -1)} \, 
\right)\, .
$$
Remark that since \eqref{Eq:PG-pi-gamma} is assumed to hold true 
then:

$-$ all the $\mathscr L_i$'s are linearly independent as affine equations in $x$ and $y$; hence 

$-$ the same holds true for the $h_i$'s (as logarithmic forms in the same variables); and 

$-$ none of the $
r_i$'s coincides with an element of $\{0,1,\infty\}\subset \mathbf P^1$ hence each ${\mathfrak R}_i$ indeed is a 4-tuple\\
${}^{}$ \hspace{0.57cm} of pairwise distinct elements of $\mathbf P^1$.
\sk

For every $i=1,\ldots,10$,  one sets $R_i=(R_{i,1}, R_{i,2}, R_{i,3})$ 
where $R_{i,s}$ stands for the decomposition
of $d{\rm Log}(U_i-\mathfrak R_{i,s})$ (that is of 
$dU_i/U_i$, $dU_i/(U_i-1)$ and $dU_i/(U_i-{\mathfrak r}_i)$ 
for $s=1,2$ and $3$ respectively)
 as a linear combination of the $h_j$'s.  
 By straightforward computations, one gets that 
\begin{align}
\label{Eq:hh}
R_1=& \, \Big( h_{1} \, , \,h_{4} \, , \,h_{5}\Big)  \nonumber 
\\ R_2= & \, 
\Big(-h_{2} \, , \,h_{7}-h_{2} \, , \,h_{3}-h_{2}\Big) \nonumber 
\\
 R_3= & \, 
\Big(-h_{1}+h_{2} \, , \,-h_{1}+h_{6} \, , \,-h_{1}+h_{10}\Big)\nonumber 
\\  R_4=  & \, 
\Big(-h_{4}+h_{6} \, , \,h_{7}-h_{4} \, , \,h_{9}-h_{4}\Big)\nonumber 
\\
 R_5= & \, 
\Big(-h_{10}+h_{5} \, , \,h_{3}-h_{10} \, , \,h_{9}-h_{10}\Big) 
\\  R_6= & \, 
\Big( -h_{3}+h_{9}-h_{4} \, , \,h_{7}-h_{3}-h_{4}+h_{5} \, , \,-h_{3}-h_{4}+h_{8}\Big)\nonumber 
\\
 R_7= & \, 
\Big(h_{3}-h_{9}+h_{6}-h_{2} \, , \,h_{7}-h_{9}+h_{10}-h_{2} \, , \,-h_{9}-h_{2}+h_{8}\Big) \nonumber 
\\
 R_8= & \, 
\Big( h_{9}+h_{1}-h_{5}-h_{6} \, , \,h_{4}-h_{5}-h_{6}+h_{10} \, , \,-h_{5}-h_{6}+h_{8}\Big)\nonumber 
\\
 R_9= & \, 
\Big( 
-h_{3}-h_{1}+h_{5}+h_{2} \, , \,-h_{3}-h_{1}+h_{10} \, , \,-h_{3}-h_{1}+h_{8}\Big)
\nonumber  \\
 R_{10}= & \, 
\Big( 
h_{7}+h_{1}-h_{4}-h_{2} \, , \,-h_{4}+h_{6}-h_{2} \, , \,-h_{4}-h_{2}+h_{8}\, \Big)
\, . \nonumber 
\end{align}

For a triple $(a,b,c)$ of pairwise distinct points on $\mathbf C$ and given a base point $\zeta \in \mathbf C\setminus \{a,b,c\}$, we consider the  weight 3 hyperlogarithm $L_{a,b,c}^{\zeta}$ 
defined  by
$$ 
L_{a,b,c}^{\zeta}(z)= \int_{\xi}^{z}   \Bigg(\int_{\xi}^{u_3} \bigg(\int_{\xi}^{u_2} \frac{du_1}{u_1-c}\bigg) \frac{du_2}{u_2-b}\Bigg)  \frac{du_3}{u_3-a} $$
for any $z$ sufficiently close to $\zeta$, and we denote by $AI_{a,b,c}^{\zeta}$ its antisymmetrization:
\begin{equation*}
AI_{a,b,c}^{\zeta}=\frac{1}{6} \, \bigg(\, L_{a,b,c}^{\zeta}-L_{a,c,b}^{\zeta}-L_{b,a,c}^{\zeta}+L_{b,c,a}^{\zeta}+L_{c,a,b}^{\zeta}-L_{c,b,a}^{\zeta}
\bigg) 
\, .
\end{equation*}

We now fix a base point $\xi\in \mathbf C^2\setminus A_{\mathcal L}$ and 
for $i=1,\ldots,10$, we set  $\xi_i=U_i(\xi)\in \mathbf C\setminus \{0,1,r_i\}$  and
 $$
 AI^i=AI^{\xi_i}_{0,1,r_i}\, .
 $$

For any $i$, the symbol 
$\mathcal S\big( AI^i(U_i)\big)$
of $AI^i\big(U_i\big)=U_i^*(AI^i)$ is the antisymmetrization of $(dU_i/U_i)\otimes (dU_i/(U_i-1))\otimes (dU_i/(U_i-r_i)$ which,  according to \eqref{Eq:hh},  can be 
written in a unique way as a linear combination of the 
$h_{ijk}= h_i\otimes h_j\otimes h_k$'s for all triples $(i,j,k)$ such that $1\leq i\leq j\leq k\leq 10$.

Since 
$$ 
\sum_{i=1}^{10}  {\rm Asym}^{3}\left(  \left( 
\frac{dU_i}{
U_i} \right) 
\otimes \left(  \frac{dU_i}{
U_i-1} \right) 
\otimes \left(  
\frac{dU_i}{
U_i-r_i} \right)   \right) =0
$$
as an elementary linear algebra computation shows (in the vector space spanned by the weight 3 tensors $h_{ijk}$'s), we get that the functional relation
 \begin{equation}
\label{Eq:Eq-explicit-d=4}
\sum_{i=1}^{10} AI^i\big(U_i\big)=0
\end{equation}
is identically satisfied on any sufficiently small neighbourhood of $\zeta$. 

What has been made explicit in the above identity is the $U_i$'s and it would be interesting to give an explicit formula for each of the functions $AI^i$.  The consideration of formula \eqref{Eq:AI-z-3} leads us to think that this might be done. We hope to come back to this in a future work.

Lastly, considering the case of Abel's five terms identity, another question arises about the existence of a global real analytic version of \eqref{Eq:Eq-explicit-d=4}. We will discuss this a bit further in \S\ref{S:single-valued} below.

\subsection{\bf Some nice properties of ${\mathcal W}_{ {\rm dP}_4}$}
\label{S:Web-WdP4}
We can not refrain from mentioning, without proof, some remarkable properties of ${\mathcal W}_{ {\rm dP}_4}$. Everything discussed below will be established in our paper to come \cite{WdP4}, this subsection can just be seen as advertising it. For all the notions of web geometry considered here, see  
\cite{PP}. 

\subsubsection{}
\label{S:Web-WdP5-properties}
Abel's identity $\boldsymbol{(\mathcal Ab)}$ gives rise to an abelian relation (ab.\,AR) for $\boldsymbol{\mathcal B}$, that we will denote by  $\boldsymbol{\mathcal Ab}$.  
One sets  $\boldsymbol{AR}\big(\boldsymbol{\mathcal B}\big)$  for 
the space of abelian relations of 
$\boldsymbol{\mathcal B}$.  It has been known for a long time that this web 
 is very special as a planar web since it enjoys several remarkable 
 properties:  
\begin{itemize}
\item[] ${}^{}$ 
\hspace{-1cm}
{\bf [\,Non linearizability\hspace{0.03cm}].} {\it 
Bol's web is not linearizable hence not equivalent to an algebraic web.}
\mk 
\vspace{-0.4cm}
\item[] ${}^{}$ 
\hspace{-1cm}
{\bf [\,Abelian relations\hspace{0.03cm}].}
{\it All the abelian relations of $\boldsymbol{\mathcal B}$ are polylogarithmic, of weight 1 or 2:}\\
${}^{}$ \hspace{0.9cm} 
{\it   $-$ the subspace  $\boldsymbol{AR}_1\big(
\boldsymbol{\mathcal B}
\big)$ of logarithmic ARs of $\boldsymbol{\mathcal B}$ has dimension 5;}\\
${}^{}$ \hspace{0.9cm} 
{\it   $-$ the subspace  $\boldsymbol{AR}_2\big(
\boldsymbol{\mathcal B}
\big)$ of dilogarithmic ARs of $\boldsymbol{\mathcal B}$ 
is spanned by $\boldsymbol{\mathcal Ab}$; and}\\
${}^{}$ \hspace{0.9cm} 
  {\it $-$ there is a decomposition in direct sum}\, 
 $\boldsymbol{AR}\big(\boldsymbol{\mathcal B}\big)=\boldsymbol{AR}_{1}\big(\boldsymbol{\mathcal B}\big)\oplus \big\langle \boldsymbol{\mathcal Ab} \big\rangle 
$\, $(\star)$. 
\mk
\item[] ${}^{}$ 
\hspace{-1cm}
{\bf [\,Rank\hspace{0.03cm}].} {\it Bol's web has maximal rank hence is exceptional}.
\mk 
\item[] ${}^{}$ 
\hspace{-1cm}
{\bf [\,Weyl group action\hspace{0.03cm}].} {\it There is a natural action of 
$W(E_4)=\mathfrak S_5$ on  the space of abelian 
  relations regarding which  $(\star)$ is the decomposition into $\mathfrak S_5$-irreducible representations.}
  \mk 
\item[] ${}^{}$ 
\hspace{-1cm}
 {\bf [\,Hexagonality \& Characterization\hspace{0.03cm}].} {\it Bol's web is hexagonal and is essentially characterized by this property 
 since 
 `for any $k\geq 3$, a  hexagonal planar $k$-web  either is linearizable and   equivalent to a web formed by $k$ pencils of lines  or $k=5$ and it is equivalent to $\boldsymbol{\mathcal B}$'.}
 \mk
\item[] ${}^{}$ 
\hspace{-1cm}
  {\bf [\,Construction 
  \textit{\textbf{\`a la}}
  GM\hspace{0.03cm}].} {\it 
   $\boldsymbol{\mathcal B}$ 
   can be constructed following the approach of Gelfand and MacPherson {\rm \cite{GM}:}  it can be seen as  the quotient, under the action of the Cartan torus      $H_4\subset {\rm SL}_5(\mathbf C)$, of a natural $H_4$-equivariant web defined on the grassmannian 
 $G_2(\mathbf C^5)$.}
  \mk
\item[] ${}^{}$ 
\hspace{-1cm} 
{\bf [\,Cluster web\hspace{0.03cm}].} 
{\it Bol's web is of cluster type: up to equivalence, it can be defined by means of the $\mathscr  X$-cluster variables of the finite type cluster algebra of type $A_2$.}
 \end{itemize}

\subsubsection{} 
\label{S:Web-WdP4-properties}

Since $\boldsymbol{\mathcal B}$ is equivalent to $\boldsymbol{\mathcal W}_{ {\rm dP}_5}$ and because what can be considered as its most important feature, namely the fact it carries the dilogarithmic abelian relation  
$\boldsymbol{\mathcal Ab}$, generalizes to $\boldsymbol{\mathcal W}_{ {\rm dP}_4}$, one can wonder whether the nice properties listed above admit analogs for this 10-web.
It turns out that it is indeed the case since it can be proved that the following properties are satisfied: 

\begin{itemize}
\item[] ${}^{}$ 
\hspace{-1cm}
{\bf [\,Non linearizability\hspace{0.03cm}].} 
{\it $\boldsymbol{\mathcal W}_{ {\rm dP}_4}$ is not linearizable hence not equivalent to an algebraic web.}\mk 
\item[] ${}^{}$ 
\hspace{-1cm}
{\bf [\,Abelian relations\hspace{0.03cm}].}
{\it All the ARs of $\boldsymbol{\mathcal W}_{ {\rm dP}_4}$ are hyperlogarithmic, of weight 1, 2 or 3:\sk}\\
${}^{}$ \hspace{-0.7cm} 
{\it   $-$ the subspace  $\boldsymbol{AR}_1\big(
\boldsymbol{\mathcal W}_{ {\rm dP}_4}
\big)$ of logarithmic ARs of $\boldsymbol{\mathcal W}_{ {\rm dP}_4}$ has dimension 20;}\sk \\
 ${}^{}$ \hspace{-0.7cm}  {\it  
  $-$ the subspace  $\boldsymbol{AR}_2\big(
\boldsymbol{\mathcal W}_{ {\rm dP}_4}
\big)$ of weight 2 hyperlogarithmic ARs has dimension 15: it is the} \\
 ${}^{}$ \hspace{-0.8cm}   {\it \textcolor{white}{$-$} direct sum
 of two subspaces,  
a first one $\boldsymbol{AR}_{2}^{\rm sym}$  of symmetric ARs, of dimension 5 and }
\\
 ${}^{}$ \hspace{-0.8cm}   {\it \textcolor{white}{$-$} another one $ \boldsymbol{AR}_{2}^{\rm asym}$ of antisymmetric ARs, of dimension 10;} \sk\\ 
${}^{}$ \hspace{-0.7cm} 
{\it   $-$ the subspace  $\boldsymbol{AR}_3\big(
\boldsymbol{\mathcal W}_{ {\rm dP}_4}
\big)$ of weight 3 hyperlogarithmic ARs of $\boldsymbol{\mathcal W}_{ {\rm dP}_4}$ 
is spanned by ${\bf Hlog^3}$;}\sk \\
${}^{}$ \hspace{-0.7cm} 
  {\it $-$ there is a decomposition in direct sum} 
 \\ 
${}^{}$ \hspace{3cm} 
$\boldsymbol{AR}\big( 
\boldsymbol{\mathcal W}_{ {\rm dP}_4} 
 \big)=\boldsymbol{AR}_{1}\big(\boldsymbol{\mathcal W}_{ {\rm dP}_4}\big)\oplus 
 \Big(
 \boldsymbol{AR}_{2}^{\rm sym}\oplus 
 \boldsymbol{AR}_{2}^{\rm asym}
 \Big)
 \oplus
 \big\langle  \,
 {\bf Hlog^3}\, 
 \big\rangle 
$  \qquad  $(\clubsuit)$
\mk
\vspace{-0.5cm}
\item[] ${}^{}$ 
\hspace{-1cm}
{\bf [\,Rank\hspace{0.03cm}].} 
{\it The web $\boldsymbol{\mathcal W}_{ {\rm dP}_4}$ has maximal rank hence is exceptional}. 
\mk 
\item[] ${}^{}$ 
\hspace{-1cm}
{\bf [\,Weyl group action\hspace{0.03cm}].} 
{\it The Weyl group $W(D_5)$ acts naturally on $\boldsymbol{AR}\big(\boldsymbol{\mathcal W}_{ {\rm dP}_4} \big)$ and $(\clubsuit)$
 actually is  its decomposition into irreducible $W(D_5)$-modules.}
  \mk 
\item[] ${}^{}$ 
\hspace{-1cm}
 {\bf [\,Hexagonality \& Characterization\hspace{0.03cm}].} 
{\it $\boldsymbol{\mathcal W}_{ {\rm dP}_4}$ is characterized by  its hexagonal 3-subwebs.}\footnote{See \S\ref{SS:Combinatorial-characterization} further for a  more precise statement.}
 \mk
\item[] ${}^{}$ 
\hspace{-1.05cm}
  {\bf [\,Construction 
   \textit{\textbf{\`a la}}
   GM\hspace{0.03cm}].} 
 {\it  
  $\boldsymbol{\mathcal W}_{ {\rm dP}_4}$  can be seen as  the quotient, under the action of the Cartan torus   $H_5\subset {\rm S0}_{10}(\mathbf C)$, of a natural $H_5$-equivariant web defined on the 
spinor variety $\mathbb S_{5}$.
}
 \mk
\item[] ${}^{}$ 
\hspace{-1.05cm}
{\bf [\,Cluster web\hspace{0.03cm}].} 
{\it $\boldsymbol{\mathcal W}_{ {\rm dP}_4}$ is of cluster type: it can be obtained by means of some of the $\mathscr X$-cluster  variables of the finite type cluster algebra of type $D_4$.} 
  \mk
 \end{itemize}
\begin{rem}
{\bf 1.}\,Without additional explanations, the statements above may sound cryptic for non-specialists of web geometry. We will come back in detail to all of this in \cite{WdP4}. \vspace{0.2cm}\\ 
${}^{}$\quad {\bf 2.} Only very few of the statements above were known before: 
\begin{itemize}
\vspace{-0.15cm}
\item[$\boldsymbol{-}$]  by means of explicit computations on a computer, we established in \cite[\S7.3.3]{PThese} that the standard model of $\boldsymbol{\mathcal W}_{{\rm dP}_4}$ on the projective plane  has maximal rank hence is exceptional. But nothing else was known about its space of ARs at that time;
\sk
\item[$\boldsymbol{-}$] in 
\cite{Pereira}, Pereira introduced a way to associate a `{resonnance web}' 
$\boldsymbol{\mathcal W}_{{}^{} \hspace{-0.05cm}\mathcal A}$ 
to  any arrangements of lines  in the plane $\mathcal A$. When 
specializing this constuction to a certain arrangement ${\bf K}_5$, one recovers the planar model of $\boldsymbol{\mathcal W}_{{\rm dP}_4}$. Then 
Pereira showed that the  `{\it hyperlogarithmic rank}\footnote{This is the immediate `hyperlogarithmic' generalization of the notion of `{\it polylogarithmic rank'} considered in \cite[\S1.4.3.2]{ClusterWebs}.} of  $\boldsymbol{\mathcal W}_{{\bf K}_5}$ is $ (20,15,1)$ which shows in particular that this web carries a unique hyperlogarithmic AR of weight 3. However, this AR was not explicited by Pereira. 
\end{itemize} 
\end{rem}

One has to have in mind that considering that almost all (and possibly all 
{\it cf.}\,\S\ref{S:single-valued} or \S\ref{SS:A conceptual interpretation?} below) the remarkable properties of the pair Abel's identity ($\boldsymbol{\big(\mathcal Ab\big)}\simeq {\bf HLog}^2$)  / (Bol's web $\boldsymbol{\mathcal B}\simeq \boldsymbol{\mathcal W}_{ {\rm dP}_5}$)   admit formally very similar counterparts for the pair $ {\bf HLog}^3$ /  $\boldsymbol{\mathcal W}_{ {\rm dP}_4}$, then the latter may be considered as the most natural weight 3 generalization of the former.  The quest to find  natural generalizations to any weight of Abel's dilogarithmic identity is a long standing one which has received many great contributions by several authors. But except until very recently and for  very few cases, all of these were sticking to the case of polylogarithms or more precisely to that of iterated integrals on $\mathbf P^1\setminus \{0,1,\infty\}$ constructed from words in the alphabet with $\omega_0=du/u$ and $\omega_1=du/(u-1)$ as letters.  The discovery of the identities ${\bf Hlog}^{r-2}$ for $r=4,\ldots,7$ 
as well as the striking fact that $\boldsymbol{\mathcal B}\simeq \boldsymbol{\mathcal W}_{ {\rm dP}_5}$ and $\boldsymbol{\mathcal W}_{ {\rm dP}_4}$ look so similar regading the numerous remarkable properties they satisfy 
may be an indication that the genuine natural generalizations of the five terms relation of the dilogarithm have to be looked for in the more general setting of hyperlogarithms on $\mathbf P^1$.

What makes the functional equations of polylogarithms so interesting is the role these identities are playing in other areas of mathematics, such as in hyperbolic geometry or in the K-theory of number fields (regulators and Zagier's conjecture). But for the moment, we are not aware of any occurence of the identity ${\bf Hlog}^3$ within other fields of mathematics, which would be necessary for it to be fully recognized as the genuine generalization in weight 3 of Abel's identity.

\subsection{\bf Webs by conics of singular del Pezzo's surfaces}
\label{S:singular}
Del Pezzo surfaces have moduli and singular degenerations (see \cite{Demazure} for instance).  Moreover, the number of lines included in it as well as  that  of conic fibrations on such a surface may vary depending on the kind of singularities. One thus gets new examples of webs when considering the webs by conics on singular del Pezzo's which might carry interesting hyperlogarithmic ARs obtained  for instance   as suitable degeneracies of the ARs associated to the identities ${\bf HLog}^{r-2}$ ($r=4,\ldots,7$). 

A rich but sufficiently interesting case to be considered is the one of del Pezzo surfaces with finitely many singular points: these are classified as well as the lines contained in each (see 
\cite{Derenthal} for a recent reference). We will not discuss all the  cases here   but just consider one example which we find interesting: the case of Cayley's cubic.\sk

Cayley's surface $\mathscr C$ is the cubic surface in $\mathbf P^3$ with four nodal points (it is unique up to projective equivalence). It is the anticanonical image in $\mathbf P^3$ of the total space  ${ \mathscr C}'$
of the 
blow-up $\beta : { \mathscr C}' \rightarrow \mathbf P^2$ of $\mathbf P^2$ at the six vertices $v_1,\ldots,v_6$ of the non Fano arrangement $A_{\rm NF}$, pictured just below. 
 \begin{figure}[h!]
\begin{center}
\scalebox{0.25}{
 \includegraphics{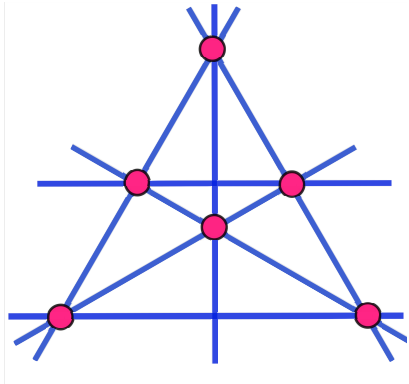}}
 \vspace{-0.1cm}
\caption{The non Fano arrangement $A_{\rm NF}$} 
\label{Fig:non-Fano}
\end{center}
\end{figure}

Four of the seven lines of $A_{\rm NF}$ pass through three of the $v_i$'s. 
The anticanonical image of the strict transform by $\beta$ of such a line is a node of $\mathscr C$ and all its singular points are obtained in this way.  

There are exactly nine lines contained in $\mathscr C$, namely 
the anticanonical images of 
the six exceptional divisors $\beta^{-1}(v_i)$  in ${\mathcal C}'$ for $i=1,\ldots,6$ and the 
images of the strict transforms of the three lines among the seven pictured in Figure \ref{Fig:non-Fano} which contain only two of the $v_i$'s.
 There are  in total nine pencils of conics on $\mathscr C$, which give rise to a 9-web on it, denoted by $\boldsymbol{\mathcal W}_{\mathscr C}$,  whose birational model on $\mathbf P^2$ is easy to describe: it  is the planar web formed by the six pencils of lines passing through the vertices of $A_{\rm NF}$ and by the three pencils of conics, one for each 4-tuple of vertices in general position among the $v_i$'s. With such a description, it is then easy to recognize 
 $\boldsymbol{\mathcal W}_{\mathscr C}$ hence to know some of the remarkable properties it satisfies: 
 \begin{quote}
 {\it  Up to equivalence, the web by conics $\boldsymbol{\mathcal W}_{\mathscr C}$ on Cayley's cubic surface coincides with Spence-Kummer web $\boldsymbol{\mathcal W}_{\hspace{-0.05cm}\mathcal S\mathcal K}$ associated to the identity with the same name satisfied by the trilogarithm.  Consequently, the web $\boldsymbol{\mathcal W}_{\mathscr C}$\sk
\begin{itemize}  
\item is not linearizable; 
\sk 
\item  has maximal rank hence is exceptional; 
\sk 
\item    has its space of ARs with a basis whose all elements 
 are polylogarithmic, except one (which is however polylogarithmic but up to a finite covering).
In particular, $\boldsymbol{\mathcal W}_{\mathscr C}$ carries two linearly independent polylogarithmic ARs of weight 3.
\end{itemize}}
 \end{quote}

We see some significant differences between the properties $\boldsymbol{\mathcal W}_{\mathscr C}$ satisfies and those enjoyed  by the 27-web  $\boldsymbol{\mathcal W}_{{\rm dP}_3}$ associated to `the' smooth cubic surface ${\rm dP}_3$: the latter carries a weight 4 hyperlogarithmic AR (namely ${\bf HLog}^4$) but does not have maximal rank whereas the former only carries trilogarithmic ARs and is exceptional.  It is not clear whether the properties of $\boldsymbol{\mathcal W}_{\mathscr C}$ can be deduced from those satisfied by $\boldsymbol{\mathcal W}_{{\rm dP}_3}$ via a degeneration process. This would be worth studying.\sk

As is well-known, the number of lines on a singular cubic surface in $\mathbf P^3$ with isolated singular points can be any one of the following integers 
1, 2, 3, 4, 5, 6, 7, 8, 9, 10, 11, 12, 15, 16, 21, or 27 and any one of these cases indeed occurs (see \cite[\S3]{BW} for instance). We thus get as many webs by conics on the corresponding cubic surfaces. We have already verified that many of these webs have interesting web-theoretic properties. 
\sk 

In a future work, we plan to study systematically and thoroughly webs by conics on del Pezzo surfaces with finitely many singular points.

%
%

%
%

%
%
\section{\bf Questions and perspectives}
\label{S:Questions}
We list below some questions or research projects that we find interesting and that the discovery of the hyperlogarithmic identities ${\bf HLog}^{r-2}$ 
suggested to us.

\subsection{\bf The case of $ {\rm dP}_1$}
As mentioned above in \S\ref{SS:r=8-HLog6}, we do not know whether the element ${\bf hlog}^{6}$ defined in 
\eqref{Eq:hero} gives rise to a hyperlogarithmic abelian relation for the 2160-webs $\boldsymbol{\mathcal W}_{ {\rm dP}_1}$. We will prove that this is indeed the case   in a forthcoming joint work with A.-M. Castravet.

\subsection{\bf Behavior of the ${\mathcal W}_{{\bf dP}_d}$'s in families and degeneracies}
\label{S:degen}
For $d\in \{2,3,4\}$ there is a $2(5-d)$-dimensional algebraic family 
$d{\mathcal P}_d$
of degree $d$ del Pezzo surfaces  giving rise to a family of webs 
${\mathscr W}_{{\rm dP}_d} \rightarrow d{\mathcal P}_d$ over it.  
First, it is interesting to ask how the webs $\boldsymbol{\mathcal W}_{{\rm dP}_d}$'s behave when ${\rm dP}_d$ varies in $d{\mathcal P}_d$.
\sk

Another kind of natural questions is  how the 
$\boldsymbol{\mathcal W}_{{\rm dP}_d}$'s may degenerate: there are modular compactifications 
$\overline{d{\mathcal P}}_d$ of the 
 moduli space $d{\mathcal P}_d$
 and it is natural to wonder first whether ${\mathscr W}_{{\rm dP}_d} \rightarrow d{\mathcal P}_d$ can be extended over $\overline{d{\mathcal P}}_d$, then if it is the case, what are the  properties (in particular those concerning its  ARs) of a web $\boldsymbol{\mathcal W}_{{\rm dP}'_d}$ on a degenerate del Pezzo surface 
 ${\rm dP}'_d\in \overline{d{\mathcal P}}_d \setminus d{\mathcal P}$. Since the 
 singular del Pezzo surfaces considered in \S\ref{S:singular}
are such degenerations, questions asked in the current subsection are generalizations of those of raised in the above subsection.

 In view of considering a very concrete example,  let us consider the webs on 
$\mathbf P^2$  
 defined by the 10 rational first integrals 
\eqref{Eq:WdP4-Ui},  denoted by 
$\boldsymbol{\mathcal W}_{{\rm dP}_4(\gamma,\pi)}$. Assuming that \eqref{Eq:PG-pi-gamma} is satisfied, each such web is a planar model of the 10-web by conics on a smooth del Pezzo quartic surface
which will be denoted by ${\rm dP}_4(\gamma,\pi)$.  Let us consider the naive specialization $\boldsymbol{\mathcal W}_{{\rm dP}_4,-1}
=\lim_{\pi,\gamma\rightarrow -1}\boldsymbol{\mathcal W}_{{\rm dP}_4(\gamma,\pi)}$
of the webs $\boldsymbol{\mathcal W}_{{\rm dP}_4(\gamma,\pi)}$'s: it is the web whose first integrals are the 
limits of all the $U_i$'s 
 in \eqref{Eq:WdP4-Ui}  when both parameters $\pi$ and $\gamma$ go to -1.
It is no longer a 10-web but the following 8-web 
$$\boldsymbol{\mathcal W}_{{\rm dP}_4,-1}
 =\boldsymbol{\mathcal W}\left(\, 
x\, , \,  
\frac{1}{y} \, , \,
 \frac{y}{x} \, , \,
  \frac{x - y}{x - 1}\, , \,
   \frac{x + 1}{x - y} \, , \,
    \frac{x - y}{(x - 1)(y + 1)} \, , \,
    \frac{y(x + 1)}{x(y + 1)} \, , \,
     \frac{x(y - 1)}{y(x - 1)}
  \,\right)\, . 
$$
One  verifies that this web has maximal rank 21 with all its ARs being generalized hyperlogarithms of weight 1 or 2.\footnote{Here, the adjective phrase `{\it generalized hyperlogarithmic'} refers to the notion of `{\it generalized iterated integrals'} considered in \cite[\S1.5.3]{ClusterWebs}.} 
   Denote by ${\bf Hlog}^3_{\gamma,\pi}$ the weight 3 hyperlogarithmic  functional identity carried by  $\boldsymbol{\mathcal W}_{{\rm dP}_4(\gamma,\pi)}$ for
$\gamma$ and $\pi$  generic (that is, satisfying \eqref{Eq:PG-pi-gamma}). From the description of the ARs of $\boldsymbol{\mathcal W}_{{\rm dP}_4,-1}$ just mentioned, it comes that ${\bf Hlog}^3_{\gamma,\pi}$ somehow disapears under the specialization $(\gamma,\pi)\rightarrow (-1,-1)$. It would be interesting to understand this phenomenon better and more generally to see how  the ARs of the generic 
web $\boldsymbol{\mathcal W}_{{\rm dP}_4(\gamma,\pi)}$ and those of the specialization $\boldsymbol{\mathcal W}_{{\rm dP}_4,-1}$ are related.

\subsection{\bf Webs by conics on real del Pezzo surfaces}
\label{S:singular}
Let $S$ be a  real del Pezzo surface with smooth complexification ${\rm dP}_d\subset \mathbf P^d$ (for $d\in \{2,3,4\}$ say).  
Then it is known that for some $S$, not all the lines contained in its complexification are defined over $\mathbf R$. In such a case, it is natural to consider the web 
$\boldsymbol{\mathcal W}_S$ 
by real pencils of conics on $S$, whose complexification is a proper subweb of $\boldsymbol{\mathcal W}_{{\rm dP}_d}$.  
 \begin{quote}
 {\it When the number of real lines contained in $S$ is strictly less than $l_r=\lvert \mathcal L_r\lvert$ (with $r=9-d$), what can be said about the real web $\boldsymbol{\mathcal W}_S$? In particular, does this web carry interesting (hyperlogarithmic?) abelian relations?}
\end{quote}

When $d=3$, it is known that the number  of real lines contained in $S$ can be anyone of the following numbers 3, 7, 15 or 27 and that each of them indeed occurs. This gives as many subcases of possibly interesting real webs by conics which would be worth studying.

\subsection{\bf Global single-valued versions}
\label{S:single-valued}
An interesting well-known feature of Abel's identity $\boldsymbol{\big(\mathcal Ab\big)}$
is that there is a global single-valued but real analytic version of it. More precisely, 
let $D$ stand for the {\it `Bloch-Wigner dilogarithm'}, which is  defined 
for any $z\in \mathbf P^1\setminus\{0,1,\infty\}$ 
by   
\begin{equation}
\label{Eq:Bloch-Wigner}
D(z)={\rm Im}\Big( {\bf L}{\rm i}_2(z) \Big)+{\rm Arg}\big(1-z\big)\cdot {\rm Log}\,\lvert \, z\, \lvert\, 
\end{equation}
where ${\rm Arg} : \mathbf C^*\rightarrow  ]-\pi, \pi]$ denotes   
the main branch of the complex argument.

The function $D$ is real-analytic and extends continuously to the whole Riemann sphere. Denoting again by $D$ this extension, one has $D(0)=D(1)=D(\infty)=0$  and the most remarkable of its features is that it satisfies the following global version of the identity $\boldsymbol{\big(\mathcal Ab\big)}$: one has 
 $$
D(x)-D(y)-D\bigg(\frac{x}{y}\bigg)-D\left(\frac{1-y}{1-x}\right)
+D\left(\frac{x(1-y)}{y(1-x)}\right)=0\,
$$
for any $x,y\in \mathbf P^1$ such that none of the five arguments of $D$ in this identity be indeterminate. \sk 

As we have considered them in this text, the hyperlogarithmic identities ${\bf HLog}^{r-2}$ $(r=4,\ldots,7$) are holomorphic, because the functions $AI^{r-2}_i$ involved in it are holomorphic.  On the other hand ${\bf HLog}^{r-2}$ is only locally satisfied, that because the $AI^{r-2}$'s extend to global but multivalued functions on the whole projective line.  
Remark that  ${\bf HLog}^{2}$ is a holomorphic version of Rogers dilogarithm $R$ which admits $D$ as a single-valued global cousin.  The point is that it has been proved in
 \cite[\S7]{Brown2004} (see also \cite[\S2.5]{VanhoveZerbini2018} or 
 \cite[\S3]{CDG2021}) that the most general hyperlogarithm $I$ admits a single-valued global version $I^{\rm sing}$.  Specializing this to the case of our antisymmetric hyperlogaritms $AI^{r-2}_i$ involved in 
 ${\bf HLog}^{r-2}$ gives 
 global single-valued functions $\widetilde{AI}^{\, r-2}_i$'s (for $i=1,\ldots, \kappa_r$) regarding which the following questions immediately arise for each $r=4,\ldots,7$: 
\begin{itemize}
\item[] ${}^{}$ \hspace{-0.8cm}1.  {\it can one give an explicit formula for $\widetilde{AI}^{r-2}_i$?
 }
\sk 
\item[] ${}^{}$ \hspace{-0.8cm}2. {\it do the $\widetilde{AI}^{\, r-2}_i$'s satisfy the global univalued version of ${\bf HLog}^{r-2}$, {\it i.e.}\,using the same notations  \\
${}^{}$ \hspace{-0.5cm}
 than in  Theorem \ref{Thm:main}, do we have 
$\sum_{i=1}^{\kappa_r} \epsilon_i \, \widetilde{AI}^{r-2}_i\big( U_i\big)=0$ identically 
on the whole $X_r$ ?}
\sk 
\end{itemize}
 
We believe that  for any $r\in \{4,\ldots,7\}$, the answer to 2. is positive, possibly up to considering suitable 
 single-valued versions of the antisymmetric hyperlogarithms $AI_i^{r-2}$. An approach for answering the questions might be to use Theorem 1.1 of  \cite[\S3]{CDG2021} but this would require first to build a `motivic lift' of the identity ${\bf HLog}^{r-2}$.

\subsection{\bf Motivic lifts}
In \cite{CDG2021}, the authors discuss motivic avatars of hyperlogarithms, some single-valued (still motivic) versions of these, classical realisations of them, as well as a motivic approach of the functional identities satisfied by the hyperlogarithms.  
Thanks to its algebraic nature ({\it cf.}\,\S\ref{SS-Key-Clef} where it is identified to a generator of the 
signature subrepresentation of the $W(E_r)$-module $\oplus_{\mathfrak c \in {\mathcal K}_r} \wedge^{r-2} \mathcal U_{\mathfrak c}$), we expect that ${\bf HLog}^{r-2}$  indeed can  be lifted into the `motivic world'. 
 It would be interesting to investigate this more rigourously. 

\subsection{\bf An approach 
 \textit{\textbf{\`a la}}
 Gelfand and MacPherson}
\label{SS:A conceptual interpretation?}

From our point of view, the representation-theoretic proof we give in sub-section 
\S\ref{SS-Key-Clef} that identity ${\bf HLog}^{r-2}$ holds true is somehow conceptual and uniform in $r$ but actually does not explain much about why this identity exists. 
\sk 

A remarkably nice construction of (the real form) of Abel's identity $\boldsymbol{\big(\mathcal Ab\big)}$ has been given by Gelfand and MacPherson in their wonderful paper \cite{GM}. In their approach:\sk
\begin{itemize}
\vspace{-0.3cm}
\item[$(i).$]
Bol's web is obtained as the quotient 
of a $H_5$-equivariant web on the grassmannian $G_2(\mathbf R^5)$ induced by 
five natural rational maps $\tilde f_i : G_2(\mathbf R^5)\dashrightarrow G_2(\mathbf R^4)$ 
of linear origin, where $H_5$ stands for the (diagonal) Cartan torus of ${\rm SL}_5(\mathbf R)$ (see \cite[\S1.1.3 and \S2.2.3]{GM}); 
\sk 
\item[$(ii).$] Abel's identity follows from several geometrical or homological facts: \sk
\begin{enumerate}
\item[1.] the closure of a generic $H_5$-orbit in  $G_2(\mathbf R^5)$ can be identified to a certain polytope,  the so-called hypersimplex $\Delta^2_3$ (see \cite[Thm.\,2.3.4]{GM});
\sk
\item[2.] a certain homological relation is satisfied between the class of the boundary of $\Delta^2_3$ and 
the homological classes of its facets when these are intrinsically considered (see \cite[Prop.\,2.1.3]{GM});\sk 
\item[3.]  the homological relation referred to above plus the version of Stokes' theorem for fibers integration give rise to a differential relation between the push-forwards under the action of the suitable Cartan tori, of the invariant 4-forms on $G_2(\mathbf R^5)$ and on $G_2(\mathbf R^4)$ representing the first Pontryagin class of the tautological bundles of these two grassmannian varieties (see \cite[Thm.\,1.2.2]{GM});
;\sk 
\item[4.] in the case under scrutiny, one of the members of the differential identity mentioned in 3.\,vanishes ({\it cf.}\,\cite[Thm.\,1.2.3]{GM}), which gives us after integration and normalization that the
sum of the five dilogarithmic terms in the left hand-side of $\boldsymbol{\big(\mathcal Ab\big)}$ vanishes identically.
\end{enumerate}
\end{itemize}
\sk

In \cite{WdP4}, we show that  all of the ingredients used by Gelfand and MacPherson to describe Bol's web $\boldsymbol{\mathcal B}\simeq \boldsymbol{\mathcal W}_{{\rm dP}_5}$ as a torus quotient of a $H_5$-equivariant web admit counterparts for the web $\boldsymbol{\mathcal W}_{{\rm dP}_4}$ (case $r=5$).  It is then natural to wonder whether 
the construction of $\boldsymbol{\big(\mathcal Ab\big)}\simeq {\bf HLog}^{2}$ mentioned in $(ii).$ just above can be generalized to the case of  ${\bf HLog}^{3}$ or not: 
\begin{itemize}
\item[] ${}^{}$ \hspace{-0.5cm} $\bullet$ {\it
does the identity ${\bf HLog}^{3}$, or a real version of it, can be constructed by means of the \\ Gelfand-MacPherson theory of `generalized dilogarithmic forms' developed in \cite{GM}?
 }
\end{itemize}

We believe that this question can be answered by the affirmative. If indeed true, we expect that this would furnish a conceptual explanation of why ${\bf HLog}^{3}$ holds true: in the case of ${\bf HLog}^{2}$, a key result is the homological identity about the hypersimplex 
$\Delta_m^n$ referred to in $(ii).2$. As mentioned in \cite[\S2.1.3]{GM}, this formula `{\it 
may be interpreted as computing the connecting
homomorphism for the triple $(\Delta_m^n, \partial \Delta_m^n, c)$ where $c$ is the codimension two skeleton of $\Delta_m^n$}'. We believe that something similar holds true for identity  ${\bf HLog}^{3}$, but in this case with respect to the polytope involved in the description {\it \`a la} Gelfand and MacPherson of $\boldsymbol{\mathcal W}_{ {\rm dP}_4}$, which is the $D_5$-moment polytope $\Delta_{D_5}$\footnote{This convex polytope is known by many other names: the weight polytope of type $D_5$, the 5-dimensional demihypercube or half measure polytope, Gosset's polytope  $1_{21}$ or even the demipenteract.}: we expect that in its algebraic (tensorial) form ({\it cf.}\,\S\ref{SS-Key-Clef}),  ${\bf HLog}^{3}$ may be interpreted as a kind of dual of a homological identity computing the connecting homomorphism for the triple $(\Delta_{D_5}, \partial \Delta_{D_5}, c_{D_5})$ where $c_{D_5}$ is the codimension two skeleton of $\Delta_{D_5}$. If this is indeed true, this  possibly could be extended to the other identities ${\bf HLog}^{r-2}$ for $r=4,\ldots,7$ as well.  We hope to clarify this in a future work.

\vspace{-0.2cm}
 \subsection{\bf Combinatorial characterization}
 \label{SS:Combinatorial-characterization}
In \S\ref{S:Web-WdP4} we discussed some nice properties of the webs $\boldsymbol{\mathcal W}_{{\rm dP}_5}$ and $\boldsymbol{\mathcal W}_{{\rm dP}_4}$, in particular the fact that both are characterized by their hexagonal 3-subwebs. But we did not explain the precise meaning of this. In this subsection, we consider this more generally and more rigorously and ask whether the characterization just mentioned extends to the other webs $\boldsymbol{\mathcal W}_{{\rm dP}_d}$'s.\sk

\vspace{-0.1cm}
Let us introduce general notions which we believe are relevant 
for studying webs. Let $\boldsymbol{\mathcal W}$ be a planar $d$-web, formed by $d$ distinct foliations $\mathcal F_1,\ldots,\mathcal F_d$ on its definition domain in $\mathbf C^2$.

For any $k=3,\ldots, d$, one sets/defines
\begin{itemize}
\vspace{-0.15cm}
\item ${\rm SubW}_k(\boldsymbol{\mathcal W})$ is the set of $k$-subwebs of  $\boldsymbol{\mathcal W}$;
\item ${\rm SubW}(\boldsymbol{\mathcal W})$ is the set of subwebs of  $\boldsymbol{\mathcal W}$ \big({\it i.e.}\,${\rm SubW}(\boldsymbol{\mathcal W})=\cup_{k=3}^d {\rm SubW}_k(\boldsymbol{\mathcal W})$\big);
\item $\boldsymbol{AR}_3(\boldsymbol{\mathcal W})$  denotes the subspace of $\boldsymbol{AR}(\boldsymbol{\mathcal W})$ spanned by the ARs of the 3-subwebs of $\boldsymbol{\mathcal W}$;  
\item ${r}_3(\boldsymbol{\mathcal W})$ stands for the dimension of $\boldsymbol{AR}_3(\boldsymbol{\mathcal W})$;  
\item  $r_{ \boldsymbol{\mathcal W}} : {\rm SubW}(\boldsymbol{\mathcal W}) \rightarrow \mathbf N$ is the function associating $r_3(\boldsymbol{W})$ to any subweb $\boldsymbol{W}$ of 
$\boldsymbol{\mathcal W}$;
\item $r_{ \boldsymbol{\mathcal W},k} : {\rm SubW}_k(\boldsymbol{\mathcal W}) \rightarrow \mathbf N$ is the   restriction of $r_{ \boldsymbol{\mathcal W}}$ to ${\rm SubW}_k(\boldsymbol{\mathcal W})$.
\end{itemize}

The functions $r_{ \boldsymbol{\mathcal W},k}$'s are combinatorial objects invariantly attached to  $\boldsymbol{\mathcal W}$. They allow to state in a nice and concise way some properties of webs.  For instance, let us consider the interesting case of Bol's web $\boldsymbol{\mathcal B}\simeq 
\boldsymbol{\mathcal W}_{{\rm dP}_5}$: since it is hexagonal, one has identically $r_{ \boldsymbol{\mathcal B},3}=1$ and $r_{ \boldsymbol{\mathcal B},4} =3$ and one verifies that  $r_{ \boldsymbol{\mathcal B},5}\equiv 5$.  On the other hand,  if $\boldsymbol{\mathcal L}$ is a planar  web formed by 5 pencils of lines, then $r_{ \boldsymbol{\mathcal L},3}\equiv 1$ (hexagonality), 
$r_{ \boldsymbol{\mathcal L},4} \equiv 3$ and $r_{ \boldsymbol{\mathcal L},5}\equiv 6$. Hence Bol's characterization of $\boldsymbol{\mathcal B}$ can be stated nicely as follows: 

\noindent {\bf Theorem.} (Bol \cite{Bol})
{\it Bol's web is characterized by $r_{ \boldsymbol{\mathcal B}}$: 
any 5-web $\boldsymbol{\mathcal W}$ such that $r_{ \boldsymbol{\mathcal W},3}\equiv 1$
and $r_{ \boldsymbol{\mathcal W},5}\leq 5$ is equivalent to Bol's web.}

This result shows that the  $r_{ \boldsymbol{\mathcal W},k}$'s can be useful to study webs up to equivalence. Given two planar $d$-webs $ \boldsymbol{\mathcal W}'$ and $ \boldsymbol{\mathcal W}''$ as above, let us say that the
functions 
 $r_{\boldsymbol{\mathcal W}'}$ and $r_{\boldsymbol{\mathcal W}''}$ (or 
$r_{\boldsymbol{\mathcal W}',3}$ and $r_{\boldsymbol{\mathcal W}'',3}$) are equivalent, denoted by $r_{\boldsymbol{\mathcal W}'}\sim r_{\boldsymbol{\mathcal W}''}$ (and similarly for $r_{\boldsymbol{\mathcal W}',3}$ and $r_{\boldsymbol{\mathcal W}'',3}$) if both coincide,  
possibly up to a relabeling of the foliations composing one of them.

In \cite{Burau}, Burau first determines $r_{\boldsymbol{\mathcal W}_{{\rm dP}_3},3}$  and then proves the following characterization result:

\noindent
 {\bf Theorem.} (Burau \cite{Burau})
{\it The class of 27-webs $\boldsymbol{\mathcal W}_{{\rm dP}_3}$'s is characterized by $r_{\boldsymbol{\mathcal W}_{{\rm dP}_3},3}$: any 27-web 
$\boldsymbol{\mathcal W}$ such that $r_{ \boldsymbol{\mathcal W},3}\sim r_{\boldsymbol{\mathcal W}_{{\rm dP}_3},3}$
 is equivalent to  the 27-web in conics on a cubic surface in $\mathbf P^3$.}

\noindent
Considering the two theorems above, one is naturally led to wonder if such results would not hold for all del Pezzo webs $\boldsymbol{\mathcal W}_{{\rm dP}_d}$. Here are a few questions regarding this that we find interesting: 
\begin{enumerate}
\item[1.] {\it A 3-subweb of $\boldsymbol{\mathcal W}_{{\rm dP}_d}$ corresponds to a triple
$\boldsymbol{\mathfrak c}$ of pairwise distinct elements of $ \boldsymbol{\mathcal K}_r$. Does the hexagonality of the associated 3-web $\boldsymbol{\mathcal W}(\boldsymbol{\mathfrak c})$ 
only depend on the $W_r$-orbit of $\boldsymbol{\mathfrak c}$ 
in $\big(\boldsymbol{\mathcal K}_r\big)^3$?}
\sk
\vspace{-0.4cm}
\item[2.] {\it Determine the functions $r_{ \boldsymbol{\mathcal W}_{{\rm dP}_d},k}$ for  $k$ small. When there are moduli for the considered del Pezzo surfaces ({\it i.e.}\,when $d\leq 4$), show that all of the $r_{ \boldsymbol{\mathcal W}_{{\rm dP}_d},k}$'s do not depend on ${\rm dP}_d$.} 
\sk
\item[3.] {\it Does $r_{ \boldsymbol{\mathcal W}_{{\rm dP}_d},3}$ determine the class of webs $ \boldsymbol{\mathcal W}_{{\rm dP}_d}$'s, {\it i.e.}\,is the following statement true:  `any $\kappa_r$-web $\boldsymbol{\mathcal W}$ such that 
$r_{ \boldsymbol{\mathcal W},3}\sim r_{ \boldsymbol{\mathcal W}_{{\rm dP}_d},3}$ is 
necessarily equivalent to a web  $\boldsymbol{\mathcal W}_{{\rm dP}_d}$'?} 
\end{enumerate}
The answers to these questions are known when $d=5$ (not to mention the trivial case $d=6$). Burau's theorem above gives an answer to 3.\,in the case when $d=3$.
The statement {\bf [\,Hexagonality \& Characterization\,]} 
mentioned in \S\ref{S:Web-WdP4-properties} means that we have verified that the answer to the third question for  $ \boldsymbol{\mathcal W}_{{\rm dP}_4}$ is affirmative as well. 
More generally, we expect that the answers to 1.\,and 3.\,are affirmative for any $d$. 
For the moment, the considered cases have been solved by ad hoc arguments. It would be interesting to have a uniform approach (independent on $d$).

 \subsection{\bf Other surfaces}
The setting we are working with in this paper can be very roughly described as follows: we deal with some surfaces (del Pezzo's in this case), each one carrying a finite number of fibrations of a fixed type, with the property that any component (a line in the case under scrutiny) of a non-irreducible fiber of one of these fibrations is also a component of a fiber of the same kind of another of the fibrations considered.
Stating things in that way makes us realise that such situations can be encountered for other kinds of algebraic surfaces than del Pezzo's.  We will discuss in a rather vague way 
the case of some elliptic projective surfaces. 
\sk 

Thanks to Kodaira, there is a classification of the analytic type 
of the possible singular fibers of an elliptic fibration on a projective surface $S$. In particular, 
there is only a finite number of possible singular fibers and all of these have rational  irreducible components. Then given a web $\boldsymbol{\mathcal W}_S$ formed by a finite number of elliptic fibrations  on $S$, all over $\mathbf P^1$ say, the components of their singular fibers form a finite set $\mathcal L_{\boldsymbol{\mathcal W}_S}$ of rational curves in $S$ and if  $L_{\boldsymbol{\mathcal W}_S}$ stands for the reduced effective divisor in $S$ whose components are the elements of $L_{\boldsymbol{\mathcal W}_S}$, it follows that 
$\boldsymbol{\mathcal W}_S$ is singularity free on the complement of 
$L_{\boldsymbol{\mathcal W}_S}$ in 
$S$. Given an elliptic fibration $\pi : S\rightarrow \mathbf P^1$ in $\boldsymbol{\mathcal W}_S$, let 
$\sigma_\pi$ be the set of its singular values ({\it i.e.} $\sigma_\pi=\{\, \lambda \in \mathbf P^1\, \lvert \, \pi^{-1}(\lambda) \mbox{  is singular}\,\}$).  
We denote by $H_\pi$ the space of global logarithmic 1-forms on $\mathbf P^1$ with poles in $\sigma_\pi$ and we define $H_{\boldsymbol{\mathcal W}_S}$ as the space of rational 1-forms on $S$ spanned by the pull-backs $\pi^* (H_\pi)$ for all $\pi$ in 
$\boldsymbol{\mathcal W}_S$. We can play now the same game as the one played successfully in this paper with del Pezzo's surfaces: for a given weight $w>0$, we are interested in  the kernel of the linear map $\oplus_{\pi \in \boldsymbol{\mathcal W}_S} \pi^* \big((H_\pi)^{\otimes w}\big)\rightarrow (H_{\boldsymbol{\mathcal W}_S})^{\otimes w}$ since any element in it will give rise to a weight $w$ hyperlogarithmic AR for 
$\boldsymbol{\mathcal W}_S$.

An interesting class of surfaces to which the approach just sketched could apply is that of projective $K3$ surfaces with a finite automorphism group.
Surfaces of this kind have been classified by different authors (see \cite{Roulleau} and the reference therein) and any of them admits only a finite number of elliptic fibrations (by \cite[Cor.\,2.7]{Sterk}).
Given such a $K3$ surface $X$, 
 the finite set of elliptic fibrations on it (up to $\mathcal J_0(X)$-equivalence, see \cite{BKW}) gives rise to 
the `{\it elliptic web}' of $X$, denoted by $\boldsymbol{\mathcal E \hspace{-0.05cm}\mathcal W}_X$ which may carry interesting hyperlogarithmic abelian relations.  \sk

Note that K3 surfaces share with del Pezzo surfaces similar features relatively to their Picard lattice: in both cases, there is a notion of `root', of `root system', with an associated Weyl group, etc. ({\it e.g.}\,see \cite[\S4]{BL}).  
 All this suggests that elliptic webs on K3 surfaces, or possibly just some of them, may carry interesting ARs. We are not aware of any results in this direction for now, everything remains to be done.

\subsection{\bf The hyperlogarithms $AI^{r-2}$ and antisymmetization in the realm of polylogarithmic functions.}
As far as we know, the totally antisymmetrized hyperlogarithms $AI^{r-2}$ (defined for any $r\geq 4$) have  already been  considered  in the mathematical literature
only in the case when $r=4$, in which case one recovers (a holomorphic version of)  Rogers' dilogarithm (see \S\ref{SSS:lalala}  above).
Indeed,  $AI^{2}$ is nothing but a holomorphic version (localized at the chosen base point)  of Rogers' dilogarithm $R$ hence it is polylogarithmic in the 
most basic way. It follows from 
\eqref{F:Asym-s} that $AI^{3}$ and $AI^{4}$ can be expressed in terms of (holomorphic) polylogarithmic functions (see also \eqref{Eq:AI-z-3} for $AI^{3}$).  These remarks lead to wonder about the other cases: 
\begin{quote}
{\it  $\bullet $\,  for $r\geq 5$, can $AI^{r-2}$ be expressed by means of classical polylogarithms?} 
\end{quote}
\sk

We are aware of only a few examples where  antisymmetrizations do occur 
in the realm of  (multivariable) polylogarithms:\footnote{Antisymmetrizations  also occur in the realm of Nielsen polylogarithms, see \cite{CGR}.}
\sk

\begin{itemize}
\item[$-$] first in  \cite{Goncharov} (for $n=3$) then in 
\cite{GoncharovSimple}  for any $n\geq 3$, Goncharov defines   what he calls the   $n$-th {\it grassmannian polylogarithm} $L_n^G$ by means of antisymmetrizations (see Definition 4.1 of \cite{GoncharovSimple}).  For any $n\geq 3$, the function 
$L_n^G$ is a holomorphic but multi-valued function defined on 
the moduli space ${\rm Conf}_{2n}^0(\mathbf P^{n-1})$
of generic configurations of $2n$ points in $\mathbf P^{n-1}$. For suitable choices of the branches of $L_n^G$ involved,  the 
following identity is (locally) 
satisfied   for 
any $\boldsymbol{y}=[y_1,\ldots,y_n]\in {\rm Conf}_{2n+1}^0(\mathbf P^{n-1})$: 
$$
\sum_{i=1}^{2n+1} (-1)^i L_n^G\big(  \boldsymbol{y}_{\hat \imath}
\big) =0
$$
where $\boldsymbol{y}_{\hat \imath}$ stands for $[y_1,\ldots, \widehat{y_i},\ldots , y_{2n}]\in {\rm Conf}^0_{2n}(\mathbf P^{n-1})$ for any $i=1,\ldots,2n+1$;\sk 
\item[$-$]  in the recent preprints  \cite{Rudenko} and \cite{MR}, the authors define what they call `{\it quadrangular polylogarithms}' ${\rm QLi}_n$ and   `{\it cluster grasmmannian polylogarithms}' ${\rm GLi}_n$ respectively, again using some antisymmetrizations, but  partial ones now (see formulas (5.1)  and (4.12) in \cite{Rudenko} and \cite{MR} respectively).  From the point of view we are interested in in this text, the  cluster grassmannian polylogarithms are most interesting since 
for any $n\geq 2$, ${\rm GLi}_n$ satisfies an interesting functional equation (see \cite[Theorem 1.5.(4)]{MR}). 
\end{itemize}

These results suggest (at least to the author of these lines) the following question: 

\begin{quote}
{\it  $\bullet $\,  for $n\geq 2$,  are Goncharov's grassmannian polylogarithm $L_n^G$, 
 Matveiakin-\\ ${}^{}$ \hspace{0.1cm} Rudenko cluster grassmannian polylogarithm ${\rm GLi}_n$ and the single-variable\\ 
 ${}^{}$ \hspace{0.1cm} hyperlogarithm $AI^{n}$ considered in this text, related in any way? } 
\end{quote}
The answer to this question being obvious (and affirmative) for $n=2$, the first case to be considered is the one when $n=3$.  For instance, one can ask whether $AI^{3}$ can be expressed by means of $L_3^G$ or of ${\rm GLi}_3$ or not.

\subsection{\bf Curvilinear webs on Fano blow-ups of projective spaces.}
 The planar webs considered in this paper arise quite naturally when considering blow-ups of the projective plane which are Fano. If one is interested by possible higher dimensional generalizations of these webs, it is natural, in a first step at least,  to consider the blow-ups of projective spaces of dimension bigger than 2 which are Fano manifolds again. 
 
 In this subsection, we would like to discuss two cases, one which is classical and has been studied recently, for which most of things are clear, a second one which looks very similar but which has not been considered beyond the two dimensional case yet.
 The former case is that of the blow-ups of $\mathbf P^n$ in $n+2$ points, the latter in $n+3$ points, this for any $n\geq 2$. \sk 
 
We fix $n\geq 2$ and a  subset $P\subset \mathbf P^n$ of  $N$ points $p_1,\ldots,p_N$ in general position, with $N=n+2$ or $n+3$.  We denote by $\boldsymbol{\mathcal W}_P$ the web by rational curves on $\mathbf P^n$ formed 
by the following foliations: 
\begin{enumerate}
\item[$-$] the linear radial foliations $\mathcal L_{p_i}$ of lines passing through $p_i$, for $i=1,\ldots,N$;
\sk
 \item[$-$] the congruences $\mathcal C_J$ of degree $n$ rational normal curves (RNCs) passing through the $p_j$'s for $j\in J$, for any $J\subset \{1,\ldots,N\}$ of cardinality $n+2$. 
\sk
\end{enumerate}
 
 When $N=n+2$, then 
 $$\boldsymbol{\mathcal W}_P=\boldsymbol{\mathcal W}\Big(\mathcal L_{p_1},\ldots,\mathcal L_{p_{n+2}}, \mathcal C \Big)$$
  where $\mathcal C$ stands here for the congruence of degree $n$ RNCs passing through all the $p_i$'s. In particular, 
 $\boldsymbol{\mathcal W}_P$ is a $(n+3)$-web for $n$ arbitrary and  it is equivalent to 
 $\boldsymbol{\mathcal W}_{ {\rm dP}_5}\simeq 
 \boldsymbol{\mathcal B}
 $  when $n=2$. 
 
 In case $N=n+3$, let $\mathcal C_{\widehat \imath}$ stand for the congruence 
 associated to $\{1,\dots,n+3\}\setminus \{ i\}$ for $i=1,\ldots,n+3$. Then one has 
$$\boldsymbol{\mathcal W}_P=\boldsymbol{\mathcal W}\Big(\, \mathcal L_{p_1},\ldots,\mathcal L_{p_{n+2}}, \mathcal C_{\widehat 1},  \ldots,  \mathcal C_{\widehat{n+3}} \,  \Big)$$
 which thus is a $2(n+3)$-web. In particular when $n=2$, one recovers (a birational model of) 
  $\boldsymbol{\mathcal W}_{ {\rm dP}_4}$.

\subsubsection{The curvilinear webs $\boldsymbol{\mathcal W}_P$ when $N=n+2$}
When $N=n+2$ (for any $n\geq 2$), then the web $\boldsymbol{\mathcal W}_P$ is a birational model of the web $\boldsymbol{\mathcal W}_{\mathcal M_{0,n+3}}$ defined on the moduli space 
$\mathcal M_{0,n+3}$ by the $n+3$ maps $f_i:\mathcal M_{0,n+3} \rightarrow \mathcal M_{0,n+2}$ induced by forgetting the $i$-th point (for $i=1,\ldots,n+3$). 

These webs have been studied by a few classical authors and more recently by Damiano and in the paper \cite{PirioM0n+3} (see the references therein).  Building on Gelfand-MacPherson theory of `generalized dilogarithmic forms', Damiano investigated the $(n-1)$-abelian relations of $\boldsymbol{\mathcal W}_{\mathcal M_{0,n+3}}$ for any $n\geq 2$, and claimed that the structure of the corresponding space of ARs $\boldsymbol{AR}^{(n-1)}
\big( \boldsymbol{\mathcal W}_{\mathcal M_{0,n+3}} \big) $ 
was for $n$ arbitrary the most straightforward generalization one can think of that of $\boldsymbol{\mathcal W}_{\mathcal M_{0,5}}\simeq \boldsymbol{\mathcal B}$ ({\it cf.}\,\S\ref{S:Web-WdP5-properties} above). In \cite{PirioM0n+3} it is proved that Damiano's arguments and claims are flawed when $n$ is odd and a correct way to handle this case is given.

The situation for $n\geq 2$ arbitrary is similar by many aspects to that of del Pezzo's surfaces considered in this paper: if $X_P$ denotes the blow-up of $\mathbf P^n$ at the $p_i$'s, then ${\bf Pic}(X_P)\simeq \mathbf Z^{n+3}$ and one can define `roots'  which are `$(-1)$-divisors' and which form a root system of Dynkin type $A_{n+2}$ in the orthogonal 
 $(K_{X_P})^\perp \subset {\bf Pic}(X_P)$ ({\it cf.}\,Section 2 of \cite{CT}).  The associated  Weyl group $W_P\simeq W(A_{n+2})$  in this case is naturally isomorphic 
 to ${\rm Aut}\big( {\mathcal M_{0,n+3}}\big) \simeq \mathfrak S_{n+3}$ and acts naturally  on  $\boldsymbol{AR}^{(n-1)}
\big( \boldsymbol{\mathcal W}_{\mathcal M_{0,n+3}} \big) $.  And one of the main result about the latter space concerns its structure as a $\mathfrak S_{n+3}$-module.\footnote{See \cite{PirioM0n+3} for more details.}  
One denotes by $\boldsymbol{AR}^{(n-1)}_C
\big( \boldsymbol{\mathcal W}_{\mathcal M_{0,n+3}} \big) $ the subspace generated by the 
$(n-1)$-ARs of the $(n+1)$-subwebs of $\boldsymbol{\mathcal W}_{\mathcal M_{0,n+3}}$. 
It is clearly a subrepesentation of the whole space $\boldsymbol{AR}^{(n-1)}
\big( \boldsymbol{\mathcal W}_{\mathcal M_{0,n+3}} \big) $. Using Gelfand-MacPherson theory and following Damiano, one constructs a special abelian relation $\boldsymbol{\mathcal E}_n \in \boldsymbol{AR}^{(n-1)}
\big( \boldsymbol{\mathcal W}_{\mathcal M_{0,n+3}} \big) $ from the ($n$-th wedge power of the) ${\rm SO}_{n+3}(\mathbf R)$-invariant representative of the Euler class of the tautological bundle on the oriented grassmannian $G_2^{or}\big(\mathbf R^{n+3}\big)$. 
Then the following dichotomy holds true, according to the parity of $n$: 
\sk
\begin{itemize}
\item[$\bullet$] when $n$ is odd, one has $\boldsymbol{\mathcal E}_n \in \boldsymbol{AR}^{(n-1)}
\big( \boldsymbol{\mathcal W}_{\mathcal M_{0,n+3}} \big) =
\boldsymbol{AR}^{(n-1)}_C
\big( \boldsymbol{\mathcal W}_{\mathcal M_{0,n+3}} \big) 
$ and the latter space is an irreducible  $\mathfrak S_{n+3}$-representation, with Young symbol $\big[31^n\big]$;
\sk
\item[$\bullet$] when $n$ is even
\begin{itemize}
\item there is a decomposition in direct sum of 
$\mathfrak S_{n+3}$-irreps 
$$ \boldsymbol{AR}^{(n-1)}
\big( \boldsymbol{\mathcal W}_{\mathcal M_{0,n+3}} \big) =
\boldsymbol{AR}^{(n-1)}_C
\big( \boldsymbol{\mathcal W}_{\mathcal M_{0,n+3}} \big) 
\oplus \big\langle  \boldsymbol{\mathcal E}_n  \big\rangle 
$$ with corresponding Young symbols $\big[221^{n-1}\big]$ and $\big[1^{n+3}\big]$;
\sk 
\item at least for $n\leq 12$ (but conjecturally for any $n$ even) there is an explicit formula for the component of $\boldsymbol{\mathcal E}_n$, involving only rational expressions and logarithms of rational expressions ({\it cf.}\,
 \cite[Prop.\,5.28]{PirioM0n+3}). In short: $\boldsymbol{\mathcal E}_n$ is a  polygarithmic AR of weight 1.\footnote{
 \label{foofoo}When $n=2$,  the weight 2 polylogarithmic identity ${\bf HLog}^2\simeq \boldsymbol{\big(\mathcal Ab\big)}$ and the weight 1 abelian relation $\boldsymbol{\mathcal E}_2$ are related as follows: the latter corresponds to the differential identity obtained by taking the total derivative of the former functional equation.}
\end{itemize} 
\mk
\item[$\bullet$] in both cases ({\it i.e.}\,regardless of the parity of $n$), the $(n-1)$-rank of $  \boldsymbol{\mathcal W}_{P} \simeq \boldsymbol{\mathcal W}_{\mathcal M_{0,n+3}}$ is equal to $(n+1)(n+2)/2$ hence this web has maximal rank.
\end{itemize}
One of the key ingredients to obtain the results above is the action of the Weyl group $W(A_{n+2})$ on the space of ARs of the web under scrutiny. All this makes  of the 
$\boldsymbol{\mathcal W}_{\mathcal M_{0,n+3}}$'s natural and meaningful generalisations of the del Pezzo conical web $\boldsymbol{\mathcal W}_{{\rm dP}_5}$, especially when $n$ is even.

\subsubsection{The curvilinear webs $\boldsymbol{\mathcal W}_P$ when $N=n+3$}
Considering the previous discussion about the case when $N=n+2$, one can wonder about the case when $N+3$, and more precisely whether the $\boldsymbol{\mathcal W}_P$'s would not satisfy natural generalizations of (some of ) the nice properties 
enjoyed by $\boldsymbol{\mathcal W}_{{\rm dP}_4}$ listed in \S\ref{S:Web-WdP4-properties}. 
\sk

The setting for any $n\geq 2$ appears as a natural generalisation of that of ${\rm dP}_4$ (which corresponds to the case when $n=2$).  As in the previous case, 
setting $X_P={\bf Bl}_P(\mathbf P^n)$, the $(-1)$-divisors give rise to a 
root system, now of Dynkin type $D_{n+3}$, in the orthogonal of the canonical class in 
 ${\rm Pic}(X_P)\simeq \mathbf Z^{n+4}$. Moreover, there is a natural action of the Weyl group $W(D_{n+3})$ on  the Picard lattice of 
  $X_P$ (see \cite{CT} and more specifically 
\cite{SturmfelsVelasco}).  It is then natural to ask the following questions: 
\begin{itemize}
\item[1.] {\it Does the Weyl group $W(D_{n+3})$ act naturally on the space $ \boldsymbol{AR}^{(n-1)}
\big( \boldsymbol{\mathcal W}_{P} \big)$ of $(n-1)$-ARs of 
$\boldsymbol{\mathcal W}_{P}$ (or otherwise on a certain subspace of it) ?} 
\sk
\item[2.] {\it If the answer to the first question is affirmative, what is the decomposition 
of 
$ \boldsymbol{AR}^{(n-1)}
\big( \boldsymbol{\mathcal W}_{P} \big)$
into irreducible $W(D_{n+3})$-representations?}
\sk
\item[3.] {\it What is the dimension of $ \boldsymbol{AR}^{(n-1)}
\big( \boldsymbol{\mathcal W}_{P} \big)$ ? Does this web has maximal $(n-1)$-th rank ?}
\sk
\end{itemize}

We do not have answers to these questions for $n\geq 3$ yet.  

In this case (namely $N=n+3$) as well, there might exist a dichotomy depending on the parity of $n$. This is suggested by the very nice results of \cite{AC}, which concern the case when $n$ is even.  In this case, we set $m=n/2\geq 1$ and consider   a projective parametrization  $\Lambda: \mathbf P^1\rightarrow \mathbf P^n$  of the unique degree $n$ rational normal curve passing  through $p_1,\ldots,p_{n+3}$. 
Let $\mathbf C=\mathbf P^1\setminus \{\infty \}$ be an affine chart 
such that $P\subset \Lambda(\mathbf C)$ and set $\lambda_i=\Lambda^{-1}(p_i)\in \mathbf C$ for $i=1,\ldots,n+3$. Let $Z$ be the complete intersection 
in $\mathbf P^{n+2}$ of the two quadrics cut out by the equations $\sum_{i=1}^{n+3} x_i^2=0$ and $\sum_{i=1}^{n+3} \lambda_i \, x_i^2=0$ and consider the variety $G=F_{m-1}(Z)\subset G_{m-1}\big( \mathbf P^{n+2}\big)$ of $(m-1)$-planes contained in $Z$.   

In \cite{AC}, the following statements have been proved to be true: 
\begin{enumerate}
\item[$(i).$] there are pseudo-isomorphisms $f: G\dashrightarrow X_P$, that is birational maps inducing isomorphisms in codimension 1 (and there exist precisely $2^{n+2}$ such pseudo-isomorphisms); 
\sk
\item[$(ii).$] $G$ contains $2^{n+2}$ subspaces of dimension $m$ which are linear (that is give rise to $m$-planes in the image) relatively to the 
Pl\"ucker embedding 
 $G\hookrightarrow  \mathbf P\big(  \wedge^m \mathbf C^{n+3}\big)$;
\sk
\item[$(iii).$] given a pseudo-isomorphism $f:  G\dashrightarrow X_P$, the pull-back 
$f^*\big( \boldsymbol{\mathcal W}_{P} \big)$  admits as first integrals
$2(n+3)$ morphisms onto rational varieties $\varphi_i: G\rightarrow T_{\varphi_i}$  and $\psi_i: G\rightarrow T_{\psi_i}$ (with $i=1,\ldots,n+3$) which all satisfy the following properties: 
\begin{itemize}
\item[$-$] any smooth fiber is isomorphic to $\mathbf P^1$; 

\item[$-$] there are exactly  $2^{n}$ singular fibers, each of which has two irreducible components, namely two 
$m$-planes in $G$ ({\it cf.}\,$(ii)$ above) intersecting transversely in one point. 
\end{itemize}
\sk
\end{enumerate}

All the properties above together 
are generalizations in higher but even dimension of the properties of ${\rm dP}_4$ which were useful regarding  the identity ${\bf HLog}^3$.  As the identity ${\bf HLog}^2\simeq \boldsymbol{\big(\mathcal Ab\big)}$ admits a generalization in any even dimension (the so-called Euler's abelian relation $(\boldsymbol{\mathcal E}_n)$, see \cite{PirioM0n+3} and footnote \footref{foofoo} above), we find natural to ask the following: 
\begin{itemize}
\item[4.] {\it Can ${\bf HLog}^3$ be generalized to the case of 
$\boldsymbol{\mathcal W}_{P}$ when $N=n+3$ is odd? 
More precisely, in this case, does 
this web admit a $(n-1)$-AR  whose components are weight 2 hyperlogarithms? If yes, does this AR  corresponds to a symbolic tensorial identity which 
is stable up to sign under the action of $W(D_{n+3})$ (hence 
spans 
 the $W(D_{n+3})$-signature representation)?} 
\sk
\end{itemize}

\vspace{-0.6cm}
\subsection{\bf Curvilinear webs on Mori dream spaces.}
The considerations above lead us to notice that any Mori dream space (ab.\,MDS) 
naturally carries a curvilinear web (by rational curves?). One can ask about the properties of such webs. 

Among the simplest examples of non Fano Mori dream spaces obtained by blowing-up projective spaces, one can mention the following ones:

\begin{itemize}
\item  the blow-up of $\mathbf P^2$ at nine points lying on a smooth conic;
\sk
\item the blow-up of $\mathbf P^3$ at seven points in general position; 
\sk
\item  the Losev-Manin space obtained by  two successive blow-ups: first  four points in general position in $\mathbf P^3$, then  the  strict transforms of the six lines passing through two of the considered points. 
\end{itemize}
As references about these examples, see \cite{CT} for the first two and \cite[\S8]{Castravet} for the third. 

Let $X$ be one of the three MDS above and let $\boldsymbol{\mathcal W}_X$ be a natural web by rational  curves on it.\footnote{We are quite vague about what has to be considered as `natural' here. These are  the webs on $X$ the foliations of which are families of (necessarily rational?) curves corresponding to generators of extremal classes in the cone of movable classes on $X$.
We hope to come back to this  with more details in a future work.} 
For any $k=0,\ldots,\dim(X)-1$, one can ask the following question:
\begin{itemize}
\item[$\bullet$] {\it What can be said about the  $k$-abelian relations of 
$\boldsymbol{\mathcal W}_{X}$? Does this web carry hyperlogarithmic or even better, polylogarithmic, $k$-ARs?} 
\end{itemize}
Curvilinear webs on Fano manifolds or more generally on MDS may be very interesting webs. It is our opinion  that they would deserve being investigated more in depth.

\section*{\bf Appendix: the structure of $\mathbf C^{{\mathcal L}_r}$ as a $W(E_r)$-module}
\label{Appendix}
In this appendix we briefly explain the computations we used to establish some representation-theoretic facts about the 
$W(E_r)$-module $\mathbf C^{{\mathcal L}_r}$. 
These computations were 
mainly performed using the software GAP3 \cite{GAP3} which has the advantage of being available online in CoCalc.\footnote{One can run GAP3 
by typing the command {\ttfamily  /ext/bin/gap3} in a \href{https://cocalc.com/features/terminal}{CoCalc terminal window}.}  
For the sake of clarity and conciseness, we will essentially discuss the case $r=5$, as all the other ones can be treated in a completely analogous way.

\subsection*{The case $r=5$}
Our aim here is to explain how, using GAP,  it is easy  
 to determine the character $\chi_5$ of the $W(D_5)$-module $\mathbf C^{{\mathcal L}_5}$ explicitly and then to perform the computations used in Section \S\ref{S:Main}.  
\sk

\vspace{-0.4cm}
First, in order to define $W(D_5)$ and its character table in GAP, we run the following commands: 

\begin{lstlisting}[language=GAP]
gap> WD5:=CoxeterGroup("D",5);;
gap> TableWD5:=CharTable(WD5);
gap> Display(TableWD5);
\end{lstlisting}

From the outcome of the last command, one can extract the  character \ref{Table:CharTableWD5} (in which the dot $\cdot$ is used instead of $0$ to make its reading easier). 

\begin{table}[!h]
\resizebox{5.8in}{1.5in}
{
\begin{tabular}{|l||c|c|c|c|c|c|c|c|c|c|c|c|c|c|c|c|c|c|}
\hline
         & \scalebox{0.8}{$\boldsymbol{\big(1^5.\big)}$} &  \scalebox{0.8}{$\boldsymbol{(1^3.1^2)}$}  & 
          \scalebox{0.8}{$\boldsymbol{\big(1.1^4\big)}$ } & 
       \scalebox{0.8}{   $\boldsymbol{\big(21^3.\big)}$ } & 
       \scalebox{0.8}{   $\boldsymbol{\big(1^2.21\big)}$} & 
       \scalebox{0.8}{   $\boldsymbol{\big(21.1^2\big)}$} & 
       \scalebox{0.8}{   $\boldsymbol{\big(.21^3\big)}$ } & 
      \scalebox{0.8}{    $\boldsymbol{\big(221.\big)}$}  &
      \scalebox{0.8}{    $ \boldsymbol{\big(1.22\big)}$ } &  
     \scalebox{0.8}{     $\boldsymbol{\big(2.21\big)}$} &  
     \scalebox{0.8}{     $\boldsymbol{\big(311.\big)}$ }  & 
     \scalebox{0.8}{     $\boldsymbol{\big(1.31\big)}$} & 
      \scalebox{0.8}{    $\boldsymbol{\big(3.11\big)}$}  & 
     \scalebox{0.8}{     $\boldsymbol{\big(32.\big)}$}  & 
    \scalebox{0.8}{      $\boldsymbol{\big(.32\big)}$} & 
    \scalebox{0.8}{      $\boldsymbol{\big(41.\big)}$} & 
    \scalebox{0.8}{      $\boldsymbol{\big(.41\big)}$} &   
    \scalebox{0.8}{      $ \boldsymbol{\big(5.\big)}$}
  \\
      \hline \hline 
$\boldsymbol{[1^2.1^3]}$     &   10   &  -2   &   2  &   -4   &   2   &   .  &   -2  &    2  &   -2&  .    &  1 &    -1  &    1  &  -1  &   1  &  .  &  .   &   . \\
 \hline
$\boldsymbol{[1.1^4]}$    &     5    &  1  &   -3  &   -3   &  -1   &   1   &   3   &   1   &   1&   -1    &  2  &    .     &-2   &  .  &   .  & -1 &   1   &   .   \\ \hline
$\boldsymbol{[.1^5 ]}$   &     1  &    1  &    1   &  -1  &   -1  &   -1  &   -1   &   1  &    1 & 1 &     1   &   1   &   1    &-1   & -1  & -1  & -1   &   1 \\ \hline
$\boldsymbol{[1^3.2]}$    &     10 &    -2   &   2   &  -2   &   .    &  2 &    -4  &   -2  &    2  &  .  &    1  &   -1  &    1   &  1  &  -1  &  .  &  .  &    .
\\ \hline
$\boldsymbol{[1^2.21]}$    &     20   &  -4    &  4   &  -2    &  2  &   -2    &  2   &   .   &   .  &  .   &  -1   &   1  &   -1  &   1  &  -1  &  .  &  .  &    .
\\ \hline
$\boldsymbol{[1.21^2]}$    &     15  &    3  &   -9    & -3    & -1   &   1    &  3  &   -1   &  -1  & 1  &    .   &   .   &   .   &  . &    .  &  1 &  -1  &    .
\\ \hline
$\boldsymbol{\big[.21^3\big]}$    &      4  &    4   &   4  &   -2  &   -2  &   -2  &   -2 &     .   &   .   &   .   &   1    &  1      & 1   &  1   &  1  &  . &   .    & -1 
\\ \hline
$\boldsymbol{\big[1.2^2\big]}$      &    10  &    2   &  -6   &   .   &   .     & .  &    .   &   2    &  2   &   -2 &    -2   &   . &     2   &  . &    . &   . &   . &     .
\\ \hline
$\boldsymbol{\big[2.21\big]}$      &    20  &   -4  &    4   &   2  &   -2  &    2   &  -2 &     .  &    .  &  .   &  -1   &   1   &  -1  &  -1    & 1  &  .  &  .   &   . \\
 \hline
$\boldsymbol{\big[.2^21\big]}$    &       5  &    5    &  5  &   -1  &   -1   &  -1  &   -1   &   1   &   1 &    1  &   -1  &   -1   &  -1  &  -1   & -1  &  1  &  1  &    . \\
 \hline
$\boldsymbol{\big[1^2.3\big]}$   &      10  &   -2  &    2   &   2   &   .  &   -2  &    4  &   -2  &    2&    .  &    1   &  -1   &   1 &   -1  &   1  &  . &   .  &    .
\\
 \hline
$\boldsymbol{\big[1.31\big]}$     &     15   &   3   &  -9   &   3   &   1  &   -1  &   -3   &  -1  &   -1&1   &   .   &   .    &  .    & .   &  . &  -1 &   1  &    . 
\\
 \hline
$\boldsymbol{\big[.31^2\big]}$      &     6  &    6  &    6   &   .    &  .    &  .   &   .  &   -2   &  -2&    -2   &   .   &   .    &  .  &   .  &   .   & . &   .  &    1
\\
 \hline
$\boldsymbol{\big[2.3\big]}$       &    10  &   -2  &    2   &   4   &  -2  &    .  &    2  &    2   &  -2&  .   &   1  &   -1   &   1   &  1 &   -1   & .  &  .   &   .
\\
 \hline
$\boldsymbol{\big[.32\big]}$    &        5   &   5  &    5   &   1   &   1    &  1   &   1   &   1  &    1&      1   &  -1  &   -1 &    -1  &   1  &   1 &  -1  & -1  &    .
\\ \hline
$\boldsymbol{\big[1.4\big]}$     &       5   &   1  &   -3    &  3   &   1  &   -1   &  -3   &   1  &    1&  -1   &   2   &   .     &-2  &   .  &   .  &  1 &  -1  &    .
\\ \hline
$\boldsymbol{\big[.41\big]}$   &         4  &    4  &    4    &  2    &  2   &   2  &    2  &    .    &  .&  .    &  1  &    1   &   1  &  -1  &  -1  &  .  &  .  &   -1
\\ \hline
$\boldsymbol{\big[.5\big]}$     &        1   &   1   &   1    &  1   &   1     & 1     & 1     & 1   &   1&
 1   &   1  &    1  &    1 &    1     &1   & 1  &  1    &  1
 \\ \hline 
\end{tabular}}
\bk
\caption{Character table of the Weyl group of type $D_5$.}
\label{Table:CharTableWD5}
\end{table}

A few explanation are in order regarding this table. As is well-known ({\it e.g.} see \cite[\S5.6]{GP}), for any odd integer $n\geq 4$,\footnote{The case when $n$ is even is a bit more involved but fully understood as well.} the characters as well as the conjugacy classes of the Weyl group of type $D_n$ can be labelled by means of the `bipartitions  of $n$', that is pairs $(\lambda,\lambda')$ of partitions such that 
$\lvert \lambda\lvert+\lvert \lambda'\lvert=n$. 
In Table \ref{Table:CharTableWD5},  we use  
$[\lambda.\lambda']$ for labelling the line associated to the corresponding character  while 
$(\lambda.\lambda')$ serves to label the column associated to the corresponding conjugacy class. When $\lambda $ is the empty partition, we just write $[.\lambda']$  and $(.\lambda')$ respectively and similarly when $\lambda'=\emptyset$. 

In view of determining $\chi_5$ as explained in the proof of Proposition \ref{Prop:Key}, we need to construct explicit representatives of the conjugacy classes in $W(D_5)$. This can be achieved in GAP3 by means of the following command: 
\begin{lstlisting}[language=GAP]
gap> List(ConjugacyClasses(WD5),x->CoxeterWord(WD5,Representative(x)));
\end{lstlisting}
%
%
which returns the following string of 18 sequences: 
{\begin{lstlisting}[language=GAP]
[ [  ], [ 1, 2 ], [ 1, 2, 3, 1, 2, 3, 4, 3, 1, 2, 3, 4 ], [ 1 ], [ 1, 2, 3 ], 
  [ 1, 2, 4 ], [ 1, 2, 3, 1, 2, 3, 4, 3, 1, 2, 3, 4, 5 ], [ 1, 4 ], 
  [ 1, 3, 1, 2, 3, 4 ], [ 1, 2, 3, 5 ], [ 1, 3 ], [ 1, 2, 3, 4 ], 
  [ 1, 2, 4, 5 ], [ 1, 3, 5 ], [ 1, 3, 1, 2, 3, 4, 5 ], [ 1, 4, 3 ], 
  [ 1, 2, 3, 4, 5 ], [ 1, 4, 3, 5 ] ]
\end{lstlisting}}
The $k$-th element of the above list corresponds to an explicit representative of the $W(D_5)$-conjugacy class associated to the $k$-th column of Table \ref{Table:CharTableWD5} in terms of the generators of {\ttfamily  WD5} the software GAP is dealing with. More precisely, if this string is $[i_1,\ldots,i_m]$, the  corresponding representative is $\zeta_{i_1}\cdots \zeta_{i_m}$, where $\zeta_1,\ldots,\zeta_5$ stand for the generators of  {\ttfamily  WD5} implicitly defined when the latter group was defined (by using the GAP command {\ttfamily CoxeterGroup("D",5)}. \sk

In view of relating GAP generators $\zeta_i$'s with the $s_i$'s considered in the proof of 
Proposition \ref{Prop:Key}, one makes explicit the labelling implemented in GAP of the vertices of the Dynkin diagram $D_5$ via the command 
{\ttfamily PrintDiagram(CoxeterGroup("D",5))} which returns the  diagram 
$$
  \xymatrix@R=0.1cm@C=0.4cm{ 
  1\ar@{-}[rd]   & &   &
 \\
  & 3  \ar@{-}[r] & 4   \ar@{-}[r]  & 5\, . \\
   2\ar@{-}[ru]   & &   &
  }
$$
Comparing it with the diagram for $D_5$ in Figure \ref{Fig:DynkinDiagram}, one can easily  express GAP generators of {\ttfamily WD5} in terms of those we considered in 
Section \S\ref{S:Main}:\footnote{Several  (two in case of $D_5$) equivalent choices are possible for expressing the $\zeta_i$'s in terms of the $s_i$'s;  all give rise to the same results of course.}  
one has 
$$
\zeta_1=s_4\, , \qquad \zeta_2=s_5
\, , \qquad \zeta_3=s_3
\, , \qquad \zeta_4=s_2
\qquad \mbox{ and } 
 \qquad \zeta_5=s_1\, .
$$
We denote by $\sigma$ the permutation such that $\zeta_{i}=s_{\sigma(i)}$ for $i=1,\ldots,5$. 

At this point, it is straightforward to compute the value of the character $\chi_5$ evaluated on any 
conjugacy class. If $[i_1,\ldots,i_m]$ stands for the string given by GAP which encodes a representative of a conjugacy class 
$\Gamma\subset W(D_5)$, then one has 
$$
\chi_5(\Gamma)={\rm Tr}\Big(  M_{\sigma(i_1)}\cdots  M_{\sigma(i_m)}
\Big) 
$$
where the $M_i$'s $(i=1,\ldots,5$) stand for the $16\times 16$ matrices defined in \eqref{Eq:Mi}. 
   Since these matrices can be computed explicitly, its is straighforward to compute the trace above hence to build the 18-tuple of values corresponding to the character 
$ \chi_5$ (for the same labelling as the one used to label the column of the character table above). Identifying $ \chi_5$ with this 18-tuple, one has 
   $$ \chi_5=\big(16, 0, 0, 8, 0, 0, 0, 4, 0, 0, 4, 0, 0, 2, 0, 2, 0, 1\big).$$
Elementary linear algebra computations give us that 
$$ \chi_5= \boldsymbol{\big[.5\big]}+ 
\boldsymbol{\big[1.4\big]}+
\boldsymbol{\big[2.3\big]}   $$ 
where we use their labels (of the form $\boldsymbol{\big[\lambda.\lambda'\big]}$) for denoting the correspond line in the character table. Since  $\boldsymbol{\big[.5\big]}={\bf 1}$, this gives us the decomposition for $\mathbf C^{{\mathcal L}_5}$ in \eqref{Eq:Decomposition-Wr}.

Once $\chi_5$ is explicitly known, it is not difficult to compute the character of 
$\wedge^3 \mathbf C^{{\mathcal L}_5}$, denoted by $\wedge^3\chi_5$, thanks to the following formula which holds true for any element $g$ of $W(D_5)$: 
\begin{equation}
\wedge^3\hspace{-0.1cm}\chi_5\big(g\big)=
\frac{1}{3!}\, \bigg(\, \chi_5\big(g\big)^3-3 \,\chi_5\Big(g^2\Big) \,\chi_5\big(g\big)+2 \, \chi\Big(g^3\Big)
\, 
\bigg)\, .
\end{equation}

Using the same notations and conventions as above, we then get that 
$$
\wedge^3\hspace{-0.1cm}\chi_5=\big( 560, 0, 0, 24, 0, 0, 0, -20, 0, 0, 8, 0, 0, 0, 0, -2, 0, 0 \Big)\, .$$
Denoting now by $\chi^{s}$ the character of $W(D_5)$ corresponding to the $s$-th line of Table \ref{Table:CharTableWD5}, we have $\wedge^3 \chi_5=\sum_{s=1}^{18} w_{s}\,\,  \chi^s$  with 
$$
 \big(  w_{s}\big)_{s=1}^{18}= \big(1, 1, 0, 4, 5, 4, 1, 1, 6, 0, 5, 6, 3, 3, 1, 2, 2, 0\big)\, . 
$$
%
%

In particular, since $w^3=0$ we get that $\chi^{3}=\boldsymbol{[.1^5]}={\bf sign}$ does not appear in the decomposition of 
$\wedge^3 \mathbf C^{{\mathcal L}_5}$ into $W(D_5)$-irreducibles, which is the content  of \eqref{m-signa} in the case under scrutiny. 
\mk 

All the computations above can be performed quite easily within GAP.  First one constructs  $ \chi_5$ and  $\wedge^3 \chi_5$ as characters of $W(D_5)$ by means of the following two commands: 
\begin{lstlisting}[language=GAP]
gap>chi5:=TableWD5.irreducibles[14]+TableWD5.irreducibles[16]
                                         +TableWD5.irreducibles[18];;
gap>Wedge3chi5:=AntiSymmetricParts(TableWD5,[chi5],3); 
\end{lstlisting}
And one computes the multiplicity of the signature representation 
as an irreducible component of 
 $\wedge^3 \mathbf C^{{\mathcal L}_5}$ by computing the scalar product 
\begin{lstlisting}[language=GAP]
gap>ScalarProduct(TableWD5,Wedge3Chi5[1],TableWD5.irreducibles[3]);
\end{lstlisting}
which turns out to be 0 as instantly returned by GAP.

\subsection*{A few words about the other cases}
All that has been discussed above can be adapted quite straightforwardly for any other  
$r\in \{4,\ldots,8\}$. Once $\chi_r$ has been determined in terms of the character table of the associated Weyl group $W(E_r)$, 
one computes its $(r-2)$-th wedge product either using the general determinental formula 
expressing  $\wedge^{r-2} \chi_r$ in terms of $\chi_r$\footnote{This classical formula can found 
in completely explicit form in 
\cite[Exercise 2.9.13.(a)]{GR}.} or the GAP command {\ttfamily AntiSymmetricParts}.  Then the multiplicity $m_{r}^{\bf sign}$ of ${\bf sign}$ in the decomposition of $\wedge^{r-2} \chi_r$ into irreducibles is given by their scalar product which we compute in GAP using {\ttfamily ScalarProduct}.\mk 

We obtain that ${\bf sign}$ does not appears in the decomposition of $\wedge^{r-2} \chi_r$ for $r=4,\ldots,7$, while its multiplicity $m_{8}^{\bf sign}$ as a component of $\wedge^{6} \chi_8$ can be computed to be 5.


\bigskip\bigskip

\vfill 
{\small  ${}^{}$ \hspace{-0.6cm} {\bf Luc Pirio} (\href{mailto:luc.pirio@uvsq.fr}{\tt  luc.pirio@uvsq.fr})\\
Laboratoire de Math\'ematiques de Versailles\\
 Universit\'e Paris-Saclay, UVSQ \& CNRS  (UMR 8100)

\end{document}